\documentclass[11pt,reqno]{amsart}
\usepackage[T1]{fontenc}
\usepackage{amsmath,amssymb,amsthm,mathtools}
\usepackage[textwidth=16cm,hcentering,top=2cm,bottom=2cm,headheight=10pt,headsep=14pt]{geometry}
\usepackage{needspace}
\usepackage{tikz}
\usepackage[hidelinks]{hyperref}
\hypersetup{
 pdftitle={On rigidity of stationary black holes under a nontrapping assumption},
 pdfauthor={Alexandru F. Radu},
 pdfsubject={Stationary black hole rigidity under nontrapping},
 pdfkeywords={stationary vacuum black holes, rigidity, Killing fields, nontrapping}
}
\allowdisplaybreaks[1]
\numberwithin{equation}{section}
\makeatletter
\renewcommand{\section}{\@startsection{section}{1}%
  \z@{.7\linespacing}{.5\linespacing}%
  {\normalfont\scshape\centering}}
\renewcommand{\subsection}{\@startsection{subsection}{2}%
  \z@{.5\linespacing}{-.5em}%
  {\normalfont\bfseries}}
\makeatother

\newcommand{\g}{\mathbf g}
\newcommand{\T}{\mathbf T}
\newcommand{\Z}{\mathbf Z}
\newcommand{\D}{\mathbf D}
\newcommand{\R}{\mathbf R}
\newcommand{\M}{\mathbf M}
\newcommand{\E}{\mathbf E}
\newcommand{\HH}{\mathcal H}
\newcommand{\LL}{\mathcal L}
\newcommand{\DD}{\mathcal D}
\newcommand{\Ric}{\operatorname{Ric}}
\newcommand{\Hess}{\operatorname{Hess}}
\newcommand{\tr}{\operatorname{tr}}
\newcommand{\Vol}{\operatorname{Vol}}
\newcommand{\RR}{\mathbb R}
\newcommand{\pr}{\partial}
\newtheorem{theorem}{Theorem}[section]
\newtheorem{proposition}[theorem]{Proposition}
\newtheorem{lemma}[theorem]{Lemma}

\theoremstyle{definition}
\newtheorem{definition}[theorem]{Definition}
\theoremstyle{remark}
\newtheorem{remark}[theorem]{Remark}

\begin{document}
\flushbottom
\title[Black hole rigidity under nontrapping]
{On rigidity of stationary black holes under a nontrapping assumption}
\author{Alexandru F. Radu}
\address{Simion Stoilow Institute of Mathematics, Romanian Academy,
Calea Grivitei Street, no.~21, 010702 Bucharest, Romania}
\email{alexandru.radu@imar.ro}
\date{}
\begin{abstract}
We prove that a smooth, regular, asymptotically flat stationary
vacuum black hole has a nonextremal Kerr exterior if no null
geodesic orthogonal to its stationary Killing field is trapped,
establishing the rigidity conjecture of Alexakis, Ionescu, and
Klainerman. In the rotating case, we start with the rotational
Killing field constructed near the bifurcation sphere and study
a maximal domain carrying its complete circle action. The main
difficulty is to recover boundary regularity and the curvature
condition needed for local Killing extension. We derive uniform
estimates along complete rotational orbits from nontrapping and
combine them with convexity along short null geodesics. These
estimates show that a nonempty remaining boundary in the ergoregion
is a timelike photon surface for the stationary quotient metric.
Its null geodesics lift to spacetime null geodesics orthogonal to
the stationary field. The vacuum circularity identities give their
continuation through contact with the ergosurface, keeping their
stationary projections in a compact set. This contradicts nontrapping
and yields the global rotational symmetry required by black hole
uniqueness.
\end{abstract}

\maketitle
\enlargethispage{2pt}
\vspace*{-2\baselineskip}
\begingroup
\footnotesize
\tableofcontents
\endgroup
\section{Introduction}\label{sec:introduction}

In this paper we prove the rigidity conjecture formulated by
Alexakis, Ionescu, and Klainerman~\cite{IK15}:
a smooth, regular, asymptotically flat stationary vacuum black
hole with no trapped null geodesics orthogonal to its stationary
Killing field has a Kerr exterior. The physical motivation is the
expectation that an isolated black hole approaches a stationary
state as gravitational radiation escapes to infinity or enters
the horizon. If this relaxation takes place, uniqueness identifies
its possible end states with the family found by
Kerr~\cite{Kerr63}. The convergence to a stationary state is the
separate stability problem.\footnote{The perturbation equations
of Regge and Wheeler~\cite{RW57} and
Teukolsky~\cite{Teukolsky73}, the decay calculations of
Price~\cite{Price72}, and mode stability of
Whiting~\cite{Whiting89} are early contributions to this problem.
Dafermos, Rodnianski, and Shlapentokh-Rothman~\cite{DRSR16}
proved scalar-wave decay in the full subextremal range.
The small-angular-momentum nonlinear stability theorem of
Klainerman and Szeftel~\cite{KS23} uses the wave estimates of
Giorgi, Klainerman, and Szeftel~\cite{GKS24} and
Shen's~\cite{Shen23} hypersurface construction.
The recent preprint of Hintz~\cite{Hintz26}
treats the full subextremal range for small initial perturbations
with specified finite asymptotic expansions and decaying remainders.}

The classical uniqueness argument first constructs an additional
symmetry and then classifies the resulting stationary-axisymmetric
solution. Hawking~\cite{Hawking72} constructs a Killing field
normal to the horizon, and analyticity permits its continuation
to the exterior, where it gives axial symmetry in the rotating
case. The global extension requires causality and completeness
hypotheses, treated by Chru\'sciel~\cite{Chrusciel97}.
The results of Carter~\cite{Carter71} and
Robinson~\cite{Robinson75} then identify the exterior with Kerr.
The potential formulation of Ernst~\cite{Ernst68} leads to the
comparison identities of Mazur~\cite{Mazur82} and
Bunting~\cite{Bunting83}, discussed by Carter~\cite{Carter85}.
We use the smooth axisymmetric theorem in
Chru\'sciel, Costa, and Heusler~\cite{CCH12}, with
the global analysis of Chru\'sciel and Costa~\cite{CC08}.
The same review describes the static results of
Israel~\cite{Israel67} and Bunting and
Masood-ul-Alam~\cite{BuntingMasood87}.\footnote{In higher dimensions,
Hollands, Ishibashi, and Wald~\cite{HIW07} and Moncrief and
Isenberg~\cite{MI08} prove nondegenerate analytic rigidity even
when the induced stationary flow on a horizon cross-section has
nonclosed orbits. Within the four-dimensional stationary-axisymmetric
class, Han, Khuri, Weinstein, and Xiong~\cite{HKWX26}
exclude regular multi-horizon equilibria, assuming nonzero angular
momentum on every degenerate component. Our horizon is connected.}

Analyticity enters at the extension of the horizon symmetry.
M\"uller zum Hagen~\cite{MH70} shows that stationary vacuum
metrics are analytic where the stationary field is timelike,
but the stationary equations lose ellipticity in the ergoregion.
We must therefore obtain the additional symmetry there from the
smooth vacuum equations. At the heart of this problem is the
ill-posed nature of continuation into the black hole exterior:
characteristic horizon data determine a Killing field in their
domain of dependence, while the exterior lies on the other side
of the characteristic hypersurfaces.

For smooth bifurcate horizons, Friedrich, R\'acz, and
Wald~\cite{FRW99} construct the horizon Killing
field in the domain of dependence, using the characteristic
initial-value theory of Rendall~\cite{Rendall90}.
Alexakis, Ionescu, and Klainerman~\cite{AIK09}
extend it to a full neighborhood of the bifurcation sphere.
In the presence of a stationary symmetry, their rotational
rigidity theorem also gives the commuting periodic field from
which we start.\footnote{For analytic
compact Cauchy horizons with closed generators, see Moncrief and
Isenberg~\cite{MI83}. In the smooth setting, Petersen and
R\'acz~\cite{PR23} construct the Killing field on
the globally hyperbolic side when the surface gravity can be
normalized to a nonzero constant. Petersen~\cite{Petersen21}
extends this field across the compact horizon. His black hole
application assumes a globally hyperbolic exterior, complete
stationary flow, and a closed spacelike section, achronal in the
exterior and compact apart from its asymptotic end, whose boundary
meets each future horizon generator once. With no white-hole region
and a null horizon generator commuting with the stationary field
and having nonzero constant surface gravity, it gives local
rigidity without a bifurcation surface.}

The local extension theorem we use is due to Ionescu and
Klainerman~\cite{IK15}. Its Carleman estimate tests
pseudoconvexity only in null directions orthogonal to the
stationary field, as in Ionescu and
Klainerman~\cite{IK09}.
The curvature wave equation is coupled to transport equations
which fix the deformation of the extended field. This formulation
incorporates the gauge into the uniqueness argument, following
the geometric comparison method of Alexakis~\cite{Alexakis09}
and the invariant Killing-field system of Ionescu and
Klainerman~\cite{IK13}; see also their earlier
characteristic result~\cite{IKCharacteristic09}.

To apply this theorem globally, one needs hypersurfaces with the
required pseudoconvexity and with an entire local side already
carrying the symmetry. The earlier global arguments obtain this
geometry from additional information about the solution.
Following the Kerr characterizations of Simon~\cite{Simon84}
and Mars~\cite{Mars99,Mars00}, Ionescu and
Klainerman~\cite{IK09}
assume an identity between the Ernst potential and the squared
self-dual Killing form on the bifurcation sphere, together with
a one-point inequality. The identity supplies the initial
vanishing of the Mars--Simon tensor, which they propagate through
the exterior.\footnote{Kr\"oncke and Petersen~\cite{KP25} determine the transverse metric expansion from
a Riemannian metric and a constant-length Killing field induced
on the horizon. Under the asymptotic flatness and smooth
bifurcate sphere assumptions of Ionescu and Klainerman~\cite{IK09},
agreement of these data with Kerr replaces the
technical horizon condition in Ionescu--Klainerman and gives local
isometry of the domain of outer communications to subextremal Kerr.}
Alexakis, Ionescu, and Klainerman~\cite{AIK10}
assume instead that a weighted contraction of the Mars--Simon
tensor is small. This allows them to control the Hessian of
$\Re(1-\sigma)^{-1}$, where $\sigma$ is the Ernst potential,
and supplies regular conditionally pseudoconvex level sets.

A different global argument of Alexakis, Ionescu, and
Klainerman~\cite{AIK14}
assumes a maximal hypersurface ending on the bifurcation sphere
and smallness of the stationary field's squared norm there.
Adapting the integral argument of Sudarsky and
Wald~\cite{SW93} on a maximal hypersurface, they estimate its
second fundamental form and show that the entire ergoregion
lies inside the local rigidity neighborhood. Under nontrapping,
neither that containment nor the earlier Ernst foliation is
available. We circumvent this difficulty by choosing a maximal
domain of the rotational action and showing that any remaining
boundary in the ergoregion must carry the null geodesics excluded
by the hypothesis. The precise continuation theorem is stated below.

\subsection{Assumptions and the main theorem}\label{subsec:assumptions}

Let $(\M,\g)$ be a connected, oriented, time-oriented smooth
four-dimensional spacetime of signature $(-,+,+,+)$, with
$\Ric(\g)=0$, and let $\T$ be a smooth stationary Killing field.
We write $\D$ for the Levi--Civita connection of $\g$.
For $A\subset\M$, we denote its chronological future and past
by $I^+(A)$ and $I^-(A)$. We use the exterior-hypersurface
formulation of Alexakis, Ionescu, and
Klainerman~\cite{AIK10,AIK14} to state the regularity assumptions.

\medskip\noindent
GR. (Global regularity.)
Let $\Sigma^0$ be an embedded spacelike partial Cauchy hypersurface,
that is acausal and edgeless, parametrized by a diffeomorphism
$\Phi_0:\{x\in\RR^3:|x|>1/2\}\to\Sigma^0$, and put
$\Sigma_1=\Phi_0(\{|x|>1\})$.
There is an asymptotic end
$\M^{(\mathrm{end})}\simeq\RR\times\{x\in\RR^3:|x|>R_0\}$,
with $\T=\pr_t$ future directed and spatial indices $i,j,k\in\{1,2,3\}$.
Let $M>0$ be the ADM mass of the end, and let $\delta_{ij}$
denote the Euclidean metric. In these coordinates,\footnote{For
an integer $m\ge0$, $u=O_m(|x|^{-a})$ means that
$|\partial^Iu(x)|\le C_I|x|^{-a-|I|}$ for every spatial
multi-index $I$ with $|I|\le m$, for sufficiently large $|x|$.
For tensor fields we impose these bounds on their components.}
\begin{equation}\label{intro:af}
\begin{aligned}
\g_{00}&=-1+\frac{2M}{|x|}+O_6(|x|^{-2}),\\
\g_{ij}&=\delta_{ij}+O_6(|x|^{-1}),\\
\g_{0i}&=O_6(|x|^{-2}).
\end{aligned}
\end{equation}
The domain of outer communications
$\E=I^-(\M^{(\mathrm{end})})\cap I^+(\M^{(\mathrm{end})})$
is globally hyperbolic. Its stationary orbits are complete and meet
$\Sigma_1$. We assume
\[
 \Sigma^0\cap I^\pm(\M^{(\mathrm{end})})=\Sigma_1,
\]
and that $\Sigma^0$ is $t=0$ on the end.

\medskip\noindent
SBS. (Smooth bifurcation sphere.)
The future event horizon $\partial I^-(\M^{(\mathrm{end})})$
and the past event horizon $\partial I^+(\M^{(\mathrm{end})})$
meet in the embedded sphere $S_0=\Phi_0(\{|x|=1\})$.
Near $S_0$ we denote these horizons by $\HH^+$ and $\HH^-$,
respectively, and assume that they are smooth transverse
nonexpanding null hypersurfaces tangent to $\T$.
For a null generator $L$, the null
second fundamental form on a spacelike cross-section is
$(X,Y)\mapsto\g(\D_XL,Y)$ for tangent $X,Y$.
Nonexpanding means that its trace in the induced metric vanishes.

The bifurcate setting is related to the usual nondegeneracy condition
by the extension theorems of R\'acz and Wald~\cite{RW92,RW96} for
Killing horizons with compact cross-sections and nonzero constant
surface gravity. We impose SBS directly on the stationary event
horizons and use it to construct the additional Killing field.

Let $Q=\E/\RR$ be the space obtained by identifying two points
of $\E$ whenever they lie on the same orbit of $\T$, and let
$\pi:\E\to Q$ send a point to its orbit. We give $Q$ the
quotient topology: a set is open when its inverse image under
$\pi$ is open in $\E$. We can state nontrapping using this topology
alone. In Proposition~\ref{prep:quotient} we show that $Q$ is a
smooth three-dimensional manifold and identify it with a smooth
hypersurface meeting each stationary orbit exactly once.

\medskip\noindent
NT. (Nontrapping.)
No maximal unparametrized null geodesic $\gamma$ in $\E$ with
$\g(\T,\dot\gamma)=0$ has $\pi(\gamma)$ contained in a compact
subset of $Q$.

\medskip
Here maximality refers to the unparametrized geodesic in $\E$,
including the case of a finite affine parameter interval. The
compact sets in NT lie in the stationary quotient of the open
exterior, away from its horizon and asymptotic ends.

\begin{theorem}\label{thm:main}
Under GR, SBS, and NT, the domain
$(\E,\g)$ is isometric to the exterior of a nonextremal Kerr spacetime.
\end{theorem}

We verify NT in every subextremal Kerr exterior in
Section~\ref{subsec:nontrapping}. It concerns only zero stationary
energy and therefore allows the usual photon trapping in Kerr.
For a stationary wave equation these are precisely the
characteristic directions relevant to unique continuation, as
Ionescu and Klainerman~\cite{IK15} explain.
They can occur only where $\T$ is spacelike or null.

We state separately the principal continuation result. Suppose
$\T|_{S_0}\not\equiv0$, and let $\Z$ be the rotational Killing
field of Alexakis, Ionescu, and Klainerman~\cite{AIK09} on a neighborhood $\mathcal O$ of
$S_0$, normalized to effective period $2\pi$.

\begin{theorem}\label{discrete:continuation}\label{prop:continuation}
Assume GR, SBS, and NT, with $\T|_{S_0}\not\equiv0$.
There exist a neighborhood $\mathcal O'\subset\mathcal O$ of
$S_0$ and a connected stationary open set $U\subset\E$ such that
\[
 (\mathcal O'\cap\E)\cup\{p\in\E:\g(\T,\T)(p)>0\}\subset U,
\]
and $\Z|_{\mathcal O'\cap\E}$ extends to a Killing field on $U$
which commutes with $\T$ and has complete flow of effective
period $2\pi$.
\end{theorem}

Once this theorem is proved, we extend the field locally across
the ergosurface and obtain a timelike Killing field near every
point of the exterior. Stationary analyticity and the global
extension theorem of Chru\'sciel~\cite{Chrusciel97}
then give the global circle action needed for axisymmetric
uniqueness. The case $\T|_{S_0}\equiv0$ follows from staticity.
We give both conclusions in Section~\ref{exit:section}.

\subsection{Main ideas of the proof}

We first explain the geometric alternative behind
Theorem~\ref{discrete:continuation}. The local theorem of
Ionescu and Klainerman~\cite{IK15}, stated in
Proposition~\ref{local:extension}, extends a commuting Killing
field from $f<0$ across a stationary
hypersurface through $p$ if its defining function satisfies
\[
 \D^2f(X,X)<0
 \quad\text{for }X\ne0\text{ with }
 Xf=\g(X,X)=\g(\T,X)=0.
\]
Here $\D$ is the spacetime Levi--Civita connection. For the
null geodesic $\gamma$ with initial velocity $X$ at $p$, tangency
and $f(p)=0$ give
\[
 f(\gamma(s))=\frac{s^2}{2}\D^2f(X,X)+O(s^3).
\]
Both branches enter the region carrying the field. We would like
to find such a hypersurface at a boundary point of the symmetry
domain, but a maximal domain need not have a smooth boundary.
Moreover, the extension theorem requires its whole local
negative side to be in that domain. These are the two geometric
issues addressed by our argument.

The vacuum circularity identities of
Papapetrou~\cite{Papapetrou66} and
Carter~\cite{Carter69} give a useful interpretation
of the curvature condition. In the ergoregion the stationary
quotient carries the Lorentzian metric $h$ defined by
\[
 \pi^\ast h=\g(\T,\T)\g-\T^\flat\otimes\T^\flat,
\]
where $\T^\flat=\g(\T,\cdot)$; see
\eqref{local:stationary-metrics}. Its null geodesics lift, after
reparametrization, to spacetime null geodesics orthogonal to $\T$, as shown in
\eqref{local:null-lift}. On a region carrying the complete
rotational action, circularity makes the projected field
hypersurface orthogonal, while Lemma~\ref{action:area} gives
\[
 \g(\T,\Z)^2-\g(\T,\T)\g(\Z,\Z)>0.
\]
Thus its projection is timelike. Writing $\rho^2=-h(\Z,\Z)$,
we can choose local coordinates
along $\Z$ in which
\[
 h=\rho^2(-d\phi^2+q_{\mathrm{opt}}),
\]
where $q_{\mathrm{opt}}$ is the Riemannian metric on a transverse
disk defined in~\eqref{local:optical-product}. We use the same
symbol for a stationary field and its
projection to $Q$.

The two null directions tangent to a smooth invariant surface
project to the two orientations of the same transverse curve.
The second fundamental form evaluated on these directions is
therefore determined by the geodesic curvature of that curve.
A positive signed geodesic curvature, with the normal pointing
out of the symmetry domain, gives a hypersurface across
which the field extends. If that curvature vanishes identically, the
transverse curve is a geodesic and the surface is a timelike
photon surface of $(Q,h)$: every null geodesic initially tangent
to it stays in it locally, in the terminology of
Claudel, Virbhadra, and Ellis~\cite{CVE01}.
The calculation in Lemma~\ref{local:optical-curvature} is the
reason these two alternatives describe all the null tangent
directions, rather than only one of them.

To implement this idea before the boundary is smooth, we represent
the part of the ergoregion still to be crossed by a closed subset
$H$ of one fixed compact set in $Q$. We require that
\[
 \pi(\gamma(a)),\pi(\gamma(b))\in H
 \quad\Longrightarrow\quad
 \pi(\gamma([a,b]))\subset H
\]
for all sufficiently short zero-energy null segments, with length
measured in a fixed auxiliary Riemannian metric on spacetime,
as specified in Definition~\ref{discrete:null-convexity}.
The strict Hessian inequality preserves this condition when a
local extension removes part of $H$, and intersections preserve
it. Compatible increasing unions of domains also preserve the
complete rotational action, since each point remains on an
orbit in one of the earlier domains. We can therefore choose a
maximal compatible extension. A strict supporting hypersurface
would permit a further invariant removal and contradict maximality,
provided that the local extensions agree on every old overlap and
extend around the complete orbit.

To compare the field at nearby boundary points, we first bound
its normalized direction and derivative along complete rotational
orbits. Let $g_+$ be a fixed auxiliary
Riemannian metric on $Q$, and let $\nabla$ be the connection of
$h$. Lemma~\ref{action:intrinsic-orbits} gives
\[
 \left|\frac{\Z}{\rho}\right|_{g_+}
 +\left|\nabla\left(\frac{\Z}{\rho}\right)\right|_{g_+}
 +|d\log\rho|_{g_+}\le C_K
\]
on orbits contained in a fixed compact subset $K$ of the
ergoregion. The constant depends on the ambient geometry and the
length bound for confined null geodesics supplied by NT, and is independent
of the Killing field and its domain. Hypersurface orthogonality
gives a closed equation for the orbit velocity and acceleration.
If the estimate failed, normalization at a point of maximum speed
on an orbit would yield a nonzero null trajectory of unbounded
auxiliary length in $K$, contrary to that length bound.

Together with closure under short null segments, the orbit
estimate supplies exterior normal segments joining nearby points
to one fixed compact subset of the domain. We use the Killing
transport equations along these segments to bound the field and
compare its boundary values. On a transverse disk, a short
geodesic of $q_{\mathrm{opt}}$ crossing into the exterior would lift there to
a zero-energy null segment with endpoints in $H$, contradicting
the defining condition. Thus the transverse remaining set is
geodesically convex. Comparison with a short geodesic chord and
one strictly curved touching curve shows that its boundary is
itself a geodesic unless further extension is possible. We verify
agreement on all exterior components and preservation of the
complete action before applying maximality. This proves the
photon-surface conclusion of
Proposition~\ref{discrete:stalled-boundary}.

A null geodesic on this surface may reach the ergosurface, where
$h$ degenerates. We then use the stationary norm and twist
potential as transverse spacetime coordinates. Their common
levels carry flat Lorentzian metrics. Using parallel fields for
these induced metrics, we obtain a geodesic equation with smooth
coefficients through the ergosurface, even though the second
parallel field need not be a spacetime Killing field. We retain
the equation for its momentum. Circularity eliminates the terms
linear in the transverse velocity from the acceleration equation,
and at contact the transverse velocity is zero. The resulting
system~\eqref{local:mechanical-reflection} is reversible, so
uniqueness forces the transverse curve
to retrace its incoming arc. After checking the full spacetime
equation for the lift, we obtain arbitrarily long null segments
whose stationary projections remain in one compact set. This
contradicts NT and proves Theorem~\ref{discrete:continuation}.

Section~\ref{prep:section} prepares the horizon field, the local
extension theorem, and the initial strict hypersurface.
Section~\ref{sec:action} derives the circularity identities,
identifies pseudoconvexity with transverse geodesic curvature,
and proves the estimate along complete orbits.
Section~\ref{sec:discrete-continuation} constructs the maximal
extension and establishes its boundary geometry.
Section~\ref{sec:remaining} continues the null geodesic through
ergosurface contact and obtains the nontrapping contradiction.
In Section~\ref{exit:section} we complete the global symmetry
and apply axisymmetric uniqueness. Appendix~\ref{app:completion}
constructs the future horizon section needed by the global
theorems and verifies preservation of the domain of outer
communications.
\section{Stationary geometry and local rigidity}
\label{prep:section}\label{sec:local}\label{canon:section}

We assume GR, SBS, and NT, with $\T$ future directed at infinity.
For the constructions involving the rotational field we also
assume $\T|_{S_0}\not\equiv0$. The horizon completion in
Section~\ref{subsec:future-completion} uses only GR and SBS
and includes the nonrotating case. We first place the stationary
orbits in a smooth quotient and extend the local rotation of
Alexakis, Ionescu, and Klainerman~\cite{AIK09} along those orbits.
The differential of the Hawking field on the normal plane then provides the
strictly null-convex hypersurface from which we start the global
extension. Proposition~\ref{prep:residual} records the initial
domain and the compact set left to cover.

We use the local extension theorem of Ionescu and
Klainerman~\cite[Theorem~2.4]{IK15} in the form stated in
Proposition~\ref{local:extension}. To apply it we must exhibit
an entire local side carrying the Killing field; the transport
and compatibility arguments below will identify the resulting
extensions on all their overlaps. Near the ergosurface,
nontrapping gives independent stationary norm and twist gradients.
These provide the regular coordinates for the geodesic argument
in Section~\ref{sec:remaining}.

We write $\D$ for the spacetime Levi--Civita connection and $\R$
for its covariant curvature tensor. Greek indices range from $0$
to $3$ and are raised with $\g$; Latin index ranges are specified
locally. Repeated upper and lower indices are contracted.
For a scalar $f$, $\D^2f$ is its covariant Hessian and
$|df|_\g^2=\g^{-1}(df,df)$. We use $d,\LL_X,\iota_X$ for
exterior differentiation, Lie differentiation and contraction.
The symbols $\flat,\sharp$ denote metric duality, and $\ast$
on forms denotes the Hodge star of $\g$, unless another metric
is indicated. A superscript $\ast$ on a map denotes pullback, and $d\Vol_\g$
is the oriented spacetime volume form.
The notation $D$ denotes coordinate differentiation.
Positive constants $c,C$ may change from line to line, with
their dependence stated in each estimate.

We use the curvature convention of Ionescu and
Klainerman~\cite{IK13}. For vector fields $X,Y,V$,
\[
 \R(X,Y)V=\D_X\D_YV-\D_Y\D_XV-\D_{[X,Y]}V,
\]
and the corresponding covariant tensor is
\[
 \R_{\alpha\beta\gamma\delta}
 =\g\bigl(\R(\partial_\alpha,\partial_\beta)\partial_\delta,
                                      \partial_\gamma\bigr).
\]
In particular, a covector $\xi$ satisfies
$[\D_\alpha,\D_\beta]\xi_\gamma
=\R_{\alpha\beta\gamma\delta}\xi^\delta$.
The Killing identity in this convention is
\eqref{local:killing-connection}.

\subsection{The stationary section and topology}

We begin by placing the stationary orbits in a smooth quotient.
Chru\'sciel and Wald~\cite{CW94} studied the topology of stationary
exteriors using the topological-censorship theorem of Friedman,
Schleich, and Witt~\cite{FSW93}. In our setting, the topology
prescribed for $\Sigma_1$ in GR will give simple connectedness
directly, once we prove properness of its projection.

We will also need a spacelike section transverse to $\T$
everywhere, including the ergoregion. To construct it, we average
a temporal function along the stationary flow. Causal comparison
with a timelike orbit in the asymptotic end makes this average
finite and, at the same time, proves properness of the action.

For a vector field $X$, $\Phi_s^X$ denotes its local flow on the indicated manifold.
Completeness on an invariant open set means that this flow is
defined there for every $s\in\RR$ and preserves the set.
We write $p\ll q$ when $q\in I^+(p)$.

\begin{proposition}\label{prep:quotient}
Under GR, the $\T$-action on $\E$ is free and proper, and
$Q=\E/\RR$ is a smooth simply connected three-manifold.
There is $\vartheta\in C^\infty(\E)$ such that
\[
 \begin{aligned}
 \T\vartheta&=1,\\
 \g^{-1}(d\vartheta,d\vartheta)&<0,
 \end{aligned}
\]
and
\[
 (\vartheta,\pi):\E\longrightarrow\RR\times Q
       \text{ is a diffeomorphism}.
\]
The map $\pi|_{\Sigma_1}:\Sigma_1\to Q$ is proper and
surjective, and $\E$ is simply connected.
\end{proposition}
\begin{proof}
Fix a future timelike stationary orbit
$\gamma(s)=\Phi_s^\T p_\infty$ in the end. We claim that the
stationary time between two points of a compact set is bounded.
The same comparison will subsequently control the projection of
$\Sigma_1$.
The definition $\E=I^-(\M^{(\mathrm{end})})\cap
I^+(\M^{(\mathrm{end})})$, followed by timelike travel within
the end, gives
\begin{equation}\label{prep:comparison}
 \text{for each compact }C\subset\E,\quad
 \gamma(a)\ll p\ll\gamma(b)\quad(p\in C)
\end{equation}
for some $a<b$. If $p,\Phi_t^\T p\in C$ and $t>b-a$,
\[
 \gamma(a+t)\ll\Phi_t^\T p\ll\gamma(b)\ll\gamma(a+t),
\]
contradicting chronology. The reversed chain excludes
$t<a-b$. Hence the inverse image of a compact set under
$(t,p)\mapsto(p,\Phi_t^\T p)$ is closed in a compact set
$[-(b-a),b-a]\times C$. This proves properness. A nonzero
period is also excluded,
since its integer multiples would violate the same bound.
The free action therefore has a smooth Hausdorff quotient $Q$.

To obtain a spacelike section of this action, we construct a
temporal function $\vartheta$ satisfying $\T\vartheta=1$.
We start with a future increasing Cauchy temporal function $\tau$,
as supplied by Bernal and S\'anchez~\cite[Theorem~1.1]{BS05}, and
average its transition between two fixed levels. Here temporal
means that $d\tau$ is timelike. Choose
$\chi'\ge0$, $\chi=0$ on $(-\infty,-1]$, and $\chi=1$
on $[1,\infty)$, with $\chi\in C^\infty(\RR)$. Define
\[
 \vartheta(p)=\int_\RR
       [\chi(\tau(\Phi_s^\T p))-\chi(s)]\,ds.
\]
The orbit $\gamma$ is inextendible by properness, so
$\tau(\gamma(s))\to\pm\infty$. Applying \eqref{prep:comparison}
after stationary translation gives
\[
 \tau(\gamma(a+s))<\tau(\Phi_s^\T p)<\tau(\gamma(b+s)).
\]
Thus the integrand and its derivatives in $p$ have locally
uniform compact support in $s$, and
\[
 \vartheta(\Phi_t^\T p)-\vartheta(p)
 =\int_\RR[\chi(s)-\chi(s-t)]\,ds=t,
\]
and differentiation under the integral gives
\[
 d\vartheta_p
 =\int_\RR\chi'(\tau(\Phi_s^\T p))
                     [(\Phi_s^\T)^{\ast}d\tau]_p\,ds .
\]
The nonzero integrands in the formula for $d\vartheta_p$ lie in one timelike
cone and are not all zero, since $\T\vartheta=1$.
Writing them as $(a(s),v(s))$ in an orthonormal coframe,
\[
 \int a(s)\,ds>\int|v(s)|\,ds
                     \ge\left|\int v(s)\,ds\right|.
\]
Therefore $d\vartheta$ is timelike. The product map
$\RR\times\{\vartheta=0\}\to\E$ is
\[
 (t,p)\longmapsto\Phi_t^\T p,
\]
and its inverse is
\[
 q\longmapsto(\vartheta(q),\Phi_{-\vartheta(q)}^\T q).
\]

To prove properness of $\pi|_{\Sigma_1}$, we first observe that the edgeless
acausal hypersurface $\Sigma^0$ is closed, since
Minguzzi~\cite[Proposition~2.132 and Definition~3.35]{Minguzzi19}
gives $\overline S\setminus S\subset\operatorname{edge}(S)$ for
achronal $S$. Hence $\Sigma_1$ is closed in $\E$. To control its stationary
parameter, we use an observation that also applies to the deformed
section in Appendix~\ref{app:completion}. Let $\mathcal S\subset\E$
be any nonempty acausal set, fix $u\in\mathcal S$, and let
$C\subset\E$ be compact. Choose $a<b$ from
\eqref{prep:comparison} for $C\cup\{u\}$. If
$p\in C$, $\Phi_t^\T p\in\mathcal S$, and $t>b-a$, then
\[
 u\ll\gamma(b)\ll\gamma(a+t)\ll\Phi_t^\T p,
\]
contrary to acausality. Reversing the chain excludes $t<a-b$.
Thus
\begin{equation}\label{prep:acausal-time-bound}
 |t|\le b-a
 \qquad(p\in C,\ \Phi_t^\T p\in\mathcal S).
\end{equation}
For compact $A\subset Q$, the section $\vartheta=0$ supplies
compact representatives $C\subset\E$. Apply
\eqref{prep:acausal-time-bound} with $\mathcal S=\Sigma_1$:
\[
 (\pi|_{\Sigma_1})^{-1}(A)
 \subset\{\Phi_t^\T p:p\in C,\ |t|\le b-a\}.
\]
Closedness proves properness, and GR gives surjectivity.

We can now use the topology of $\Sigma_1$. Its projection has
degree one, and lifting that map to the universal cover of $Q$
will show that this cover has one sheet. Orient $Q$ by $\T$ and
the spacetime orientation, and $\Sigma_1$ positively at infinity. An asymptotic orbit has
one transverse intersection with $\Sigma_1$, by acausality,
so $\deg(\pi|_{\Sigma_1})=1$. Since $\Sigma_1$ is simply
connected, the map lifts to the connected universal cover
$\varpi:\widetilde Q\to Q$. The lift
$F:\Sigma_1\to\widetilde Q$, with
$\varpi\circ F=\pi|_{\Sigma_1}$, is proper because
\[
 F^{-1}(B)\subset(\pi|_{\Sigma_1})^{-1}(\varpi(B))
\]
is closed and the right side is compact for compact $B$.
Let $q_\infty$ be an asymptotic regular value and let
$q_0\in\varpi^{-1}(q_\infty)$ be the lift hit by $F$.
The unique asymptotic intersection has positive Jacobian, so
\[
 \deg F=\sum_{x\in F^{-1}(q_0)}\operatorname{sgn}\det dF_x=1.
\]
For any other $q_1\in\varpi^{-1}(q_\infty)$,
\[
 F^{-1}(q_1)=\varnothing
 \quad\Longrightarrow\quad
 \deg F=\sum_{x\in F^{-1}(q_1)}\operatorname{sgn}\det dF_x=0,
\]
contradicting constancy of degree on the connected manifold
$\widetilde Q$. The covering has one sheet, so $\pi_1(Q)=0$ and
\[
\pi_1(\E)=\pi_1(\RR\times Q)=0.\qedhere
\]
\end{proof}

We use the same symbol for a stationary scalar on $\E$ and its
function on $Q$. A field commuting with $\T$ also projects to
$Q$; when no ambiguity is possible we retain its symbol, so a
contraction with $\g$ uses the spacetime field and a contraction
with a quotient metric uses its projection. The map
$\pi:\E\to Q$ always denotes the stationary projection.

Auxiliary Riemannian metrics are denoted by $g_+$ on $Q$ and
$\mathbf g_+$ on $\E$, with $\mathbf g_+$ stationary.
The norm $|\cdot|_+$ refers to the induced positive tensor norm
on the space in question. When both spaces occur in an estimate,
we specify $|\cdot|_{g_+}$ and $|\cdot|_{\mathbf g_+}$.
Other metric dependence
is indicated on $\Hess,\Delta$, and norms. The notation
$A\sim B$ means $C^{-1}B\le A\le CB$, with the dependence
of $C$ specified in each estimate.

\subsection{The rotational field near the bifurcation sphere}
\label{subsec:horizon-null}
We recall the double-null construction near
$S_0=\HH^+\cap\HH^-$, where $\HH^+$ and $\HH^-$ are the
future and past horizon hypersurfaces in SBS. This will fix our
normalizations and exhibit the horizon identities used in local
rigidity. We follow Alexakis, Ionescu, and
Klainerman~\cite{AIK09}, using the null geometry of
Christodoulou and Klainerman~\cite{CK93}. Throughout the
construction we restrict to a sufficiently small neighborhood
$\mathcal O$ of $S_0$.

Choose future null normals $L,\underline L$ to $S_0$ with
$\g(L,\underline L)=-1$, tangent to $\HH^+$ and $\HH^-$,
respectively, and extend them affinely along these hypersurfaces.
Their affine parameters satisfy
\begin{align*}
 L\underline u&=1\quad\hbox{on }\HH^+,\\
 \underline L u&=1\quad\hbox{on }\HH^-,
\end{align*}
with $u=0$ on $\HH^+$ and $\underline u=0$ on $\HH^-$.
Issuing the transverse null geodesics from these sections gives
two optical functions near $S_0$:
\begin{align*}
 \g^{-1}(du,du)&=0,\\
 \g^{-1}(d\underline u,d\underline u)&=0.
\end{align*}
Their level hypersurfaces form a double-null foliation, and the
nearby exterior is $u<0<\underline u$. Its gradient generators
are $L=-(du)^\sharp$ and $\underline L=-(d\underline u)^\sharp$.
For example,
\[
 L^\alpha\D_\alpha L_\beta
 =\D^\alpha u\,\D_\alpha\D_\beta u
 =\tfrac12\D_\beta\bigl(\g^{-1}(du,du)\bigr)=0,
\]
so the gradient normalization is affine. The same identity holds
for $\underline L$, and $\g(L,\underline L)=-1$ on the two
horizons.

Let $E_A$, $A=1,2$, be a local tangent frame on a spacelike
intersection of the optical levels, with induced metric $\gamma$.
Its null second fundamental forms are
\begin{align}
 \chi_{AB}&=\g(\D_{E_A}L,E_B)
              =-\D^2u(E_A,E_B),\label{prep:null-second-forms}\\
 \underline\chi_{AB}&=\g(\D_{E_A}\underline L,E_B)
              =-\D^2\underline u(E_A,E_B).\notag
\end{align}
The traces and norms use $\gamma$. The affine Raychaudhuri
equations, together with vacuum and SBS, give
\begin{align*}
 0=L(\tr\chi)&=-|\chi|_\gamma^2-\Ric_\g(L,L)
                    \quad\hbox{on }\HH^+,\\
 0=\underline L(\tr\underline\chi)
   &=-|\underline\chi|_\gamma^2
     -\Ric_\g(\underline L,\underline L)
                    \quad\hbox{on }\HH^-.
\end{align*}
Thus $\chi=0$ on $\HH^+$ and $\underline\chi=0$ on $\HH^-$.
These are the horizon identities entering the local rigidity
theorem. We will distinguish them from the second fundamental
form of a timelike extension hypersurface in
Section~\ref{subsec:optical}.

Alexakis, Ionescu, and Klainerman~\cite[Theorem~1.2 and
Propositions~4.1--4.2]{AIK09} now give, after shrinking
$\mathcal O$, a rotational Killing field $\Z$ satisfying
$[\T,\Z]=0$. We normalize its effective period to $2\pi$:
$\Phi_{2\pi}^\Z=\operatorname{Id}$, with $2\pi$ the least
positive parameter acting identically on the whole domain.
This normalization allows shorter individual orbits, which we
will address in Lemma~\ref{local:free-action}.

To use $\Z$ in the exterior, we extend its circle action along
the stationary orbits. The point requiring proof is that the
translated local fields agree wherever their domains overlap.
We arrange this by choosing an acausal hypersurface
$\Sigma_{\rm loc}$ by the normal exponential map and showing, with the timelike Hawking field,
that each stationary orbit meets it at most once. The next lemma gives the resulting product parametrization of a
stationary neighborhood.
Figure~\ref{fig:local-horizon} shows the position of
$\Sigma_{\rm loc}$ relative to $\Sigma_1$.

\begin{lemma}\label{prep:collar}
Let $\exp$ be the exponential map of $\g$.
There exist $\delta>0$ and a $\Z$-invariant outward unit
spacelike normal $e_1$ to $S_0$ such that
\[
 \Sigma_{\rm loc}=\{\exp_p(re_1(p)):p\in S_0,\ 0<r<\delta\}
\]
is an acausal spacelike hypersurface and
\[
 \RR\times\Sigma_{\rm loc}\longrightarrow\mathcal U_{\mathrm{loc}}
       \subset\E,\qquad (t,p)\longmapsto\Phi_t^\T p
\]
is a diffeomorphism onto an open set. The field $\Z$ extends
from $\mathcal O\cap\mathcal U_{\mathrm{loc}}$ to a Killing field on
$\mathcal U_{\mathrm{loc}}$ with $[\T,\Z]=0$ and complete flow
of effective period $2\pi$.
The normal exponential parametrization extends smoothly to
$\{0\}\times S_0$, and, for every compact $K\subset Q$,
\[
 \pi(\exp_p(re_1(p)))\notin K
       \quad(p\in S_0,\ 0<r<\delta_K)
\]
for some $\delta_K>0$.
\end{lemma}
\begin{proof}
Let $\mathbf K_0$ denote the Hawking field constructed by
Alexakis, Ionescu, and Klainerman~\cite[Theorem~1.1]{AIK09}.
Their horizon normalization is
\[
 \mathbf K_0=\underline uL-u\underline L
       \quad\hbox{on }(\HH^+\cup\HH^-)\cap\mathcal O.
\]
In particular, $\mathbf K_0|_{S_0}=0$.
Their Proposition~3.5 proves that
$\g(\mathbf K_0,\mathbf K_0)<0$ in the nearby exterior.
Propositions~4.1--4.2 give $[\T,\mathbf K_0]=0$ and a constant
$\lambda_0$ for which $\T+\lambda_0\mathbf K_0$ generates a
periodic action. If $t_0>0$ is its effective period, our choice
of sign and normalization is
\[
 \Z=\frac{t_0}{2\pi}(\T+\lambda_0\mathbf K_0).
\]
Here $\lambda_0\ne0$, since otherwise $\T$ would have periodic
orbits in $\E$, contrary to Proposition~\ref{prep:quotient}.
We rescale the Hawking field by setting
\[
 \mathbf K=-\lambda_0\mathbf K_0
\]
and define the horizon angular velocity relative to $\Z$ by
\[
 \Omega_H=-\frac{2\pi}{t_0}.
\]
Thus our $\mathbf K$ is a constant multiple of their Hawking
field, normalized so that the coefficient of $\T$ is one:
\[
 \mathbf K=\T+\Omega_H\Z.
\]
It vanishes on $S_0$, commutes with $\T$ and $\Z$, and satisfies
\[
 \g(\mathbf K,\mathbf K)
 =\lambda_0^2\g(\mathbf K_0,\mathbf K_0)<0
       \quad\hbox{in }\mathcal O\cap\E.
\]
Shrinking around the compact
fixed set $S_0$ gives both flows for $|s|\le 2\pi/|\Omega_H|$, and
\[
 \Phi_{2\pi/|\Omega_H|}^\T=
 \Phi_{2\pi/|\Omega_H|}^{\mathbf K}\Phi_{-2\pi\operatorname{sgn}\Omega_H}^\Z
 =\Phi_{2\pi/|\Omega_H|}^{\mathbf K}.
\]
Fix an exterior point $p$ in this neighborhood and choose
$a<b$ with $\gamma(a)\ll p\ll\gamma(b)$ as in
\eqref{prep:comparison}, where $\gamma$ is the timelike
stationary orbit in the asymptotic end. If $\mathbf K$ were
past directed, repeated stationary translates of this timelike
segment would give, for $2\pi n/|\Omega_H|>b-a$,
\[
 \begin{gathered}
 \gamma(a+2\pi n/|\Omega_H|)\ll\Phi_{2\pi n/|\Omega_H|}^\T p\ll p,\\
 p\ll\gamma(b)\ll\gamma(a+2\pi n/|\Omega_H|),
 \end{gathered}
\]
contradicting \eqref{prep:comparison}. Thus $\mathbf K$
is future directed.

We choose the normal coordinates equivariantly: average a
future timelike normal to $S_0$ under the circle action, normalize
it to $e_0$, and let $e_1$ be the orthogonal unit normal pointing
into the exterior. The differential of $\mathbf K$ acts as a
boost on their plane, with positive rate $\kappa$. To see that
the rate is constant along $S_0$, we use the Killing identity:
\[
 \begin{aligned}
 X^\alpha\D_\alpha\D_\beta\mathbf K_\gamma
 &=X^\alpha\R_{\gamma\beta\alpha\delta}\mathbf K^\delta=0
                    \qquad(X\in TS_0),\\
 \kappa^2&=-\tfrac12(\D_\alpha\mathbf K_\beta)
                              (\D^\alpha\mathbf K^\beta),\\
 X(\kappa^2)&=-(\D^\beta\mathbf K^\gamma)
                       X^\alpha\D_\alpha\D_\beta\mathbf K_\gamma=0.
 \end{aligned}
\]
In fact $\D_{e_0}\mathbf K=\kappa e_1$ and
$\D_{e_1}\mathbf K=\kappa e_0$ at $S_0$, whereas
$\D_X\mathbf K=0$ for $X\in TS_0$. Exponentiating this
endomorphism on the normal plane gives
\[
 \begin{aligned}
 d\Phi_t^{\mathbf K}e_0&=\cosh(\kappa t)e_0+\sinh(\kappa t)e_1,\\
 d\Phi_t^{\mathbf K}e_1&=\sinh(\kappa t)e_0+\cosh(\kappa t)e_1,\\
 d\Phi_t^{\mathbf K}
 \left(\frac{V-U}{2}e_0+\frac{V+U}{2}e_1\right)
 &=\frac{e^{\kappa t}V-e^{-\kappa t}U}{2}e_0
   +\frac{e^{\kappa t}V+e^{-\kappa t}U}{2}e_1.
 \end{aligned}
\]
We apply the exponential map equivariantly to obtain the exact
coordinate representation
\begin{equation}\label{prep:normal-coordinates}
 \begin{gathered}
 (U,V,p)\longmapsto
 \exp_p\left(\frac{V-U}{2}e_0+\frac{V+U}{2}e_1\right),\\
 \mathbf K=\kappa(V\partial_V-U\partial_U),\\
 \Z U=\Z V=0,\\ \T=\mathbf K-\Omega_H\Z.
 \end{gathered}
\end{equation}
The normal-exponential coordinates $U,V$ identify the exterior
with $U,V>0$ and the future horizon with $U=0$.
The function $\tau_1=(V-U)/2$ is temporal near $S_0$,
since $d\tau_1(e_0)=1$ and $|d\tau_1|_\g^2=-1$ there.

\begin{figure}[htbp]
\centering
\begin{tikzpicture}[x=1cm,y=1cm,font=\normalsize,
                    line cap=round,line join=round]
 \begin{scope}
  \fill[black!3] (0,0)--(2.35,2.35)--(4.7,2.35)--(4.7,-2.35)--(2.35,-2.35)--cycle;
  \filldraw[fill=black!14,draw=black!60,dashed]
        (0,0) circle[radius=.43];
  \draw[line width=.85pt] (0,0)--(2.35,2.35);
  \draw[line width=.85pt] (0,0)--(2.35,-2.35);
  \draw[dashed,line width=.7pt] (4.7,2.35)--(4.7,-2.35);
  \node[above left] at (1.23,1.23) {$\HH^+$};
  \node[below left] at (1.23,-1.23) {$\HH^-$};
  \node[above] at (4.7,2.35) {large radius};
  \draw[->,line width=1.15pt] (0,0)--(5.1,0);
  \node[above=4pt] at (2.1,0) {$\Sigma_1$};
  \node at (2.2,1.03) {$\E$};
  \node[align=center] at (2.75,-.67) {asymptotic\\end};
  \filldraw[fill=white,line width=.65pt] (0,0) circle[radius=2pt];
  \node[anchor=east] at (-.36,-.3) {$S_0$};
  \node[above left] at (-.19,.32) {$\mathcal O$};
  \draw[->] (-.8,1.1)--(-.8,1.9);
  \node[above] at (-.8,1.9) {future};
  \node at (2.15,-2.9) {(a) The exterior and its section};
 \end{scope}
 \begin{scope}[xshift=10.3cm]
  % Horizontal and vertical coordinates are (U+V)/2 and (V-U)/2.
  \begin{scope}
   \clip (-2.1,-1.9) rectangle (2.5,1.9);
   \fill[black!5] (0,0)--(-2.1,2.1)--(2.1,2.1)--cycle;
   \fill[black!5] (0,0)--(-2.1,-2.1)--(2.1,-2.1)--cycle;
   \fill[black!16]
     (0,0)--(1.9,1.9)
     --({sqrt(1.05*1.05+1.9*1.9)},1.9)
     -- plot[domain=1.9:-1.9,samples=65,variable=\s]
              ({sqrt(1.05*1.05+\s*\s)},\s)
     --(1.9,-1.9)--cycle;
  \end{scope}
  \draw[dashed,black!60,line width=.55pt]
       (-2.1,-1.9) rectangle (2.5,1.9);
  \draw[line width=.8pt] (-1.95,-1.95)--(2.08,2.08);
  \draw[line width=.8pt] (-1.95,1.95)--(2.08,-2.08);
  \node[above right] at (2.08,2.08) {$U=0$};
  \node[below right] at (2.08,-2.08) {$V=0$};
  \node[anchor=east] at (-2.22,0) {$\mathcal O$};
  \node at (-.13,1.42) {causal future};
  \node at (-.13,-1.42) {causal past};
  \draw[black!55,line width=.65pt,domain=-1.9:1.9,samples=65,
        variable=\s] plot ({sqrt(1.05*1.05+\s*\s)},\s);
  \draw[line width=1.25pt] (0,0)--(1.05,0);
  \filldraw[fill=white,line width=.65pt] (1.05,0) circle[radius=1.7pt];
  \node[anchor=west,fill=white,inner sep=1.5pt]
        at (1.2,.17) {$\Sigma_{\rm loc}$};
  \draw[line width=.8pt,domain=-1.62:1.62,samples=65,
        variable=\s] plot ({sqrt(.55*.55+\s*\s)},\s);
  \draw[->,line width=.85pt,domain=.58:.88,samples=16,
        variable=\s] plot ({sqrt(.55*.55+\s*\s)},\s);
  \node[anchor=west,fill=white,inner sep=1.5pt] at (1.49,.83) {$\mathbf K$};
  \draw[black!65,line width=.45pt] (1.47,.82)--(1.09,.82);
  \node[anchor=west,fill=white,inner sep=1.5pt]
        at (1.5,-.82) {$0<UV<\delta^2$};
  \draw[black!65,line width=.45pt] (1.48,-.82)--(1.02,-.82);
  \filldraw[fill=white,line width=.65pt] (0,0) circle[radius=2pt];
  \node[anchor=east] at (-.35,-.03) {$S_0$};
  \node at (.27,-2.9) {(b) Enlargement of $\mathcal O$};
 \end{scope}
\end{tikzpicture}
\caption{The horizon neighborhood and the exterior section, with angular
variables suppressed. In (a), $\Sigma_1$ runs from $S_0$ to the
asymptotic end and continues past the dashed large-radius cutoff.
Panel (b) enlarges the normal chart
\eqref{prep:normal-coordinates}. The Hawking field is first constructed
in the lightly shaded future and past regions, and local rigidity gives $\Z$ throughout
$\mathcal O$. The segment $\Sigma_{\rm loc}=\{U=V\in(0,\delta)\}$
sweeps out the darker exterior region $0<UV<\delta^2$ under the
stationary flow. Both $\T$ and $\mathbf K$ project to the displayed
hyperbolas.}
\label{fig:local-horizon}
\end{figure}

To make $\Sigma_{\rm loc}$ acausal, we extend its local temporal
function. Acausality will then exclude repeated stationary
intersections by the timelike curve obtained from the Hawking flow
and the periodic rotation. We write $D(\Sigma^0)$ for its domain
of dependence in $\M$ and match the local function to a temporal
function $\tau_0$ there with zero level $\Sigma^0$.
Since $\Sigma^0$ is acausal and edgeless,
Minguzzi~\cite[Corollary~3.36, Proposition~3.43, and
Theorem~3.45]{Minguzzi19} shows that $D(\Sigma^0)$ is open,
causally convex and globally hyperbolic. The hypersurface
$\Sigma^0$ is a spacelike Cauchy hypersurface of this domain,
so the theorem of Bernal and S\'anchez~\cite[Theorem~1.2]{BS06} supplies
$\tau_0$ there with $\tau_0^{-1}(0)=\Sigma^0$.
For the following cutoff calculation, choose a smooth Riemannian
metric near $S_0$, use $|\cdot|_+$ for its norm, and write
$\ell=\operatorname{dist}_+(\cdot,S_0)$ in a tubular neighborhood.
For $0<r_1<r_2$, choose
$\chi=1$ on $\ell\le r_1$, $\chi=0$ on $\ell\ge r_2$, with
\[
\begin{aligned}
|\tau_1-\tau_0|&\le C\ell,\\
|d\chi|_+&\le\frac{C}{\ell\log(r_2/r_1)}.
\end{aligned}
\]
For example, take
$\chi=\eta(\log(\ell/r_1)/\log(r_2/r_1))$, where $\eta$ is
smooth, equals one on $(-\infty,0]$, and vanishes on $[1,\infty)$.
The estimate for $\tau_1-\tau_0$ follows from their common
zero value on $S_0$ and bounded first derivatives. Consequently
\[
 |(\tau_1-\tau_0)d\chi|_+
 \le C\ell\,\frac{\|\eta'\|_\infty|d\ell|_+}
                   {\ell\log(r_2/r_1)}
 \le\frac{C}{\log(r_2/r_1)}.
\]
Set $\tau=(1-\chi)\tau_0+\chi\tau_1$. Its differential is
\[
 d\tau
 =(1-\chi)d\tau_0+\chi d\tau_1
       +(\tau_1-\tau_0)d\chi.
\]
Here $r_2$ is fixed so that $d\tau_0$ and $d\tau_1$ are future
increasing timelike covectors on the closed tubular neighborhood.
For $0\le a\le1$, convexity of this cone and compactness give
constants $\mu,C_0>0$, independent of $r_1$, such that
\[
 \begin{aligned}
 \g^{-1}((1-a)d\tau_0+a d\tau_1,
                       (1-a)d\tau_0+a d\tau_1)&\le-\mu,\\
 |(1-a)d\tau_0+a d\tau_1|_+&\le C_0.
 \end{aligned}
\]
The norm of $\g^{-1}$ is bounded on this same neighborhood.
Expanding the quadratic form at the displayed convex combination,
and using $|(\tau_1-\tau_0)d\chi|_+\le C/\log(r_2/r_1)$, gives
\[
 \g^{-1}(d\tau,d\tau)
 \le-\mu+\frac{C C_0}{\log(r_2/r_1)}
              +\frac{C}{\log(r_2/r_1)^2}\le-\frac\mu2
\]
once $r_1$ is sufficiently small. Scaling the last summand in
$d\tau$ by any number in $[0,1]$ obeys the same estimate, so
$d\tau$ stays in the same time-oriented cone. Extend $\tau$ by
$\tau_0$ outside this neighborhood. Its zero level agrees near $S_0$ with the normal
section in the statement and is acausal, also in $\M$ by
causal convexity.

We can now prove uniqueness of the stationary intersection.
Suppose $p,\Phi_t^\T p\in\Sigma_{\rm loc}$ and write
$t=2\pi n/|\Omega_H|+s>0$, $0\le s<2\pi/|\Omega_H|$. If $s>0$, then
\[
 \Phi_{-s\Omega_H}^\Z p
 \ll\Phi_s^{\mathbf K}\Phi_{-s\Omega_H}^\Z p
 =\Phi_s^\T p.
\]
Append the $n$ stationary translates of the full timelike
segment. If $s=0$, use those segments alone. In either case
two points of $\Sigma_{\rm loc}$ are chronologically related,
a contradiction. Reversal handles $t<0$. The product map
is transverse because $\Z$ is tangent and $\mathbf K$ is
timelike, so it is a diffeomorphism onto its open image.
On this image define
\[
 \Z_{\Phi_t^\T p}=d\Phi_t^\T|_p\Z_p.
\]
The injectivity just proved makes this definition single valued.
We check agreement with the original field by choosing
the original normal neighborhood to have $0<U,V<\delta$.
For a point $q$ in this neighborhood, put
\[
 s=\frac{1}{2\kappa}\log\frac{V(q)}{U(q)}.
\]
The local stationary flow from $q$ to parameter $-s$ stays in
the normal neighborhood: both normal coordinates lie between
their values at $q$ and $\sqrt{U(q)V(q)}<\delta$.
Its endpoint $p'$ satisfies
\[
 U(p')=V(p')=\sqrt{U(q)V(q)},
 \qquad p'\in\Sigma_{\rm loc}.
\]
Thus $q=\Phi_s^\T p'$ by a flow segment entirely in the original
neighborhood. If also $q=\Phi_t^\T p$ with
$p\in\Sigma_{\rm loc}$, injectivity of the stationary product
gives $(t,p)=(s,p')$. Along this local segment, $[\T,\Z]=0$ gives
$d\Phi_t^\T\Z_p=\Z_{\Phi_t^\T p}$ by differentiating the
pushforward in $t$. For the metric and bracket,
\[
\begin{aligned}
(\Phi_t^\T)^{\ast}(\LL_\Z\g)
       &=\LL_{(\Phi_t^\T)^{\ast}\Z}(\Phi_t^\T)^{\ast}\g
       =\LL_\Z\g=0,\\
[\T,\Z]&=0.
\end{aligned}
\]
Rotational equivariance of the normal exponential map makes
$\Sigma_{\rm loc}$ invariant. Hence the extended rotational
flow is given for every $s\in\RR$ by
\begin{align*}
\Phi_s^\Z(\Phi_t^\T p)&=\Phi_t^\T(\Phi_s^\Z p),\\
\shortintertext{so that}
\Phi_{2\pi}^\Z(\Phi_t^\T p)&=\Phi_t^\T p.
\end{align*}
A smaller period on the image would restrict to a smaller
period on the original neighborhood, so the period remains effective.

It remains to show that the end approaching $S_0$ leaves every
compact subset of $Q$. If instead $p_j\to S_0$ in this section
and $\pi(p_j)$ stays in a compact subset of $Q$, write
$p_j=\Phi_{t_j}^\T q_j$ with compact representatives $q_j$.
Since $\Sigma_{\rm loc}$ is acausal,
\eqref{prep:acausal-time-bound} gives a uniform bound for $|t_j|$.
Then $p_j$ has an interior limit in $\E$, contradicting
$p_j\to S_0$.
\end{proof}

We retain the local Hawking field
$\mathbf K=\T+\Omega_H\Z$, the constant $\kappa>0$ in its normal action,
and the invariant normal frame $e_0,e_1$ constructed in the proof.
The normal coordinates \eqref{prep:normal-coordinates} have
$\mathbf K=\kappa(V\partial_V-U\partial_U)$ and $\Z U=\Z V=0$.
We also retain the temporal function $\tau$ on the interior domain
of dependence of $\Sigma^0$, equal to $(V-U)/2$ near $S_0$
and to $\tau_0$ outside the compact neighborhood used above.
These are the horizon data used in the initial hypersurface
construction and in Section~\ref{exit:section}.

For the horizon completion we will also need the nonrotating
case, which uses only GR and SBS. We record its normalization
here so that the same normal chart serves both cases.
If $\T|_{S_0}=0$, we have $\D_X\T=0$ for $X\in TS_0$.
The skew endomorphism $\D\T$ consequently vanishes on $TS_0$
and preserves its Lorentzian normal plane; here
$(\D\T)(X)=\D_X\T$, with components $\D_\beta\T^\alpha$.
In the null basis
$L,\underline L$ fixed in Section~\ref{subsec:horizon-null},
it has the form
\[
\begin{aligned}
\D_L\T&=\kappa L,\\
\D_{\underline L}\T&=-\kappa\underline L,\\
\tfrac12\operatorname{tr}((\D\T)^2)&=\kappa^2.
\end{aligned}
\]
The Killing identity and $\T|_{S_0}=0$ give
\begin{align*}
X^\alpha\D_\alpha\D_\beta\T_\gamma
   &=X^\alpha\R_{\gamma\beta\alpha\delta}\T^\delta=0,\\
\shortintertext{and hence}
X(\kappa^2)
   &=\operatorname{tr}((\D\T)\D_X(\D\T))=0
                         \qquad(X\in TS_0).
\end{align*}
Thus $\kappa$ is constant and nonzero: if $\kappa=0$, both
$\T$ and $\D\T$ vanish at any point of $S_0$, and uniqueness
for~\eqref{local:killing-connection} would give $\T=0$ everywhere.
Let $\mathbf K_0$ be the local Hawking Killing field of
Alexakis, Ionescu, and Klainerman~\cite[Theorem~1.1 and Proposition~2.3]{AIK09}. For the optical functions
$u,\underline u$ of Section~\ref{subsec:horizon-null}, their normalization
$\mathbf K_0=\underline uL-u\underline L$ on the horizons gives
\[
\begin{aligned}
 \D_L\mathbf K_0&=L,\\
 \D_{\underline L}\mathbf K_0&=-\underline L,\\
 \D_X\mathbf K_0&=0\qquad(X\in TS_0)
\end{aligned}
\quad\hbox{on }S_0.
\]
Thus $\T-\kappa\mathbf K_0$ and its first covariant derivative
vanish on $S_0$, so~\eqref{local:killing-connection} gives
$\T=\kappa\mathbf K_0$ near $S_0$. Their Proposition~3.5
makes $\T$ timelike in the nearby exterior. It is future
directed there: otherwise, for $p$ in that neighborhood and
small $h>0$, stationarity would give
$\Phi_{nh}^\T(p)\ll p$ for every positive integer $n$.
Choosing $a<b$ with $\gamma(a)\ll p\ll\gamma(b)$ by
\eqref{prep:comparison}, we would obtain
\[
 \gamma(a+nh)\ll\Phi_{nh}^\T(p)\ll p\ll\gamma(b),
\]
contrary to chronology when $nh>b-a$. Since $\mathbf K_0$ is
future directed on the future horizon, $\kappa>0$.
We therefore take $\mathbf K=\T$ and $\Omega_H=0$ in this case.
Choose any smooth oriented orthonormal normal frame $e_0,e_1$
in this case. The exponential construction and the temporal
cutoff argument in Lemma~\ref{prep:collar} apply unchanged:
\begin{align}
 \T&=\kappa(V\partial_V-U\partial_U),
           \label{prep:nonrotating-normalization}\\
 \E\cap\mathcal O&=\{U,V>0\},\notag\\
 \HH^+\cap\mathcal O&=\{U=0,V>0\}.\notag
\end{align}
In particular, the temporal function $\tau=(V-U)/2$ near $S_0$
extends to the causally convex domain $D(\Sigma^0)$ and agrees
with $\tau_0$ outside a compact neighborhood. Its zero level
is a spacelike acausal section. Thus both cases give the same
future-directed horizon generator $\mathbf K$, surface gravity
$\kappa>0$, normal chart, and temporal section for the completion
in Appendix~\ref{app:completion}.

\subsection{Regularity of the ergosurface}

\begin{definition}[Ergoregion and ergosurface]\label{def:ergoregion}
The ergoregion and ergosurface are, respectively,
\[
 \{p\in\E:\g(\T,\T)(p)>0\},
\]
and
\[
 \{p\in\E:\g(\T,\T)(p)=0\}.
\]
The same terms denote their images in $Q$, and
$E_0=\pi(\{\g(\T,\T)=0\})$.
Sets defined by these scalars are understood in the ambient
space specified in each statement.
\end{definition}

We now use nontrapping to show that the ergosurface is regular.
The reason is geometric: a null stationary orbit becomes a null
geodesic if the metric gradient of $\g(\T,\T)$ is proportional
to $\T$, and its
stationary projection is then a single point. To exclude precisely
this possibility, we compute the norm of $d[\g(\T,\T)]$ and its
relation to the twist. We first construct the potential
$d\psi=\iota_\T(\ast d\T^\flat)$. Put
$\mathbf F_{\alpha\beta}=\D_\alpha\T_\beta$.
The Killing equation gives, for every Killing field $X$,
\begin{equation}\label{local:killing-acceleration}
 (\D_XX)_\alpha
 =-X^\beta\D_\alpha X_\beta
 =-\tfrac12\D_\alpha[\g(X,X)].
\end{equation}
For $\T$, its differentiated identity also gives
\[
 \begin{aligned}
 (d\T^\flat)_{\alpha\beta}
 &=\D_\alpha\T_\beta-\D_\beta\T_\alpha=2\mathbf F_{\alpha\beta},\\
 \D_{[\alpha}\mathbf F_{\beta\gamma]}&=0,\\
 \D^\alpha\mathbf F_{\alpha\beta}&=\g^{\alpha\gamma}\R_{\beta\alpha\gamma\delta}\T^\delta=0,\\
 \LL_\T\T^\flat&=(\LL_\T\g)(\T,\cdot)+\g([\T,\T],\cdot)=0.
 \end{aligned}
\]
The brackets in $\D_{[\alpha}\mathbf F_{\beta\gamma]}$ denote
antisymmetrization. The contracted curvature in the divergence identity vanishes by
$\Ric(\g)=0$. Thus $d(\ast d\T^\flat)=0$. Since the Killing
flow preserves
the orientation and Hodge star, we obtain
\[
 d\bigl(\iota_\T(\ast d\T^\flat)\bigr)
 =\LL_\T(\ast d\T^\flat)-\iota_\T d(\ast d\T^\flat)=0.
\]
By Proposition~\ref{prep:quotient}, $\E$ is simply connected,
so this closed form has a global potential $\psi$. Since
$\T\psi=0$, we also regard $\psi$ as a function on $Q$.
If $Y$ is any local Killing field commuting with $\T$, its flow
preserves $\T^\flat$ and the Hodge star. Cartan's formula then gives
\begin{equation}\label{canon:commuting-twist}
 d(Y\psi)=\LL_Y(d\psi)
 =\LL_Y\bigl(\iota_\T(\ast d\T^\flat)\bigr)=0.
\end{equation}
Thus $Y\psi$ is constant on each connected domain of $Y$; no
periodicity assumption is needed.
The next lemma shows that $\g(\T,\T)$ and $\psi$ provide two
independent transverse coordinates at the ergosurface.

\begin{lemma}\label{canon:regularity}
The surface $E_0\subset Q$ is smooth and embedded. On
$\pi^{-1}(E_0)$,
\[
\begin{aligned}
\g^{-1}(d[\g(\T,\T)],d\psi)&=0,\\
|d[\g(\T,\T)]|_\g^2&=|d\psi|_\g^2>0.
\end{aligned}
\]
In particular, $d(\psi|_{E_0})\ne0$.
\end{lemma}
\begin{proof}
We compute at a point $p\in\pi^{-1}(E_0)$ in a null frame
adapted to $\T$, which is nonzero by
Proposition~\ref{prep:quotient}. Write
$\boldsymbol{\epsilon}=d\Vol_\g$ and choose an oriented null frame at $p$
$L=\T,\underline L,E_1,E_2$, with
$\g(L,\underline L)=-1$, $E_1,E_2$ orthonormal, and
$\boldsymbol{\epsilon}(L,\underline L,E_1,E_2)=-1$.
Writing $\mathbf F_{LA}=\mathbf F(L,E_A)$, $A\in\{1,2\}$, the contractions are
\[
\begin{aligned}
 \D_\alpha[\g(\T,\T)]&=2\T^\beta\mathbf F_{\alpha\beta},\\
 \D_\alpha\psi&=\T^\beta\boldsymbol{\epsilon}_{\beta\alpha}{}^{\gamma\delta}\mathbf F_{\gamma\delta}.
\end{aligned}
\]
The chosen orientation gives
\[
\begin{aligned}
 (\ast\mathbf F)_{L1}&=-\mathbf F_{L2},\\
 (\ast\mathbf F)_{L2}&=\mathbf F_{L1}.
\end{aligned}
\]
Here raising a null index interchanges $L,\underline L$ and
introduces a minus sign. Hence the Killing equation gives
\[
 \begin{aligned}
 d[\g(\T,\T)](L)&=d\psi(L)=0,\\
 d[\g(\T,\T)](E_A)&=-2\g(\D_L\T,E_A),\quad A=1,2,\\
 d\psi(E_1)&=d[\g(\T,\T)](E_2),\\
 d\psi(E_2)&=-d[\g(\T,\T)](E_1).
 \end{aligned}
\]
Because both covectors annihilate $L$, their metric contractions
reduce to the two spacelike frame components. We obtain
\[
 \begin{aligned}
 \g^{-1}(d[\g(\T,\T)],d\psi)
 &=-d[\g(\T,\T)](L)d\psi(\underline L)
   -d[\g(\T,\T)](\underline L)d\psi(L)\\
 &\quad+(-2\mathbf F_{L1})(-2\mathbf F_{L2})
            +(-2\mathbf F_{L2})(2\mathbf F_{L1})=0,\\
 |d\psi|_\g^2&=(-2\mathbf F_{L2})^2+(2\mathbf F_{L1})^2
                         =|d[\g(\T,\T)]|_\g^2.
 \end{aligned}
\]
We obtain
\begin{equation}\label{canon:twist}
 \begin{aligned}
 |d[\g(\T,\T)]|_\g^2=|d\psi|_\g^2
 &=4\sum_{A=1}^2\g(\D_L\T,E_A)^2,\\
 \g^{-1}(d[\g(\T,\T)],d\psi)&=0
       \qquad(\g(\T,\T)=0).
 \end{aligned}
\end{equation}
If the common value were zero, then
$\D^\alpha[\g(\T,\T)]=\lambda\T^\alpha$ for some $\lambda\in\RR$.
Stationarity preserves this identity, with the same $\lambda$,
along $\gamma(t)=\Phi_t^\T p$. Equation~\eqref{local:killing-acceleration}
gives $\D_\T\T=-\lambda\T/2$. With $ds/dt=e^{-\lambda t/2}$
and hence $\gamma'(s)=e^{\lambda t/2}\T$, we obtain
\[
 \D_{\gamma'(s)}\gamma'(s)
 =e^{\lambda t}\left(\D_\T\T+\frac{\lambda}{2}\T\right)=0.
\]
The stationary orbit is complete and proper, so it has no
endpoint in $\E$. After the displayed reparametrization it is
therefore a maximal zero-energy null geodesic whose stationary
projection is a point. This contradicts NT and proves strict
positivity in \eqref{canon:twist}. The two orthogonal spacelike
gradients are then independent, proving smoothness of $E_0$
and $d(\psi|_{E_0})\ne0$.
\end{proof}

\subsection{The local extension theorem}

We now state the local theorem that we will apply at each
extension. Ionescu and Klainerman~\cite[Theorem~2.4]{IK15} show
that a Killing field can be continued across a stationary
hypersurface if the tangent zero-energy null geodesics bend
strictly into the side carrying the field. Their condition is the
stationary refinement of the classical pseudoconvexity criterion
for wave equations, as in
H\"ormander~\cite{Hormander85}. We take the
initial side to be $f<0$, which fixes the sign in the definition
below.

\begin{definition}[$\T$-conditional pseudoconvexity]
\label{def:conditional-pseudoconvexity}
Let $f$ be smooth near $p\in\E$, with
$\T f=0$, $f(p)=0$, and $df(p)\ne0$.
It is $\T$-conditionally pseudoconvex at $p$ if
\begin{equation}\label{local:convex}
 \D^2f(X,X)<0
 \quad\text{for every }X\in T_p\E\setminus\{0\}
 \text{ with }df(X)=\g(X,X)=\g(\T,X)=0.
\end{equation}
\end{definition}

Condition~\eqref{local:convex} is called strict $\T$-null
convexity in~\cite[Definition~2.2]{IK15}. We use the
conditional pseudoconvexity terminology of
Ionescu and Klainerman~\cite{IK09}.

\begin{proposition}[Ionescu--Klainerman]\label{local:extension}
Let $p\in V\subset\E$, with $V$ open, and let
$f\in C^\infty(V)$ satisfy $\T f=0$, $f(p)=0$ and $df(p)\ne0$.
Assume $\T(p)\ne0$ and \eqref{local:convex}.
If $\Z$ is a smooth Killing field on $V\cap\{f<0\}$ with
$[\T,\Z]=0$, then, after shrinking $V$ about $p$, it has a
unique Killing extension to $V$ commuting with $\T$.
\end{proposition}

This is Ionescu and Klainerman~\cite[Theorem~2.4]{IK15}.
Uniqueness here is uniqueness on a sufficiently small connected
neighborhood: the difference of two extensions and its first
covariant derivative vanish on an open set, so the Killing
transport equations \eqref{local:killing-connection} make their
difference zero throughout that neighborhood. In each application
we first fix the field and the hypersurface and then choose the
extension neighborhood.

At the ergosurface, stationarity removes the second derivative
along $\T$, while \eqref{local:killing-acceleration} fixes the
sign of the remaining term:
\begin{equation}\label{local:ergosurface-convexity}
 \begin{aligned}
 \D^2[-\g(\T,\T)](\T,\T)
 &=-\T\bigl(\T[\g(\T,\T)]\bigr)
                           +(\D_\T\T)[\g(\T,\T)]\\
 &=-\tfrac12|d[\g(\T,\T)]|_\g^2<0,
 \end{aligned}
\end{equation}
where strictness is Lemma~\ref{canon:regularity}.
A null vector orthogonal to the nonzero null vector $\T$ is
proportional to $\T$. Thus \eqref{local:ergosurface-convexity}
is precisely the extension condition for the local side
$\g(\T,\T)>0$. In the open ergoregion we instead use the
stationary quotient of Geroch~\cite{Geroch71}, where the two
null tangent directions remain distinct.

\begin{definition}[Metrics on the space of stationary orbits]\label{def:stationary-metrics}
On $\pi(\{\g(\T,\T)\ne0\})\subset Q$, define
\begin{equation}\label{local:stationary-metrics}
\begin{aligned}
 \pi^{\ast}\bar h
 &=\g-\frac{\T^\flat\otimes\T^\flat}{\g(\T,\T)},\\
 h&=\g(\T,\T)\bar h\quad\hbox{where }\g(\T,\T)>0.
\end{aligned}
\end{equation}
The differential $d\pi$ identifies $\bar h$ with $\g|_{\T^\perp}$.
Thus $\bar h$ is Riemannian where $\g(\T,\T)<0$ and Lorentzian
where $\g(\T,\T)>0$. On the latter region,
\[
\pi^{\ast}h=\g(\T,\T)\g-\T^\flat\otimes\T^\flat.
\]
The Lorentzian metrics $h,\bar h$ have the same unparametrized
null geodesics.
\end{definition}

On $\{\g(\T,\T)>0\}$, horizontal means orthogonal to $\T$.
Each field on $Q$ has a unique stationary horizontal lift.
We use $\nabla$ and $\nabla f$ for the connection and gradient
of $h$, and indicate other connections by a metric superscript.
For a spacelike covector, $|df|_h=\sqrt{h^{-1}(df,df)}>0$.
We will use $h$ both for the extension condition and for
nontrapping, so we verify the Hessian and geodesic identities
relating it to the spacetime metric. For stationary horizontal
fields $X,Y,V$ on $\E$, the Koszul formula gives
\[
 \begin{aligned}
 2\g(\D_XY,V)
 &=X\g(Y,V)+Y\g(X,V)-V\g(X,Y)\\
 &\quad-\g(X,[Y,V])+\g(Y,[V,X])+\g(V,[X,Y])\\
 &=2\bar h(\nabla^{\bar h}_{d\pi X}(d\pi Y),d\pi V).
 \end{aligned}
\]
The last equality uses $\g|_{\T^\perp}=\pi^{\ast}\bar h$ and
$d\pi[X,Y]=[d\pi X,d\pi Y]$; vertical bracket components
are orthogonal to $X,Y,V$. Thus
\[
 d\pi(\D_XY)=\nabla^{\bar h}_{d\pi X}(d\pi Y),
\]
while its stationary component is
\[
 \g(\D_XY,\T)=-\g(Y,\D_X\T).
\]
A stationary function annihilates this vertical component.
We are therefore left with the conformal change from $\bar h$
to $h$, whose correction terms will vanish on null vectors
tangent to a level of $f$. For fields $X,Y,V$ on $Q$, since $\nabla h=0$ and $\bar h=\g(\T,\T)^{-1}h$,
\[
 \begin{aligned}
 (\nabla_X\bar h)(Y,V)
 &=-\frac{X[\g(\T,\T)]}{\g(\T,\T)^2}h(Y,V),\\
 2\bar h(\nabla^{\bar h}_XY-\nabla_XY,V)
 &=(\nabla_X\bar h)(Y,V)+(\nabla_Y\bar h)(X,V)
                              -(\nabla_V\bar h)(X,Y)\\
 &=-\frac{X[\g(\T,\T)]\bar h(Y,V)
              +Y[\g(\T,\T)]\bar h(X,V)}{\g(\T,\T)}
    +\frac{V[\g(\T,\T)]\bar h(X,Y)}{\g(\T,\T)}.
 \end{aligned}
\]
Raising the last slot with $\bar h^{-1}=\g(\T,\T)h^{-1}$ gives
\begin{equation}\label{local:conformal-connection}
 \nabla^{\bar h}_XY=\nabla_XY
 -\frac{X[\g(\T,\T)]Y+Y[\g(\T,\T)]X
                  -h(X,Y)\nabla[\g(\T,\T)]}{2\g(\T,\T)}.
\end{equation}
Consequently, for a stationary function $f$ and stationary horizontal
fields $X,Y$ on $\E$,
\[
 \begin{aligned}
 \D^2f(X,Y)
 &=X(Yf)-(\D_XY)f
  =\Hess_{\bar h}f(d\pi X,d\pi Y)\\
 &=\Hess_hf(d\pi X,d\pi Y)
   +\frac{X[\g(\T,\T)]Yf+Y[\g(\T,\T)]Xf}{2\g(\T,\T)}\\
 &\quad-\frac{h(d\pi X,d\pi Y)}{2\g(\T,\T)}
             h^{-1}(d[\g(\T,\T)],df).
 \end{aligned}
\]
For a null tangent $X$ to $f=0$, both correction terms vanish:
\begin{equation}\label{local:stationary-hessian}
 Xf=\g(X,X)=\g(\T,X)=0\quad\Longrightarrow\quad
 \D^2f(X,X)=\Hess_hf(d\pi X,d\pi X).
\end{equation}
On the timelike surface $f=0$ in $Q$, take $n=\nabla f/|df|_h$.
For a tangent vector $X$, use $Xf=|df|_h h(n,X)=0$ to obtain
\[
 \Hess_hf(X,X)
 =h(\nabla_X(|df|_h n),X) =X(|df|_h)h(n,X)+|df|_h h(\nabla_Xn,X)
  =|df|_h h(\nabla_Xn,X).
\]
We can therefore test \eqref{local:convex} on the two null
tangents of the quotient surface. If its spacetime inverse image is invariant under $\Z$
and $\operatorname{span}\{\T,\Z\}$ is timelike, its nonzero conormal lies in the spacelike plane
$(\operatorname{span}\{\T,\Z\})^\perp$, so it is spacelike.

This quotient description also preserves the geodesics in NT.
To see this, take a stationary horizontal field $X$ on $\E$. Then
\begin{align*}
\g(X,\D_X\T)
 &=X^\alpha X^\beta\D_\alpha\T_\beta=0,\\
\shortintertext{and hence}
\g(\D_XX,\T)&=0.
\end{align*}
Together with the projection formula, this gives
\[
 d\pi(\D_XX)=\nabla^{\bar h}_{d\pi X}(d\pi X).
\]
If $\g(X,X)=0$, then \eqref{local:conformal-connection} gives
\begin{equation}\label{local:null-lift}
\begin{aligned}
 d\pi(\D_XX)
 &=\nabla_{d\pi X}(d\pi X)-X(\log\g(\T,\T))\,d\pi X,\\
 \g(\D_XX,\T)&=0.
\end{aligned}
\end{equation}
For completeness, let $\lambda$ be an affine parameter of an
$h$-null geodesic and $X$ the velocity of its horizontal lift.
The two identities above imply the spacetime equation
\[
 \D_XX=-X(\log\g(\T,\T))X.
\]
Choose a new parameter $s$ by
$ds/d\lambda=\g(\T,\T)^{-1}>0$. Then
$d\gamma/ds=\g(\T,\T)X$, and the two terms cancel:
\[
 \begin{aligned}
 \D_{d\gamma/ds}(d\gamma/ds)
 &=\g(\T,\T)^2\D_XX
       +\g(\T,\T)X[\g(\T,\T)]X=0.
 \end{aligned}
\]
Conversely, an affine zero-energy null geodesic projects to a
$\bar h$-null geodesic and hence, after this parameter change
is reversed, to an $h$-null geodesic. These are local identities
on the open ergoregion; no lower bound for $\g(\T,\T)$ is
asserted at the ergosurface. When a projected curve stays in a
compact subset of the open ergoregion, this scalar has positive
lower and upper bounds, so the two parameter lengths are
comparable there. The geodesics considered in NT are maximal
as unparametrized curves, so the identification is independent
of affine completeness.

\begin{definition}[Area function]\label{def:orbit-area}
Following Chru\'sciel and Costa~\cite{CC08} and
Chru\'sciel, Costa, and Heusler~\cite{CCH12}, define
\[
 \begin{aligned}
 W=W[\T,\Z]
 &:=-\det\begin{pmatrix}
 \g(\T,\T)&\g(\T,\Z)\\
 \g(\T,\Z)&\g(\Z,\Z)
 \end{pmatrix}\\
 &=\g(\T,\Z)^2-\g(\T,\T)\g(\Z,\Z).
 \end{aligned}
\]
On $W>0$, put $\rho=\sqrt W$. Here $W$ is the area function,
and $\rho$ is its positive square root. If $[\T,\Z]=0$, their
flow parameters $(\tau,\phi)$ give the absolute orbit area density
$\rho\,|d\tau\,d\phi|$.
On this set, let $\sigma$ be the Riemannian metric induced by $\g$ on
$(\operatorname{span}\{\T,\Z\})^\perp$.
Where the fields commute and are Killing, $\sigma$ also denotes
the metric on the local orbit space, defined by orthogonal lifts.
\end{definition}

\subsection{Compatibility of the extensions}\label{subsec:compatibility}

The local theorem gives an extension near each boundary point.
To assemble these extensions and retain the period of the initial
rotation, we use the transport equations for a Killing field and its
first covariant derivative. Alexakis, Ionescu, and
Klainerman~\cite{AIK14} use the corresponding
transport system to estimate an approximate Killing field.
Foote, Han, and Oh~\cite{FHO13} formulate the
associated connection and attribute the transport identities to
Kostant~\cite{Kostant55}. We use the exact identities below to
glue the local fields and then extend their period by continuity
of the local flow.
Uniqueness for this system says that a Killing field is determined
by its value and first covariant derivative at one point.
For a Killing field $\Z$, its covariant derivative
$A=\D(\Z^\flat)$ is a two-form, with
$A_{\alpha\beta}=\D_\alpha\Z_\beta$.
Set $\mathcal J=(\Z^\flat,A)$. The Killing identity gives
\begin{equation}\label{local:killing-connection}
\begin{aligned}
\D_\alpha\Z_\beta&=A_{\alpha\beta},\\
\D_\gamma A_{\alpha\beta}
       &=\R_{\beta\alpha\gamma\delta}\Z^\delta.
\end{aligned}
\end{equation}
Thus $\mathcal J$ is parallel for the connection $\DD$ on
$T^{\ast}\M\oplus\Lambda^2T^{\ast}\M$ defined by
\[
 \DD_\gamma(\xi,A)=
 \bigl(\D_\gamma\xi_\beta-A_{\gamma\beta},\,
       \D_\gamma A_{\alpha\beta}
         -\R_{\beta\alpha\gamma\delta}\xi^\delta\bigr),
\]
where $\xi^\delta=\g^{\delta\epsilon}\xi_\epsilon$.
Conversely, if $(\xi,A)$ is a parallel section on an open set,
then $A=\D\xi$ and
\[
 (\LL_{\xi^\sharp}\g)_{\alpha\beta}
 =\D_\alpha\xi_\beta+\D_\beta\xi_\alpha
 =A_{\alpha\beta}+A_{\beta\alpha}=0.
\]
Hence $\xi^\sharp$ is Killing. In a fixed bundle frame write
$\DD_{\partial_\alpha}=\partial_\alpha+\Gamma^\DD_\alpha$.
Use the component norm in this frame. Along a coordinate path
$\gamma$, two parallel sections $\mathcal J_1,\mathcal J_2$ satisfy,
for $s\ge0$,
\begin{equation}\label{local:parallel-transport-bound}
 |\mathcal J_1-\mathcal J_2|(\gamma(s))
 \le \exp\!\left(\int_0^s
          |\dot\gamma^\alpha(v)\Gamma^\DD_\alpha(\gamma(v))|\,dv\right)
             |\mathcal J_1-\mathcal J_2|(\gamma(0)).
\end{equation}
In particular, agreement at one point gives agreement on a
connected common domain. We also use uniqueness of isometries:
if $F(p)=p$ and $dF_p=\operatorname{Id}$, then
\[
 F(\exp_pv)=\exp_{F(p)}(dF_pv)=\exp_pv.
\]
The set $\{p:F(p)=p,\ dF_p=\operatorname{Id}\}$ is therefore open and closed
in the domain of $F$.

We will also use these equations to obtain boundary limits of
$\Z$ and $\D\Z$. Fixed coordinate segments give estimates
depending on the connection and the initial values of $\mathcal J$,
even when the boundary is only a graph.
Consider a relatively compact
cylinder $(s_-,s_+)\times B$, with $B\subset\RR^3$, and fix
$s_-<s_0<s_+$ and $u:B\to(s_0,s_+)$. Suppose that $\Z$ is
Killing on the open set $\{(s,y):s<u(y)\}$. Extend $\mathcal J$ by solving
\begin{equation}\label{local:graph-transport}
 \begin{aligned}
 \DD_{\partial_s}\mathcal J^0&=0,\\
 \mathcal J^0(s_0,y)&=\mathcal J(s_0,y).
 \end{aligned}
\end{equation}
For a tangential multi-index $I$, differentiation gives
\[
 (\partial_s+\Gamma^\DD_s)\partial_y^I\mathcal J^0
 =-\sum_{\substack{J\le I\\J\ne0}}\binom{I}{J}
       (\partial_y^J\Gamma^\DD_s)
                      \partial_y^{I-J}\mathcal J^0.
\]
Here $J\le I$ is componentwise and $J\ne0$. Integrating from
$s_0$ and applying Gronwall gives, for each
$I$ and each fixed $y$ in a smaller base,
\[
 \sup_s|\partial_y^I\mathcal J^0(s,y)|
 \le C_m\left(
       |\partial_y^I\mathcal J(s_0,y)|
       +\sum_{|J|<|I|}
                   \sup_s|\partial_y^J\mathcal J^0(s,y)|\right),
       \qquad |I|\le m.
\]
Induction on $|I|$, and then the original transport equation
for each $s$ derivative, yield
\[
 \|\mathcal J^0\|_{C^m}
 \le C_m\|\mathcal J(s_0,\cdot)\|_{C^m}.
\]
The initial norm uses the same smaller base, and the constants
depend only on the fixed cylinder and the connection coefficients.
In particular, the displayed estimates are uniform as the graph
$u$ varies between $s_0$ and $s_+$.
Every segment from $(s_0,y)$ to a point with $s<u(y)$ remains
in $\{s<u(y)\}$. ODE uniqueness therefore gives
$\mathcal J^0=\mathcal J$ on this set, including all derivative
limits at its boundary.

Before gluing the local extensions, we record why completeness
also preserves their overlap. We will use this fact whenever a new
invariant neighborhood is added to the domain of the rotation.

\begin{lemma}\label{local:complete-flows}
Let $U_1,U_2$ be open subsets of a smooth manifold, and let $X_i$
be a smooth complete vector field on $U_i$, $i=1,2$. If
$X_1=X_2$ on $U_1\cap U_2$, then, for every $t\in\RR$,
\[
 \begin{aligned}
 \Phi_t^{X_1}|_{U_1\cap U_2}&=\Phi_t^{X_2}|_{U_1\cap U_2},\\
 \Phi_t^{X_1}(U_1\cap U_2)&=U_1\cap U_2.
 \end{aligned}
\]
The field $X$ defined by $X|_{U_i}=X_i$ is complete on
$U_1\cup U_2$. If $P>0$ and
$\Phi_P^{X_i}=\operatorname{Id}_{U_i}$ for $i=1,2$, then
$\Phi_P^X=\operatorname{Id}_{U_1\cup U_2}$.
\end{lemma}
\begin{proof}
For $p\in U_1\cap U_2$, uniqueness gives agreement of the two
integral curves near $t=0$. If their maximal common interval had a
finite upper endpoint $b$, completeness and continuity would give
\[
 \Phi_b^{X_1}(p)
 =\lim_{t\uparrow b}\Phi_t^{X_1}(p)
 =\lim_{t\uparrow b}\Phi_t^{X_2}(p)
 =\Phi_b^{X_2}(p)\in U_1\cap U_2.
\]
Local uniqueness at this point extends their agreement past $b$.
The lower endpoint is treated in the same way. The common flow
therefore preserves the intersection, with equality obtained by
applying the inverse flow. Each point of the union already has its
complete orbit in one $U_i$, so these orbits define the complete
flow of $X$. Its period follows by restriction to $U_1$ and $U_2$.
\end{proof}

\begin{lemma}\label{local:compatibility}
Let $U\subset\E$ be stationary and open, with a smooth field
$\Z$ satisfying $\LL_\Z\g=0$ and $[\T,\Z]=0$ on $U$.
Let $C\subset\partial U\cap\{\vartheta=0\}$ be compact,
where $\vartheta$ is the stationary temporal function of
Proposition~\ref{prep:quotient}. Suppose every $p\in C$ has
an open neighborhood $V_p$ carrying a Killing field $\Z_p$ with
$[\T,\Z_p]=0$ and $\Z_p=\Z$ on $V_p\cap U$.
Then the fields $\Z_p$ glue, after restriction, to a stationary
Killing field $\widetilde\Z$ on a stationary neighborhood $V$
of $C$, with $[\T,\widetilde\Z]=0$ and
$\widetilde\Z=\Z$ on $V\cap U$.
If $\Z$ has complete period-$2\pi$ flow on $U$ and the projected
local flow of $\widetilde\Z$ preserves $\pi(C)$, then $V$ can
be chosen invariant, with complete flow and
$\Phi_{2\pi}^{\widetilde\Z}=\operatorname{Id}_V$.
\end{lemma}
\begin{proof}
We first glue the local fields near $C$, then use the compact
period interval to obtain an invariant neighborhood. Choose
finitely many local extensions $\Z_i$ covering $C$, and set
$\mathcal J_i=(\Z_i^\flat,\D(\Z_i^\flat))$, regarded as sections
of $T^{\ast}\M\oplus\Lambda^2T^{\ast}\M$ with the connection
\eqref{local:killing-connection}. At a point $p\in C$
in their common domain, choose $p_j\in U$ tending to $p$ there.
Then
\[
 \mathcal J_i(p)-\mathcal J_k(p)
 =\lim_{j\to\infty}
      (\mathcal J_i(p_j)-\mathcal J_k(p_j))=0.
\]
Killing transport gives equality on every overlap component
meeting $C$. To exclude the other components, use the stationary
product coordinates $(\vartheta,x)$ and restrict the local
domains to $(-\epsilon_i,\epsilon_i)\times V_i$. Choose
$V_i'\Subset V_i$ covering $\pi(C)$. On $V_i\cap V_k$, the
set where $\mathcal J_i(0,x)=\mathcal J_k(0,x)$ is both open and closed.
Hence
\[
 \{x\in\overline{V_i'}\cap\overline{V_k'}:
             \mathcal J_i(0,x)\ne\mathcal J_k(0,x)\}
\]
is compact and disjoint from $\pi(C)$. Remove these finitely
many compact sets from $\bigcup_iV_i'$ and restrict to a
neighborhood $V_0$ of $\pi(C)$. The local fields agree there.
Their components and the metric coefficients are independent of
$\vartheta$, so they define one Killing field on $\RR\times V_0$.
For $x\in V_0\cap\pi(U)$, we have
$(\widetilde\Z,\D\widetilde\Z)=(\Z,\D\Z)$ at $(0,x)$. Killing transport along $\RR\times\{x\}\subset U$
proves equality on the entire old overlap.

We next obtain a common flow interval before taking a limit of
the period identity. Choose a cutoff supported in $V_0$ and
equal to one on a relatively compact neighborhood of $\pi(C)$.
Multiplying the projection of $\widetilde\Z$ by this cutoff
gives a compactly supported field on $Q$.
Its complete flow is denoted by $\widetilde\Phi_s$.
The compact invariant set $\pi(C)$ stays where the cutoff is
one. Choose $\eta>0$ smaller than its distance to the complement
of this region, using $g_+$. Smooth dependence of the complete
cut-off flow, uniformly on the compact interval $[-2\pi,2\pi]$,
gives a neighborhood $O_0$ of $\pi(C)$ such that
\[
 \sup_{|s|\le2\pi}
  \operatorname{dist}_{g_+}(\widetilde\Phi_s x,\pi(C))<\eta
                         \qquad(x\in O_0).
\]
Thus these trajectories solve the original projected equation
for the whole period interval. The spacetime lift is
\[
 \Phi_s^{\widetilde\Z}(\vartheta,x)
 =\left(\vartheta+
       \int_0^s(\widetilde\Z\vartheta)(\widetilde\Phi_vx)\,dv,
                    \widetilde\Phi_sx\right).
\]
The integrand and its derivatives are bounded on this fixed
compact family of trajectories. Thus the lift exists on
$[-2\pi,2\pi]\times(\RR\times O_0)$ and depends smoothly on
its initial point. For $p\in C$, choose $p_j\in U\cap\{\vartheta=0\}$ tending
to $p$. The new and old flows agree by ODE uniqueness, so their
period identities pass to the limit in $C^1$:
\begin{align*}
 \Phi_{2\pi}^{\widetilde\Z}(p)
 &=\lim_j\Phi_{2\pi}^{\Z}(p_j)=p,\\
 d\Phi_{2\pi}^{\widetilde\Z}|_p
 &=\lim_jd\Phi_{2\pi}^{\Z}|_{p_j}=\operatorname{Id}.
\end{align*}
The new flow preserves $\g$, since
\[
 \frac d{ds}(\Phi_s^{\widetilde\Z})^{\ast}\g
       =(\Phi_s^{\widetilde\Z})^{\ast}(\LL_{\widetilde\Z}\g)=0.
\]
Isometry uniqueness makes its time-$2\pi$ map the identity on
a stationary neighborhood $O$ of $C$. Choose a smaller
stationary neighborhood $O'$ whose projected closure is
compact in $\pi(O)$ and whose flow segments for
$0\le s\le2\pi$ remain in the constructed domain. Set
\[
 V=\bigcup_{0\le s\le2\pi}\Phi_s^{\widetilde\Z}(O').
\]
Every orbit segment from $O'$ closes after time $2\pi$.
Repeating this segment gives an integral curve for all real
times, and uniqueness identifies it with the flow from every
point on that segment. Hence $V$ is invariant and the flow is
complete and period-$2\pi$. Commutation with $\T$ makes $V$
stationary.
\end{proof}

We will apply this lemma at the boundary of the maximal domain, after obtaining
the limits of $\Z$ and $\D\Z$ from its complement in
Lemma~\ref{normal:parallel-boundary} and local invariance in
Proposition~\ref{closure:orbit-tube}. There is one further point
about the period. An effective circle action can have shorter
individual orbits, whereas we will need period $2\pi$ on every
orbit of an invariant boundary. Its coorientation lets us prove
this by examining the differential of a stabilizing isometry.

\begin{lemma}\label{local:free-action}
Let $U\subset\E$ be connected, stationary and open, and let
$\Z$ be Killing on $U$, with $[\T,\Z]=0$ and complete effective
period $2\pi$. Suppose $\Sigma\subset\pi(U)$ is a smooth
cooriented surface, possibly with boundary, preserved by the
projected action, and $W[\T,\Z]>0$ on $\pi^{-1}(\Sigma)$.
Then there is an invariant open neighborhood $O$ of $\Sigma$
in $\pi(U)$ such that
\[
 \{s\in\RR:\Phi_s^\Z(q)=q\}=2\pi\mathbb Z
 \qquad(q\in O).
\]
\end{lemma}
\begin{proof}
We first show that an element fixing a projected point fixes
its spacetime representatives. For $q=\pi(p)\in\Sigma$, let
$H_q\subset\RR/(2\pi\mathbb Z)$ be the projected stabilizer.
Freeness of the stationary action defines $a:H_q\to\RR$ by
\[
 \Phi_s^\Z p=\Phi_{a(s)}^\T p.
\]
Commutation makes $a$ a continuous homomorphism, and compactness
gives
\[
 |na(s)|=|a(ns)|\le\max_{H_q}|a|\qquad(n\in\mathbb Z).
\]
Thus $a=0$: every projected stabilizer also fixes $p$.

The plane spanned by $\T,\Z$ is Lorentzian. Let $N$ be the
cooriented unit normal to $\pi^{-1}(\Sigma)$, and let $E$ complete
it to an orthonormal basis of $\{\T,\Z\}^{\perp_\g}$.
For $s\in H_q$, the flow differential fixes $\T,\Z,N$.
It sends $E$ to $\pm E$, and preservation of orientation fixes
the positive sign. Hence $d\Phi_s^\Z|_p=\operatorname{Id}$. Isometry uniqueness
on the connected domain $U$ and effectiveness give
$s=0\bmod2\pi$.

Freeness persists near $q$. Otherwise there are
$q_j\to q$ and nonzero stabilizers $s_j\in\RR/(2\pi\mathbb Z)$.
A subsequence converges to a stabilizer of $q$, so $s_j\to0$.
The projected field is nonzero at $q$, because $W>0$, and a
flow box has a coordinate $z^1$ in which $\Z z^1=1$.
On a smaller neighborhood there is $\epsilon>0$ such that all
flow segments with $|s|<\epsilon$ stay in this box. A fixed
point there would satisfy
\[
 0=z^1(\Phi_s^\Z(q'))-z^1(q')=s,
\]
contradicting $s_j\ne0$. The circle saturation of the union
of these free neighborhoods is open, invariant and free.
\end{proof}

\subsection{The initial hypersurface}

We now choose the initial domain for the global extension. Extend the normal parameter of Lemma~\ref{prep:collar}
to the stationary coordinate $r$ on $\mathcal U_{\mathrm{loc}}$ by
requiring $\T r=0$ and
\[
 r\bigl(\exp_p(se_1(p))\bigr)=s
 \qquad(p\in S_0,\ 0<s<\delta).
\]
We use small normal tubes because the Hawking field
$\mathbf K=\T+\Omega_H\Z$ satisfies $\D_{e_1}\mathbf K=\kappa e_0$ at $S_0$,
while the metric on the transverse sphere stays nondegenerate.
For a null tangent $X$, let $Y$ be its component tangent to the
transported sphere, measured in the metric of $S_0$. The null
condition gives the leading term $-r^{-1}|Y|^2$ in $\D^2r(X,X)$. We compute it below and choose
two tubes uniformly over $S_0$: the smaller one separates the
remaining compact set from the horizon, and the larger one
provides the initial domain of the rotation.

\Needspace{10\baselineskip}
\begin{proposition}\label{prep:residual}
Let $\pi:\E\to Q=\E/\RR$ send each point to its $\T$-orbit.
Let $\Z$ be the rotational Killing field and $e_1$ the invariant
outward unit spacelike normal to $S_0$ from Lemma~\ref{prep:collar},
with corresponding stationary Gaussian coordinate $r$.
For sufficiently small $0<\delta_-<\delta_+<\delta$, where $\delta$
is the normal-coordinate radius in Lemma~\ref{prep:collar}, define in $Q$
\[
 \begin{aligned}
 C_\pm&=\pi\{\exp_p(re_1(p)):p\in S_0,\ 0<r<\delta_\pm\},\\
 \mathcal D_0&=\{\g(\T,\T)\ge0\}\setminus C_-,\\
  H_0&=\{\g(\T,\T)\ge0\}\setminus C_+.
 \end{aligned}
\]
Then $C_\pm$ are open, $\overline{C_-}\subset C_+$, and
$\pi^{-1}(C_+)$ is invariant under the complete flow of $\Z$,
of effective period $2\pi$. The sets $\mathcal D_0,H_0$
are compact, and $\mathcal D_0$ is a manifold with corners.
The relative boundary of $H_0$ in $\{\g(\T,\T)\ge0\}$ is a
smooth annulus, transverse to $E_0=\{\g(\T,\T)=0\}\subset Q$,
whose quotient by the induced circle action of $\Z$ is one compact arc.
On a spacetime neighborhood of the inverse image of this annulus,
the stationary Gaussian coordinate $r$ satisfies
\[
\begin{aligned}
\{r<\delta_+\}&\subset\pi^{-1}(C_+),\\
|dr|_\g^2&=1,\\
\D^2r(X,X)&<0
 \quad\bigl(0\ne X\in T\E,\ Xr=\g(X,X)=0\bigr).
\end{aligned}
\]
For every component $D$ of $\mathcal D_0$, the set
$D\cap\{\g(\T,\T)>0\}$ is path connected, and $D$ has a
nonempty relatively open boundary portion in $\partial C_-$ at positive distance
from $H_0$ in the fixed auxiliary Riemannian metric $g_+$ on $Q$.
\end{proposition}

\begin{proof}
We compute $\D^2r(X,X)$ for null tangent vectors $X$ in
coordinates adapted to the Hawking field. Use the invariant normals of Lemma~\ref{prep:collar}
and propagate $\exp_p(re_1)$ by $\mathbf K=\T+\Omega_H\Z$.
The translates are disjoint by rotational invariance and
Lemma~\ref{prep:collar}. Let $t$ be the $\mathbf K$ parameter,
let $y^A$, $A=1,2$, be coordinates on $S_0$, and write
$E_A=\partial_{y^A}$. We use $\g|_{TS_0}$ for the induced
sphere metric, whose components are $\g_{AB}(0)$, and write
$\D_r=\D_{\partial_r}$. Along the normal geodesics, identify
vectors by $\D$-parallel
transport and recall the boost data
\begin{align*}
 \mathbf K(0)&=0,\\
 \D_r\mathbf K(0)&=\kappa e_0.
\end{align*}
Smoothness gives $\mathbf K(r)=\kappa r e_0+O(r^2)$ and
$E_A(r)=E_A(0)+O(r)$. Since $e_0\perp T S_0$,
\[
 \left.\partial_r\g(\mathbf K,E_A)\right|_{r=0}
 =\g(\kappa e_0,E_A(0))=0.
\]
The metric coefficients and their radial derivatives therefore
satisfy, uniformly on a finite cover of $S_0$,
\begin{align*}
 \g_{tt}&=-\kappa^2r^2+O(r^3),\\
 \partial_r\g_{tt}&=-2\kappa^2r+O(r^2),\\
 \g_{tA}&=O(r^2),\\
 \partial_r\g_{tA}&=O(r),\\
 \g_{AB}&=\g_{AB}(0)+O(r),\\
 \partial_r\g_{AB}&=O(1).
\end{align*}
The radial coefficients are exact. Since $[\partial_r,E_A]=0$
and $\partial_r$ is the unit velocity of a geodesic,
\[
 \partial_r\g(\partial_r,E_A)
 =\g(\partial_r,\D_{E_A}\partial_r)
 =\tfrac12 E_A\g(\partial_r,\partial_r)=0.
\]
The Killing equation similarly gives
\[
 \partial_r\g(\mathbf K,\partial_r)
 =\g(\D_r\mathbf K,\partial_r)=0.
\]
Both inner products vanish at $r=0$, so
\[
 \g=dr^2+\g_{tt}\,dt^2+2\g_{tA}\,dt\,dy^A
                              +\g_{AB}\,dy^A dy^B.
\]
The coordinate $r$ is invariant under $\mathbf K,\Z$ and
hence under $\T$.

Let $X=X^t\mathbf K+Y^AE_A$ be null and tangent to a level of $r$,
and write $|Y|^2=\g_{AB}(0)Y^AY^B$. The coefficient bounds above
give
\[
 0=-\kappa^2r^2(X^t)^2+|Y|^2
       +O(r^3(X^t)^2+r^2|X^t||Y|+r|Y|^2).
\]
The mixed term is bounded using
$2\kappa r|X^t||Y|\le\kappa^2r^2(X^t)^2+|Y|^2$.
There is consequently a fixed constant $C_1$, depending only
on the metric in the chosen normal neighborhood and $\kappa$,
such that
\[
 \left|\kappa^2r^2(X^t)^2-|Y|^2\right|
 \le C_1r\bigl(\kappa^2r^2(X^t)^2+|Y|^2\bigr).
\]
For $C_1r\le1/2$, moving the first square to the left gives
\[
 \frac13|Y|^2\le\kappa^2r^2(X^t)^2\le3|Y|^2,
 \qquad
 \left|\kappa^2r^2(X^t)^2-|Y|^2\right|
                                      \le4C_1r|Y|^2.
\]
These bounds control the apparent singularity in $X^t$.
Since $\g_{rr}=1$ and $\g_{rt}=\g_{rA}=0$, the second
fundamental form of a level of $r$ is one half its radial
metric derivative. Hence
\[
 \begin{aligned}
 \D^2r(X,X)
 &=\tfrac12(\partial_r\g_{tt})(X^t)^2
       +(\partial_r\g_{tA})X^tY^A
       +\tfrac12(\partial_r\g_{AB})Y^AY^B,\\
 \left|\D^2r(X,X)+\kappa^2r(X^t)^2\right|
 &\le C\bigl(r^2(X^t)^2+r|X^t||Y|+|Y|^2\bigr)
 \le C_2|Y|^2.
 \end{aligned}
\]
Substituting the null equation cancels the leading time component:
\begin{equation}\label{prep:strict-calculation}
 \begin{aligned}
 \left|\D^2r(X,X)+\frac{|Y|^2}{r}\right|
 &\le\left|\D^2r(X,X)+\kappa^2r(X^t)^2\right|
       +\frac{\left|\kappa^2r^2(X^t)^2-|Y|^2\right|}{r}\\
 &\le(C_2+4C_1)|Y|^2,\\
 \D^2r(X,X)&\le-\frac{|Y|^2}{2r}<0
              \qquad(X\ne0).
 \end{aligned}
\end{equation}
For the last line choose $2(C_2+4C_1)r\le1$; a nonzero null
$X$ has $Y\ne0$ by the preceding comparison. Compactness of
$S_0$ and invariance under the Hawking flow make $C_1,C_2$,
and thus the allowed positive range of $r$, independent of the
point and the null tangent vector.
In particular, at the level $r=\delta_+$ in the statement,
\[
\begin{aligned}
\D^2(r-\delta_+)&=\D^2r,\\
|d(r-\delta_+)|_\g^2&=1,\\
\{r<\delta_+\}&\subset\pi^{-1}(C_+).
\end{aligned}
\]
We can therefore apply Proposition~\ref{local:extension} at
$r=\delta_+$: the strict inequality holds, and the field is defined
on the entire local side $r<\delta_+$.

The restriction of the circle action to $S_0\simeq S^2$ has
exactly two fixed points. Indeed, at a zero of $\Z|_{S_0}$ its
differential is a nonzero skew rotation, so the zero is isolated
with index one; the index theorem gives two zeros. The action on
the Lorentzian normal plane is trivial, since a compact connected
subgroup of its orientation-preserving Lorentz group is trivial.
On $T_pS_0$ the weight is $\pm1$: a larger weight would give a
nonidentity circle element fixing $p$ with identity differential
on all of $T_p\M$, contrary to effectiveness and isometry uniqueness.
Consequently, at either pole $p$ and for $V\in T_pS_0$,
\begin{align*}
 \D_V\T&=-\Omega_H\D_V\Z,\\
 \g(\D_V\Z,\D_V\Z)&=\g(V,V),\\
 \Hess_{S_0}[\g(\T,\T)]_p(V,V)
 &=2\g(\D_V\T,\D_V\T)=2\Omega_H^2\g(V,V).
\end{align*}
Here $\T(p)=0$, so the restriction of the spacetime Hessian
is the intrinsic Hessian. These local facts are all we need
from the axis geometry.

It remains to describe the part of this level in the ergoregion.
We show that it is an annulus meeting $E_0$ transversely. Identify
each small positive level of $r$ in $Q$ with $S_0$ by
$p\mapsto\pi(\exp_p(re_1(p)))$ and transport geodesic coordinates $y^1,y^2$
from either pole. At the pole on $r=0$,
\[
 \partial_{y^A}\partial_{y^B}[\g(\T,\T)]
 =2\Omega_H^2\delta_{AB}.
\]
Hence the coordinate Hessian remains bounded below by $c>0$
times the Euclidean metric on a fixed disk, for all sufficiently
small $r\ge0$.
Its center remains on the axis, so rotational invariance
gives $\partial_y[\g(\T,\T)](r,0)=0$, whereas
\[
 \g(\T,\T)(r,0)=\g(\mathbf K,\mathbf K)(r,0)<0
 \qquad(r>0).
\]
The norm remains positive on the compact complement of the
two coordinate disks, including their boundaries, for small $r$.
For every Euclidean unit vector $\omega$ in the coordinate plane,
\[
 \begin{aligned}
 \frac{d}{d\lambda}\g(\T,\T)(r,\lambda\omega)
 &=\int_0^\lambda
   \partial_{y^A}\partial_{y^B}[\g(\T,\T)](r,t\omega)
                     \omega^A\omega^B\,dt\\
 &\ge c\lambda.
 \end{aligned}
\]
Each radial segment therefore crosses the zero set exactly
once and transversely. The two zero sets are smooth circles,
and the part of this level where $\g(\T,\T)>0$ is the annulus
between the disks they enclose.
The stationary lift of each level has invariant radial
coorientation. On the closed annulus,
\[
 \g(\mathbf K,\mathbf K)<0\le\g(\T,\T),
\]
with
\[
 \mathbf K=\T+\Omega_H\Z.
\]
Thus $\T,\Z$ are independent and span a Lorentzian plane.
Lemma~\ref{local:free-action}, applied
on the connected local stationary neighborhood, makes the projected
action free on this annulus.
The quotient of this compact connected annulus by the free
circle action is a compact connected one-dimensional manifold
with boundary. Each boundary circle is preserved: a continuous
circle action cannot interchange the two components. Its
orbits are open in that boundary circle by nonvanishing of the
generator and closed by compactness, so each boundary circle
is one orbit. The quotient therefore has two boundary points
and is a compact interval.
Choose $0<\delta_-<\delta_+$ sufficiently small and let $C_\pm$ be the
stationary images of $0<r<\delta_\pm$. These are the required nested
neighborhoods, and their boundary hypersurfaces satisfy
\eqref{prep:strict-calculation}.

For compactness of $\mathcal D_0$ and $H_0$, we examine their
two possible ends. At the
horizon end, the exact boost coordinates
\eqref{prep:normal-coordinates} move every sufficiently nearby
point into the neighborhood parametrized by $0<r<\delta_-$:
\begin{align*}
\Phi_t^{\mathbf K}(U,V,p)&=(e^{-\kappa t}U,e^{\kappa t}V,p),\\
\shortintertext{so that}
\Phi_{\frac1{2\kappa}\log(U/V)}^{\mathbf K}(U,V,p)
        &=(\sqrt{UV},\sqrt{UV},p).
\end{align*}
Throughout this segment both normal coordinates stay between
$\min(U,V)$ and $\max(U,V)$. The endpoint lies on
$\Sigma_{\rm loc}$ at $r=\sqrt{UV}$. Since the $\mathbf K$-
and $\T$-saturations coincide, every sequence in $\Sigma_1$
approaching $S_0$ eventually projects into $C_-$.
At infinity $\g(\T,\T)<0$. Hence, for some $\epsilon>0$,
$R<\infty$,
\[
 \mathcal D_0,\ H_0\ \subset\
 \pi\bigl(\Phi_0(\{1+\epsilon\le|x|\le R\})\bigr).
\]
Both are closed and thus compact. The transverse
intersections with the smooth ergosurface make
$\mathcal D_0$ a manifold with corners.

A connected manifold with corners is path connected. Join two
points of $D\cap\{\g(\T,\T)>0\}$ by a path in $D$.
Near its intersections with the face $\g(\T,\T)=0$, use
$\g(\T,\T)$ as one boundary coordinate. At a corner, keep
the other boundary coordinate, $r-\delta_-$, fixed. A partition
of unity in finitely many such charts gives a smooth vector
field $V$ satisfying, on the two faces near this path,
\[
 V[\g(\T,\T)]>0\quad\hbox{on }\{\g(\T,\T)=0\},
 \qquad Vr=0\quad\hbox{on }\{r=\delta_-\}.
\]
Choose its cutoff to equal one near the intersections of the
path with $\{\g(\T,\T)=0\}$ and vanish near the two endpoints. Its flow
for a sufficiently small positive time preserves $D$ and moves
every point of the path where $\g(\T,\T)=0$ into
$\{\g(\T,\T)>0\}$. Points already in this set stay there:
a first crossing of the zero face would contradict the
displayed strict derivative.
It therefore supplies the required path in
$D\cap\{\g(\T,\T)>0\}$ with the original endpoints.

Finally, every component $D$ of $\mathcal D_0$ meets
$\partial C_-$. Otherwise $\partial D$ is a nonempty compact
surface contained in $E_0$: it is nonempty because $D$ is compact
and $Q$ is connected and noncompact. At each point of $E_0$ the
regularity in Lemma~\ref{canon:regularity} gives a single positive
side, so each component of $\partial D$ is both open and closed
in $E_0$. On this compact surface $\psi$ attains a maximum,
contradicting $d(\psi|_{E_0})\ne0$.
At an intersection of $D$ with $\partial C_-$, transversality to
$E_0$ supplies points with $\g(\T,\T)>0$. A small relatively
open boundary portion there has closure in $C_+$ and hence
positive $g_+$-distance from $H_0$. Compactness of the manifold
with corners $\mathcal D_0$ gives finitely many components, so
these distances have a common positive lower bound.
\end{proof}

We take $\pi^{-1}(C_+)$ as the initial domain of $\Z$. The
complete action there, together with its extension to
$\mathcal U_{\mathrm{loc}}$ in Lemma~\ref{prep:collar}, supplies
the field on a neighborhood of the initial relative boundary.
\subsection{A future section of the horizon}\label{subsec:future-completion}

In our one-ended setting, the $I^+$-regularity condition of
Chru\'sciel and Costa~\cite[Definition~1.1]{CC08} requires,
besides completeness of $\T$ and global hyperbolicity of $\E$,
a connected acausal spacelike hypersurface $\mathcal S\subset\E$
containing an asymptotic section. Its closure must be a topological
manifold with boundary, consisting of a compact part and that
end, and its boundary must meet every
generator of the future horizon precisely once, with
\[
 \partial\overline{\mathcal S}
 \subset\partial\E\cap I^+(\M^{(\mathrm{end})}).
\]
We will construct $\overline{\mathcal S}$ with smooth boundary.

To verify the asymptotic-flatness condition in
Chru\'sciel and Costa~\cite{CC08}, we compute the second fundamental form on the asymptotic slice. In the coordinates
$(t,x^1,x^2,x^3)$ of GR we have $\T=\partial_t$.
Let $\gamma_{ij}=\g_{ij}$,
$i,j\in\{1,2,3\}$, be the induced metric, with inverse
$\gamma^{ij}$ and Levi--Civita connection $\nabla^\gamma$.
The shift vector has components $\gamma^{ij}\g_{0j}$, so
lowering its index with $\gamma$ gives $\g_{0i}$.
The induced metric is $\delta_{ij}+O_6(|x|^{-1})$.
The future unit normal and second
fundamental form satisfy
\[
 \begin{aligned}
 \mathbf n&=\sqrt{-\g^{00}}(\partial_t-\gamma^{ij}\g_{0j}\partial_i),\\ 
 k_{ij}&=-\g(\D_i\mathbf n,\partial_j),\\
 \frac{2k_{ij}}{\sqrt{-\g^{00}}}
 &=-\partial_t\gamma_{ij}
       +\nabla_i^\gamma\g_{0j}+\nabla_j^\gamma\g_{0i}\\
 &=\partial_i\g_{0j}+\partial_j\g_{0i}
                  -2\Gamma(\gamma)^\ell_{ij}\g_{0\ell}.
 \end{aligned}
\]
We used stationarity, $\partial_t\gamma=0$, on the last line.
The falloff in~\eqref{intro:af} gives
$\sqrt{-\g^{00}}=1+O_6(|x|^{-1})$ and
\[
\begin{aligned}
 \partial_i\g_{0j}+\partial_j\g_{0i}&=O_5(|x|^{-3}),\\
 \Gamma(\gamma)^\ell_{ij}\g_{0\ell}&=O_5(|x|^{-4}).
\end{aligned}
\]
Thus, for every spatial multiindex $I$,
\[
 \partial^I k_{ij}=O(|x|^{-3-|I|}),\qquad |I|\le5.
\]
In particular,
\begin{equation}\label{exit:asymptotic-data}
\begin{aligned}
 \gamma_{ij}-\delta_{ij}&=O_6(|x|^{-1}),\\
 k_{ij}&=O_5(|x|^{-3}).
\end{aligned}
\end{equation}
These bounds verify the required asymptotic falloff.

GR already gives completeness of $\T$, global hyperbolicity,
and the required asymptotic end. To obtain the remaining
$I^+$-regularity conditions, we deform the section through
$S_0$ slightly to the future and keep its end fixed.
Its images under the stationary flow provide a future horizon
whose generators are complete to the future.
We record the resulting extension here and defer its construction
to Appendix~\ref{app:completion}.

\begin{lemma}\label{exit:regular-completion}
Assume GR and SBS. The spacetime $(\E,\g)$ embeds isometrically
as the domain of outer communications of a smooth stationary
vacuum spacetime which is $I^+$-regular in the sense of
Chru\'sciel and Costa~\cite[Definition~1.1]{CC08} and has one connected nondegenerate
future horizon. Its stationary field restricts to $\T$.
If $\T|_{S_0}\not\equiv0$, the circle action generated by the field
$\Z$ of Lemma~\ref{prep:collar} extends to a neighborhood of this
horizon.
\end{lemma}
The proof is given in Appendix~\ref{app:completion}. Its causal
point is that the added region cannot send a future causal curve
back into the exterior. In the normal coordinates $U,V$ of
\eqref{prep:normal-coordinates}, the exterior has $UV>0$.
The differential inequality~\eqref{app:causal-barrier} for the
positive part of $UV$ shows that a future causal curve starting
in $UV\le0$ remains there. This keeps the domain of outer
communications equal to $\E$ and preserves acausality of the
deformed section.
\section{Null geometry and complete rotational orbits}
\label{sec:action}\label{sec:barriers}

We now bound the rotational direction and its first derivative
independently of the domain carrying the field. The intrinsic statement is
Lemma~\ref{action:intrinsic-orbits}: a bound for the lengths of
confined null geodesics controls the normalized direction of a
hypersurface-orthogonal periodic timelike Killing field and its
first derivative. We first derive the vacuum
circularity identities and use NT to check that this lemma applies
to the timelike field $\Z$ on the Lorentzian quotient.
The same circularity identities express the local extension
condition as positive geodesic curvature on a transverse disk.
Together, this curvature formula and the orbit estimate will let
us determine the boundary geometry in
Section~\ref{sec:limit-geometry}.

\subsection{Circularity}

We begin with the vacuum circularity calculation of
Papapetrou~\cite{Papapetrou66} and
Carter~\cite{Carter69}.
We derive the identities without dividing by the determinant
of the two Killing fields, and determine the constant twist
contractions from their limiting values on $S_0$. Nontrapping
will then exclude degeneracy of $\operatorname{span}\{\T,\Z\}$
and show that this plane is timelike throughout the domain.

Throughout this section, let $U\subset\E$ be a connected
$\T$-invariant open set containing
$\pi^{-1}(C_+)$ from
Proposition~\ref{prep:residual}. Let $\Z$ be a Killing field on $U$
which agrees on $\pi^{-1}(C_+)$ with the field of
Lemma~\ref{prep:collar}. Assume
$[\T,\Z]=0$ and that its flow is complete with effective period
$2\pi$. The estimates below are uniform
over all such $U,\Z$. Auxiliary norms on $Q$ and $\E$ refer,
respectively, to the fixed metrics $g_+$ and $\mathbf g_+$.
We use $\nabla$ for the connection of the Lorentzian metric
$h$ on $\{\g(\T,\T)>0\}\subset Q$, as in
Definition~\ref{def:stationary-metrics}; $\D$ remains the
spacetime connection. As before, $\Z$ also denotes the projected
field on $\pi(U)$: contractions with $\g$ use its spacetime
value, and contractions with $h$ use its projection. Recall the area
function $W=\g(\T,\Z)^2-\g(\T,\T)\g(\Z,\Z)$ from
Definition~\ref{def:orbit-area}. On the ergoregion,
\[
 h(\Z,\Z)=\g(\T,\T)\g(\Z,\Z)-\g(\T,\Z)^2=-W.
\]
Thus $W>0$ is precisely the assertion that the rotational
orbits are timelike for $h$. Circularity explains how nontrapping
enters: if $\operatorname{span}\{\T,\Z\}$ were degenerate,
a constant combination of the two fields would have a null orbit
which is a geodesic up to reparametrization and projects to
one compact circle. The circle action then excludes an axis in the closed ergoregion.
The estimate for the normalized rotational direction will use the
complete-orbit argument below.

\begin{lemma}\label{action:area}
The circularity identities
\begin{equation}\label{action:circularity}
\begin{aligned}
 \T^\flat\wedge\Z^\flat\wedge d\T^\flat&=0,\\
 \T^\flat\wedge\Z^\flat\wedge d\Z^\flat&=0
\end{aligned}
\end{equation}
hold on $U$. On $U\cap\{\g(\T,\T)\ge0\}$ one has
\[
\begin{aligned}
\g(\Z,\Z)&>0,\\
W[\T,\Z]&>0,
\end{aligned}
\]
and the distribution $\{\T,\Z\}^{\perp_\g}$ is integrable.
\end{lemma}

\begin{proof}
We first derive circularity in a form that also applies at a
point where $\operatorname{span}\{\T,\Z\}$ is degenerate.
For either Killing field $X=\T,\Z$,
the vacuum computation preceding Lemma~\ref{canon:regularity}
gives $d(\ast dX^\flat)=0$ and
$\LL_X(\ast dX^\flat)=0$. In particular,
\[
 d\bigl(\iota_X(\ast dX^\flat)\bigr)
  =\LL_X(\ast dX^\flat)-\iota_Xd(\ast dX^\flat)=0.
\]
The contraction identity
$\iota_X(\ast dX^\flat)=\ast(X^\flat\wedge dX^\flat)$ identifies
this with the twist one-form. Commutation gives
$\LL_\Z\T^\flat=0$ and $\LL_\T\Z^\flat=0$.
Applying Cartan's formula, we obtain
\[
 d\bigl(\iota_\Z\ast(\T^\flat\wedge d\T^\flat)\bigr)
 =\LL_\Z\ast(\T^\flat\wedge d\T^\flat)
       -\iota_\Z d\ast(\T^\flat\wedge d\T^\flat)=0,
\]
and similarly
\[
 d\bigl(\iota_\T\ast(\Z^\flat\wedge d\Z^\flat)\bigr)=0.
\]
On $S_0$ the relation $\T=-\Omega_H\Z$ gives
\[
 \T^\flat\wedge\Z^\flat=0,
\]
so both constant twist contractions have zero limit there.
They consequently vanish on the connected set $U$. This proves
\eqref{action:circularity}.
Where $\T,\Z$ are independent, let $A,B$ be smooth local
sections of $\{\T,\Z\}^{\perp_\g}$ and choose $E_1,E_2$ with
\[
 \det\begin{pmatrix}
 \T^\flat(E_1)&\T^\flat(E_2)\\
 \Z^\flat(E_1)&\Z^\flat(E_2)
 \end{pmatrix}\ne0.
\]
For $X=\T,\Z$, evaluation of the four-form gives
\[
 \begin{aligned}
 0&=(\T^\flat\wedge\Z^\flat\wedge dX^\flat)(E_1,E_2,A,B)\\
  &=\det\begin{pmatrix}
 \T^\flat(E_1)&\T^\flat(E_2)\\
 \Z^\flat(E_1)&\Z^\flat(E_2)
 \end{pmatrix}dX^\flat(A,B),\\
 \g(X,[A,B])
 &=A\g(X,B)-B\g(X,A)-dX^\flat(A,B)=0.
 \end{aligned}
\]
The common orthogonal distribution is involutive wherever
the two fields are independent. Recalling the twist potential
$d\psi=\ast(\T^\flat\wedge d\T^\flat)$, we also obtain
$\Z\psi=0$ on $U$.

Where $\Z\ne0$, its periodic orbit is spacelike by causality:
$\g(\Z,\Z)$ is constant on the orbit, and a nonpositive value
would give a closed causal curve. The fields $\T,\Z$ are
independent there, since commutation and dependence at a point
would give a periodic stationary orbit, contrary to
Proposition~\ref{prep:quotient}.
We claim that their plane is nondegenerate. If it were
degenerate at $p$, we could follow its null direction by fixing
\[
 \Omega=-\frac{\g(\T,\Z)}{\g(\Z,\Z)}(p)
\]
and put $V=\T+\Omega\Z$. This nonzero Killing field
spans the null direction of that plane at $p$. Since $\T,\Z$
are complete on $U$ and commute, $V$ is complete. Its flow
preserves both fields, so along the whole orbit through $p$,
\[
 \g(V,V)=\g(V,\Z)=\g(V,\T)=0.
\]
Circularity and the Killing equation give
\[
 \begin{aligned}
 V^\flat\wedge\Z^\flat\wedge dV^\flat&=0,\\
 \iota_VdV^\flat&=-d[\g(V,V)],\\
 0=\iota_V(V^\flat\wedge\Z^\flat\wedge dV^\flat)
   &=-V^\flat\wedge\Z^\flat\wedge d[\g(V,V)].
 \end{aligned}
\]
Thus $d[\g(V,V)]$ is a linear combination of
$V^\flat,\Z^\flat$. Commutation shows that it annihilates
$\Z$. Since
$\g(V,\Z)=0$ and $\g(\Z,\Z)>0$, its $\Z^\flat$
coefficient must vanish. Along this orbit,
$d[\g(V,V)]=\lambda V^\flat$ for a smooth scalar $\lambda$,
so \eqref{local:killing-acceleration} gives
$\D_VV=-\lambda V/2$.
The orbit is a reparametrized null geodesic of zero stationary
energy. Recall the temporal coordinate $\vartheta$ of
Proposition~\ref{prep:quotient}, normalized by $\T\vartheta=1$.
It satisfies
\[
 \vartheta(\Phi_s^V p)
       =s+\vartheta(\Phi_{\Omega s}^{\Z}p).
\]
The second term is bounded by periodicity, so this geodesic
has no endpoint in $\E$ and is maximal as an unparametrized
geodesic. Its stationary projection lies in one compact
$\Z$-orbit, contrary to NT.

To determine the signature on the whole domain, we also need
the local structure of the axis. At a fixed point $p$, we average
a future timelike vector $V$:
\[
 \widetilde V=\int_0^{2\pi}d\Phi_s^\Z|_p V\,ds.
\]
The vector $\widetilde V$ is future timelike and fixed.
Normalize it and choose an orthonormal basis in
$\widetilde V^\perp$. The circle
representation has the form
\[
 d\Phi_s^\Z|_p=\operatorname{diag}
 \left(1,1,
 \begin{pmatrix}\cos(ms)&-\sin(ms)\\\sin(ms)&\cos(ms)\end{pmatrix}\right),
 \qquad m\in\mathbb Z\setminus\{0\}.
\]
The case $m=0$ would make every isometry the identity by its
value and differential at $p$. The fixed space is therefore
a Lorentzian two-plane. Since
\[
 \Phi_s^\Z(\exp_p v)=\exp_p(d\Phi_s^\Z|_p v),
\]
the local fixed set is the exponential image of that plane,
so the axis is smooth of codimension two.
Perturbing paths off this submanifold shows that its complement
is connected. The determinant never vanishes there.
Near $S_0$, the span contains the timelike field
$\mathbf K=\T+\Omega_H\Z$, so it is timelike everywhere
off the axis.

Suppose an axis point $p$ lies in the ergoregion. Choose a unit
vector $E_1$ in the rotated spacelike plane and let $E_2$ be its
orthogonal rotation, with weight $m\ne0$. Along
$\gamma(r)=\exp_p(rE_1)$, parallel transport gives
\[
 \Z(\gamma(r))=mrE_2+O(r^2).
\]
The stationary vector $\T(p)$ lies in the fixed plane. Therefore
\[
 \left.\frac d{dr}\g(\T,\Z)(\gamma(r))\right|_{r=0}
 =\g(\D_{E_1}\T,\Z)(p)+\g(\T,mE_2)(p)=0.
\]
It follows that
\begin{align*}
 \g(\T,\Z)(\gamma(r))&=O(r^2),\\
 \g(\Z,\Z)(\gamma(r))&=m^2r^2+O(r^3),\\
 W(\gamma(r))&=-\g(\T,\T)(p)m^2r^2+O(r^3)<0
\end{align*}
for small $r>0$, contradicting positivity off the axis.
At an axis point on the ergosurface, both $\T$ and
$\bigl(d[\g(\T,\T)]\bigr)^\sharp$ lie in the fixed Lorentzian two-plane.
Since $\T\ne0$ is null and annihilates $\g(\T,\T)$,
$\bigl(d[\g(\T,\T)]\bigr)^\sharp\in\RR\T$, and hence
\[
|d[\g(\T,\T)]|_\g^2=0,
\]
contrary to the strict positivity proved after
\eqref{canon:twist}. This excludes the axis throughout the
set in the statement.
\end{proof}

For every smooth $\T,\Z$-invariant hypersurface in the
ergoregion with an invariant normal, Lemma~\ref{local:free-action}
and $W>0$ show that its projected circle action is free, with
effective period $2\pi$. The same argument applies on any smooth
invariant part of the ergosurface where the field is defined.

With circularity and $W>0$ established, we can choose local
coordinates adapted to the timelike rotational direction.
For a spacetime vector $A$ orthogonal to $\T$ and $\Z$,
\[
 h(\Z,d\pi A)
 =\g(\T,\T)\g(\Z,A)-\g(\T,\Z)\g(\T,A)=0.
\]
Thus $d\pi$ identifies $\{\T,\Z\}^{\perp_\g}$ with
$\Z^{\perp_h}$. Choosing stationary sections $A,B$ of the
first distribution, circularity gives
\[
 [d\pi A,d\pi B]=d\pi[A,B]\in\Z^{\perp_h}.
\]
The projected distribution is therefore integrable. We flow one
of its local integral surfaces in $Q$ by $\Z$, using the flow
parameter $\phi$. This gives
\begin{equation}\label{action:reduction}
 h=-W\,d\phi^2+q,
\end{equation}
where $\Z=\partial_\phi$ and $q=\g(\T,\T)\sigma$.
Here $\sigma$ is the metric of Definition~\ref{def:orbit-area},
identified with a metric on the transverse surface by the
$\T$-horizontal lift, and $\rho=\sqrt W$.
The Riemannian metric $q$ acts on the transverse surface,
and its coefficients and $\rho$ are independent of $\phi$.
Equivalently, $q$ is the restriction of $h$ to $\Z^{\perp_h}$.
For a $\Z$-invariant scalar $u$, the notation
$|du|_q^2=h^{-1}(du,du)$ therefore has the same meaning in
every such chart.

\subsection{Curvature and conditional pseudoconvexity}\label{subsec:optical}

The three-dimensional metric $h$ in \eqref{action:reduction}
is locally static with respect to $\Z$. Its associated optical
metric, in the convention of Casey, Dunajski, Gibbons, and
Warnick~\cite{CDGW11}, is
\[
 q_{\mathrm{opt}}=W^{-1}q=\frac{\g(\T,\T)}{W}\sigma.
\]
Then
\begin{equation}\label{local:optical-product}
 \rho^{-2}h=-d\phi^2+q_{\mathrm{opt}},
\end{equation}
so the $\phi$ direction is orthogonal to the transverse surfaces
and has constant squared length $-1$ in the rescaled metric. An affinely
parametrized null geodesic $(\phi(s),\gamma(s))$ of this product satisfies
\begin{equation}\label{local:optical-geodesics}
\begin{aligned}
 \phi''&=0,\\
 \nabla^{q_{\mathrm{opt}}}_{\dot\gamma}\dot\gamma&=0,\\
 |\dot\gamma|_{q_{\mathrm{opt}}}^2&=(\phi')^2.
\end{aligned}
\end{equation}
Consequently, a unit-speed geodesic $\gamma$ of $q_{\mathrm{opt}}$
gives a null geodesic $s\mapsto(\phi_0+s,\gamma(s))$ of
$\rho^{-2}h$. After reparametrization this is a null geodesic
of $h$, which lifts horizontally, with a further change of
parameter, to a spacetime null geodesic orthogonal to $\T$
by \eqref{local:null-lift}.
We shall relate the signed geodesic curvature in $q_{\mathrm{opt}}$
to the second fundamental form of a timelike surface in $(Q,h)$.
The conformal correction vanishes on null tangent vectors, which
will give the extension inequality and identify its equality case.

\begin{lemma}\label{local:optical-curvature}
Let $\Sigma$ be a smooth $\Z$-invariant timelike surface in
$\pi(U\cap\{\g(\T,\T)>0\})$, represented locally by a curve
in a transverse surface with metric $q_{\mathrm{opt}}$. Let
$e_{\mathrm{opt}},N_{\mathrm{opt}}$ be the unit tangent and normal
of this curve for $q_{\mathrm{opt}}$, identified with their
$\phi$-independent lifts to the product chart. Define the signed
geodesic curvature by
\[
 \nabla^{q_{\mathrm{opt}}}_{e_{\mathrm{opt}}}e_{\mathrm{opt}}
       =\kappa_{\mathrm{opt}}N_{\mathrm{opt}}.
\]
Define its $h$-unit normal and null tangents by
\begin{align*}
 n&=\rho^{-1}N_{\mathrm{opt}},\\
 e_4&=\frac{-\partial_\phi+e_{\mathrm{opt}}}{\sqrt2\rho},\\
 e_3&=\frac{-\partial_\phi-e_{\mathrm{opt}}}{\sqrt2\rho},
\end{align*}
and put $B(X,Y)=-h(\nabla_Xn,Y)$ for $X,Y\in T\Sigma$.
Writing $B_{ab}=B(e_a,e_b)$ for $a,b\in\{3,4\}$,
we have $h(e_3,e_4)=-1$ and
\begin{equation}\label{local:optical-null-curvature}
 B_{33}=B_{44}=\frac{\kappa_{\mathrm{opt}}}{2\rho}.
\end{equation}
If $f$ is a smooth defining function for $\Sigma$, with $\Z f=0$
and $nf>0$, its pullback by $\pi$ satisfies on $\pi^{-1}(\Sigma)$
\begin{equation}\label{flow:null-extension}
 \D^2(f\circ\pi)(X,X)=-(nf)B(d\pi X,d\pi X)
 \quad\bigl(X(f\circ\pi)=\g(X,X)=\g(\T,X)=0\bigr).
\end{equation}
Thus $\kappa_{\mathrm{opt}}>0$ is equivalent to
\eqref{local:convex}. If $\kappa_{\mathrm{opt}}\equiv0$, the integral curves of
$e_3$ and $e_4$ are null geodesics of $(Q,h)$ up to reparametrization.
\end{lemma}
\begin{proof}
For this calculation write $\widehat h=\rho^{-2}h$ and let
$\widehat\nabla$ be its connection. The conformal connection
calculation in \eqref{local:conformal-connection}, now with
$h=\rho^2\widehat h$, gives
\[
 \nabla_XY
 =\widehat\nabla_XY+(X\log\rho)Y+(Y\log\rho)X
            -\widehat h(X,Y)\operatorname{grad}_{\widehat h}\log\rho.
\]
Extend $N_{\mathrm{opt}}$ independently of $\phi$. For $X\in T\Sigma$,
$\widehat h(X,N_{\mathrm{opt}})=0$, and differentiation of
$n=\rho^{-1}N_{\mathrm{opt}}$ gives the cancellation
\begin{align*}
 \nabla_Xn
 &=-\rho^{-1}(X\log\rho)N_{\mathrm{opt}}
     +\rho^{-1}\widehat\nabla_XN_{\mathrm{opt}}
     +\rho^{-1}(X\log\rho)N_{\mathrm{opt}}
     +\rho^{-1}(N_{\mathrm{opt}}\log\rho)X\\
 &=\rho^{-1}\bigl\{\widehat\nabla_XN_{\mathrm{opt}}
                     +(N_{\mathrm{opt}}\log\rho)X\bigr\}.
\end{align*}
Hence, for tangent $X,Y$,
\[
 B(X,Y)=-\rho\widehat h(\widehat\nabla_XN_{\mathrm{opt}},Y)
             -\rho(N_{\mathrm{opt}}\log\rho)\widehat h(X,Y).
\]
For $a=3,4$, the second term vanishes when $X=Y=e_a$.
The connection of $-d\phi^2+q_{\mathrm{opt}}$ and
$q_{\mathrm{opt}}(e_{\mathrm{opt}},N_{\mathrm{opt}})=0$ therefore give
\begin{align*}
 B(e_a,e_a)
 &=-\frac1{2\rho}q_{\mathrm{opt}}
        (\nabla^{q_{\mathrm{opt}}}_{e_{\mathrm{opt}}}N_{\mathrm{opt}},
                                                   e_{\mathrm{opt}})\\
 &=\frac1{2\rho}q_{\mathrm{opt}}
        (N_{\mathrm{opt}},\nabla^{q_{\mathrm{opt}}}_{e_{\mathrm{opt}}}
                                             e_{\mathrm{opt}})
 =\frac{\kappa_{\mathrm{opt}}}{2\rho}.
\end{align*}
This proves \eqref{local:optical-null-curvature}. The normalization
$h(e_3,e_4)=-1$ agrees with the normalization of the horizon null
pair in Alexakis, Ionescu, and Klainerman~\cite{AIK09}.
Here $e_3,e_4$ are tangent to a timelike surface in $(Q,h)$,
and $B$ differentiates its spacelike normal. The horizon forms
$\chi_{AB},\underline\chi_{AB}$ of
\eqref{prep:null-second-forms} instead differentiate null normals
to spacelike two-surfaces in spacetime.

For the defining function in the statement, $\nabla f=(nf)n$
on $\Sigma$, where $nf=df(n)$. By
\eqref{local:stationary-hessian},
\begin{align*}
 \D^2(f\circ\pi)(X,X)
 &=\Hess_hf(d\pi X,d\pi X)\\
 &=h\bigl(\nabla_{d\pi X}((nf)n),d\pi X\bigr)\\
 &=(nf)h(\nabla_{d\pi X}n,d\pi X)
 =-(nf)B(d\pi X,d\pi X).
\end{align*}
The conormal is spacelike, since
\[
 \g^{-1}(\pi^{\ast}df,\pi^{\ast}df)=\g(\T,\T)h^{-1}(df,df)
              =\g(\T,\T)(nf)^2>0.
\]
Every nonzero $h$-null vector tangent to $\Sigma$ is proportional
to $e_3$ or $e_4$, proving the equivalence with \eqref{local:convex}.

If $\kappa_{\mathrm{opt}}\equiv0$, set
$Y=-\partial_\phi+e_{\mathrm{opt}}$. The product connection gives
\begin{align*}
 \widehat\nabla_YY
 &=\nabla^{q_{\mathrm{opt}}}_{e_{\mathrm{opt}}}e_{\mathrm{opt}}=0,\\
 \nabla_YY&=2(Y\log\rho)Y.
\end{align*}
Since $e_4=Y/(\sqrt2\rho)$, its acceleration is
\[
 \nabla_{e_4}e_4
 =\frac{\nabla_YY-(Y\log\rho)Y}{2\rho^2}
 =(e_4\log\rho)e_4.
\]
Replacing $e_{\mathrm{opt}}$ by $-e_{\mathrm{opt}}$ gives
$\nabla_{e_3}e_3=(e_3\log\rho)e_3$. Thus both normalized
null fields have integral curves that are geodesics up to
reparametrization.
\end{proof}

To see the side on which the extension takes place, let
$\gamma,\underline\gamma$ be the affine $h$-null geodesics from
$p\in\Sigma$ with velocities $e_4,e_3$. Then
\begin{align*}
 f(\gamma(s))&=-\tfrac12(nf)(p)B_{44}(p)s^2+O(s^3),\\
 f(\underline\gamma(s))&=-\tfrac12(nf)(p)B_{33}(p)s^2+O(s^3).
\end{align*}
When $\kappa_{\mathrm{opt}}>0$, both geodesics lie in $f<0$
for small nonzero $s$. This is the side on which the Killing
field is defined in Proposition~\ref{local:extension}. When
$\kappa_{\mathrm{opt}}$ vanishes identically, they remain on the
surface, giving the geodesics to which we will apply NT.
Figure~\ref{fig:optical-convexity} shows these null directions
and their projections onto the surface with metric $q_{\mathrm{opt}}$.

\begin{figure}[htbp]
\centering
\begin{tikzpicture}[x=1.45cm,y=1.25cm,font=\normalsize,
                   line cap=round,line join=round]
 \node at (0,2.35) {(a) Timelike surface in $(Q,h)$};
 \node at (5.75,2.35) {(b) Surface with metric $q_{\mathrm{opt}}$};
 \begin{scope}[shift={(-.25,0)}]
  % Local product model: x=.3y^2, tau=-phi, projected obliquely.
  % The dashed ambient null lines leave the curved sheet on its f<0 side.
  \draw[black!60,line width=.55pt]
   plot[domain=-1.3:1.3,samples=50,variable=\s]
       ({.3*\s*\s+.8*\s-.46},{.35*\s+1.2075});
  \draw[black!60,line width=.55pt]
       (-.993,.7525)--(-.073,-1.6625);
  \draw[black!60,line width=.55pt]
       (1.087,1.6625)--(2.007,-.7525);
  \draw[black!60,line width=.55pt]
   plot[domain=-1.3:1.3,samples=50,variable=\s]
       ({.3*\s*\s+.8*\s+.46},{.35*\s-1.2075});
  \draw[black!35,line width=.4pt] (-.46,1.2075)--(.46,-1.2075);
  \draw[line width=.85pt]
   plot[domain=-1.3:1.3,samples=50,variable=\s]
       ({.3*\s*\s+.8*\s},{.35*\s});
  \node at (.92,1.10) {$\Sigma$};
  \draw[dashed,line width=.75pt] (-.42,-1.47)--(.50,1.75);
  \draw[dashed,line width=.75pt] (1.32,-.77)--(-1.56,.91);
  \node[above right] at (.50,1.75) {$\gamma$};
  \node[above left] at (-1.56,.91) {$\underline\gamma$};
  \draw[->,line width=.9pt] (0,0)--(.20,.70);
  \draw[->,line width=.9pt] (0,0)--(-.60,.35);
  \node[right=4pt,fill=white,inner sep=1pt] at (.20,.70) {$e_4$};
  \node[above left=1pt,fill=white,inner sep=1pt] at (-.60,.35) {$e_3$};
  \draw[->,line width=.9pt] (0,0)--(1.50,0);
  \node[below=3pt] at (1.30,0) {$n$};
  \draw[->,line width=.8pt] (2.34,-.56)--(2.02,.28);
  \node[above right=1pt] at (2.02,.28) {$-\partial_\phi$};
  \fill (0,0) circle[radius=1.7pt];
  \node[below left=4pt,fill=white,inner sep=1pt] at (0,0) {$p$};
 \end{scope}
 \begin{scope}[shift={(5.75,0)}]
  \fill[black!10]
   plot[domain=-1.7:1.7,samples=50] ({.28*\x*\x},\x)
   --(-1.35,1.7)--(-1.35,-1.7)--cycle;
  \draw[line width=.85pt]
   plot[domain=-1.7:1.7,samples=50] ({.28*\x*\x},\x);
  \node[right] at (.86,1.62) {$f=0$};
  \node at (-.86,.14) {$f<0$};
  \node at (1.58,.77) {$f>0$};
  \draw[dashed,line width=.65pt] (0,-1.68)--(0,1.68);
  \draw[->,line width=.85pt] (0,0)--(0,.77);
  \draw[->,line width=.85pt] (0,0)--(0,-.77);
  \node[left=3pt] at (0,.86) {$e_{\mathrm{opt}}$};
  \node[left=3pt] at (0,-.86) {$-e_{\mathrm{opt}}$};
  \draw[->,line width=.85pt] (0,0)--(1.66,0);
  \node[below] at (1.01,-.08) {$N_{\mathrm{opt}}$};
  \fill (0,0) circle[radius=1.7pt];
  \node[below right=3pt] at (0,0) {$p$};
 \end{scope}
 \node at (2.85,-2.22) {$W^{-1}h=-d\phi^2+q_{\mathrm{opt}}$};
\end{tikzpicture}
\caption{The second fundamental form of $\Sigma$ and the geodesics
of $q_{\mathrm{opt}}$.
In (a), $e_4,e_3$ are tangent to the timelike surface $\Sigma$,
whereas the dashed ambient null geodesics with these initial
velocities leave it on the side $f<0$ for both signs of the
parameter. The curved sheet is shown in a local perspective view,
with the circle parameter $\phi$ unwrapped. In (b), the two
projections onto the transverse surface traverse the same
$q_{\mathrm{opt}}$-geodesic in opposite directions. This curve lies
in the shaded side where $\Z$ is defined because
$\kappa_{\mathrm{opt}}>0$. The normal arrows point toward $f>0$.
The dashed geodesics are drawn schematically as straight lines.}
\label{fig:optical-convexity}
\end{figure}

\subsection{Nontrapping on compact sets}

For the compactness arguments, we need a quantitative consequence
of NT: a zero-energy null segment confined over a fixed compact
set has bounded auxiliary length. We obtain it by parametrizing
the geodesic flow at unit speed for $\mathbf g_+$, which gives
a smooth equation through the ergosurface as well.

Recall the fixed stationary auxiliary metric $\mathbf g_+$ on
$\E$. Its unit tangent bundle is invariant under
$(p,v)\mapsto(\Phi_t^\T p,d\Phi_t^\T v)$.
We work on the quotient of this bundle by that action, with
base projection $[p,v]\mapsto\pi(p)\in Q$, and write $(p,v)$
for a representative, with $v\in T_p\E$ of $\mathbf g_+$-length
one. The geodesic equation parametrized by
$\mathbf g_+$-arclength induces a smooth vector field on this
quotient; $v$ remains a spacetime vector, not a vector in $TQ$.

\begin{lemma}\label{action:escape}
The equation
\begin{equation}\label{action:normalized-flow}
 \D_vv=-\tfrac12(\D_v\mathbf g_+)(v,v)v
\end{equation}
defines a smooth stationary local flow on the $\mathbf g_+$-unit
tangent bundle and preserves
\[
 \g(v,v)=\g(\T,v)=0.
\]
A trajectory whose projection remains in a compact $K\subset Q$
has a limiting position and unit velocity modulo stationary
translations at every finite parameter endpoint, and its
spacetime lift extends through that endpoint.
Under NT, there is $C_K<\infty$, depending only on $K$, the
stationary geometry and $\mathbf g_+$, such that every zero-energy
null geodesic segment $\gamma:I\to\E$ satisfies
\[
 \pi(\gamma(I))\subset K\quad\Longrightarrow\quad
                 \int_I|\dot\gamma|_{\mathbf g_+}\,ds\le C_K.
\]
\end{lemma}
\begin{proof}
We use the unit initial data to pass to limits of geodesics.
If an affinely parametrized geodesic is reparametrized by
$\mathbf g_+$-arclength, its velocity $v$ satisfies
$\D_vv\in\RR v$. Differentiating $\mathbf g_+(v,v)=1$
then determines the coefficient in \eqref{action:normalized-flow}.
Conversely, a solution of that equation satisfies
\begin{align*}
 \frac d{ds}\mathbf g_+(v,v)
 &=(\D_v\mathbf g_+)(v,v)+2\mathbf g_+(\D_vv,v)\\
 &=(\D_v\mathbf g_+)(v,v)\bigl(1-\mathbf g_+(v,v)\bigr).
\end{align*}
Thus unit initial velocities remain unit. Metric compatibility
and the Killing equation for $\T$ give
\[
 \begin{aligned}
 \frac d{ds}\g(v,v)&=-(\D_v\mathbf g_+)(v,v)\g(v,v),\\
 \frac d{ds}\g(\T,v)
 &=-\tfrac12(\D_v\mathbf g_+)(v,v)\g(\T,v).
 \end{aligned}
\]
The normalized geodesic vector field preserves the closed set
$\g(v,v)=\g(\T,v)=0$ in the unit tangent bundle. The quotient unit
bundle over $K$ is compact. For a fixed auxiliary distance on
this bundle, boundedness of the vector field gives
\[
 \operatorname{dist}\bigl([\gamma(s),v(s)],
                         [\gamma(s'),v(s')]\bigr)
       \le C_K|s-s'|.
\]
Thus $[\gamma(s),v(s)]$ has a limit at every finite endpoint. In the
stationary product coordinates supplied by
Proposition~\ref{prep:quotient}, its lift satisfies
\[
 |\vartheta(\gamma(s))-\vartheta(\gamma(s'))|
       \le C_K|s-s'|,
\]
because $d\vartheta$ is stationary and bounded on the compact
unit bundle. The position and velocity therefore have a
spacetime limit, and local existence for
\eqref{action:normalized-flow} continues the lift.

If the asserted length bound failed, we could take confined segments
with lengths tending to infinity. We center them at their
midpoints and orient them to the future, so that a limit of
their unit initial data will give a geodesic defined in both
directions.
After passing to a subsequence, their unit initial data satisfy
$\xi_j=[p_j,v_j]\to\xi$ in the compact quotient unit bundle
over $K$.
Cut off the induced normalized field on the quotient outside a
relatively compact neighborhood of this bundle.
Denote its complete cut-off flow by $\Phi_s$ and use the same
auxiliary Riemannian distance on the quotient bundle.
In a finite coordinate cover of this compact set, the smooth
vector field has bounded first derivatives. The variational
equation for its flow and Gronwall's inequality give, for
every fixed $R$,
\[
 \sup_{|s|\le R}\operatorname{dist}(\Phi_s\xi_j,\Phi_s\xi)
       \le C_R\operatorname{dist}(\xi_j,\xi)\longrightarrow0.
\]
For large $j$ the centered original segment contains $[-R,R]$,
so its limit stays in the closed quotient unit bundle over $K$ there.
As $R$ is arbitrary, the limiting trajectory solves the
original equation on $\RR$ without entering the cutoff region.
On the compact set of stationary classes of future null vectors
with $\mathbf g_+(v,v)=1$ over $K$, the stationary covector
$d\vartheta$ satisfies
\[
 0<c_K\le d\vartheta(v)\le C_K.
\]
The spacetime lift exists on every finite interval and obeys
$\vartheta(\gamma(s))\to\pm\infty$. It is a maximal
zero-energy null geodesic confined over $K$, contrary to NT.

\end{proof}

\begin{remark}\label{action:localized-nontrapping}
We can locate more precisely the part of the exterior on which
NT imposes a restriction. The local horizon geometry already
excludes compactly trapped null geodesics near the horizon,
while a zero-energy null geodesic cannot enter a region where
$\T$ is timelike. We show below that it therefore suffices to
test NT on a single compact part of the closed ergoregion,
separated from the horizon. This gives exactly the length bound
used in Theorem~\ref{discrete:continuation}, where an obstruction
to extending the Killing field produces a zero-energy null
geodesic confined over that compact set.

Assume GR and SBS, with $\T|_{S_0}\not\equiv0$. In the stationary
horizon coordinates of Lemma~\ref{prep:collar}, choose any sufficiently
small $\delta_+>0$. Let $C_+\subset Q$ be the image of
$0<r<\delta_+$ and put $H_0=\{\g(\T,\T)\ge0\}\setminus C_+$,
as in the definitions of Proposition~\ref{prep:residual}.
Every maximal unparametrized zero-energy null
geodesic whose stationary projection is contained in a compact
subset of $Q$ has its projection entirely in $H_0$.
Consequently NT is equivalent to the existence of $C<\infty$
such that every zero-energy null geodesic segment $\gamma:I\to\E$
satisfies
\[
 \pi(\gamma(I))\subset H_0
 \quad\Longrightarrow\quad
 \int_I|\dot\gamma|_{\mathbf g_+}\,ds\le C.
\]
\end{remark}
\begin{proof}
We use the local horizon geometry to exclude a minimum of the
radial coordinate along a trapped geodesic. The construction of
$C_+$, compactness of $H_0$, and the calculation
\eqref{prep:strict-calculation} use GR and SBS. In particular,
throughout $\pi^{-1}(C_+)$,
\[
 \D^2r(X,X)<0
 \quad\bigl(X\ne0,\ \g(X,X)=0,\ Xr=0\bigr).
\]
The normalized-flow and finite-endpoint arguments in the preceding
proof use the stationary product and compactness. Hence a maximal
zero-energy null geodesic confined over a compact $K\subset Q$
can be parametrized by $\mathbf g_+$-arclength on all of $\RR$.
Let $\gamma$ be such a parametrization and suppose that its
projection meets $C_+$. By the last assertion of
Lemma~\ref{prep:collar}, compact subsets of $Q$ stay away from
$r=0$. Therefore
\[
 0<r_{\ast}:=
 \inf\{r(\gamma(s)):\pi(\gamma(s))\in C_+\}<\delta_+.
\]
Choose times $s_j$ with $\pi(\gamma(s_j))\in C_+$ and
$r(\gamma(s_j))\to r_{\ast}$.
Using the stationary coordinate $\vartheta$ of
Proposition~\ref{prep:quotient}, define
\[
 \gamma_j(s)=
 \Phi_{-\vartheta(\gamma(s_j))}^{\T}\bigl(\gamma(s_j+s)\bigr).
\]
The initial positions lie on $\vartheta=0$ over $K$, and their
velocities have $\mathbf g_+$-length one. After passing to a
subsequence, these initial data converge to $(p,v)$ with
$|v|_{\mathbf g_+}=1$. Since $r_{\ast}\in(0,\delta_+)$ and
$S_0$ is compact, the normal parametrization gives
$\pi(p)\in C_+$ and $r(p)=r_{\ast}$.
Smooth dependence for \eqref{action:normalized-flow} now gives a
limiting geodesic $\gamma_\infty$ with initial data $(p,v)$ on a
neighborhood of $s=0$.
For sufficiently small $|s|$, its projection and those of
$\gamma_j$ lie in $C_+$, so
\begin{align*}
 r(\gamma_\infty(s))&\ge r_{\ast},\\
 r(\gamma_\infty(0))&=r_{\ast},\\
 dr(v)&=0.
\end{align*}
The null and zero-energy constraints pass to the limit. At this
minimum the normalized geodesic equation gives
\[
 \begin{aligned}
 0\le\left.\frac{d^2}{ds^2}r(\gamma_\infty(s))\right|_{s=0}
 &=\D^2r(v,v)
   -\tfrac12(\D_v\mathbf g_+)(v,v)\,dr(v)\\
 &=\D^2r(v,v)<0,
 \end{aligned}
\]
a contradiction. Thus $\pi(\gamma(\RR))\cap C_+=\varnothing$.
Where $\T$ is timelike its orthogonal complement is positive
definite, so it contains no nonzero null vector. We conclude that
\[
 \pi(\gamma(\RR))
 \subset\{\g(\T,\T)\ge0\}\setminus C_+=H_0.
\]
Under NT, Lemma~\ref{action:escape} applied to the compact set
$H_0$ gives the stated length bound. Conversely, if NT fails,
the preceding argument gives a normalized geodesic on $\RR$
with projection in $H_0$. The bound would imply
\[
 2R=\int_{-R}^{R}|\dot\gamma(s)|_{\mathbf g_+}\,ds\le C
 \qquad(R>0),
\]
which is impossible.
\end{proof}

The exclusion of the horizon neighborhood uses
\eqref{prep:strict-calculation}, without imposing NT there.
For a broader analysis of null geodesics near nondegenerate
Killing horizons, see Gannot~\cite{Gannot19}, who also uses
zero-energy nontrapping to obtain Fredholm properties of
stationary wave operators under his global hypotheses.

\subsection{Estimates along complete orbits}

We now prove a bound for the rotational direction that is
uniform over the domains carrying $\Z$. The quantity to control
is its auxiliary size $|\Z|_{g_+}/\rho$. Although
$h(\Z/\rho,\Z/\rho)=-1$, the unit timelike hyperboloid is
noncompact, so we need to use the geometry of the complete orbit.

Recall that $U\supset\pi^{-1}(C_+)$ is connected and stationary,
and that $\Z$ is Killing on $U$, commutes with $\T$, and has
a complete flow of effective period $2\pi$. Circularity gives a system
of ordinary differential equations for the velocity and acceleration along each orbit, with
constant coefficients in a parallel frame. We will show that an
unbounded sequence of $|\Z|_{g_+}/\rho$ yields a null limit of
this system whose auxiliary length is infinite. Since the limit
stays over a compact set, the length estimate of
Lemma~\ref{action:escape} gives the contradiction. The orbit
argument is intrinsic to Lorentzian geometry, so we first state
it under its geometric hypotheses and then verify them here.

\begin{lemma}\label{action:intrinsic-orbits}
Let $h$ be a smooth Lorentzian metric on a three-dimensional
manifold $Q$, let $g_+$ be a smooth Riemannian metric, and let
$K\subset Q$ be compact. Suppose every $h$-null geodesic segment
contained in $K$ has $g_+$-length at most $L_K$.
For every hypersurface-orthogonal timelike Killing field $\Z$
on an open subset of $Q$, put $\rho=\sqrt{-h(\Z,\Z)}$.
On every periodic $\Z$-orbit contained in $K$,
\begin{equation}\label{action:intrinsic-bound}
 \left|\frac{\Z}{\rho}\right|_{g_+}
 +\left|\nabla\left(\frac{\Z}{\rho}\right)\right|_{g_+}
 +|d\log\rho|_{g_+}\le C_K.
\end{equation}
Here $\nabla$ is the connection of $h$, and $g_+$ induces all
vector, covector, and tensor norms in the estimate. The constant depends
only on $K,h,g_+$ and $L_K$, independently of the field and
its domain.
\end{lemma}
\begin{proof}
Hypersurface orthogonality gives local coordinates in which
$\Z=\partial_\phi$ and $h=-\rho^2d\phi^2+q$, with $q$ Riemannian
and all coefficients independent of $\phi$. Thus the $h$-gradient
$\nabla\rho$ is the $q$-gradient on the transverse surface.
Let $x^A$, $A=1,2$, be coordinates on the transverse surface,
and let $\Gamma$ denote the Christoffel symbols of $h$ in
$(\phi,x^1,x^2)$. From $h_{\phi\phi}=-\rho^2$ and $h_{A\phi}=0$,
\begin{align*}
 \Gamma^\phi_{A\phi}
 &=\tfrac12(-\rho^{-2})(-2\rho\partial_A\rho)
   =\rho^{-1}\partial_A\rho,\\
 \Gamma^A_{\phi\phi}
 &=-\tfrac12q^{AB}\partial_B(-\rho^2)
   =\rho q^{AB}\partial_B\rho.
\end{align*}
Moreover $\Gamma^\phi_{AB}=\Gamma^A_{B\phi}=0$, since the coefficients
are independent of $\phi$. It follows that
\begin{align*}
 \nabla_\Z\Z&=\rho\nabla\rho,\\
 \nabla_\Z\nabla\rho
 &=\rho^{-1}q^{AB}(\partial_A\rho)(\partial_B\rho)\Z
   =\rho^{-1}|d\rho|_q^2\Z,\\
 \nabla_\Z(\nabla_\Z\Z)&=|d\rho|_q^2\Z.
\end{align*}
Both coefficients $\rho,|d\rho|_q$ are constant along an orbit.
We normalize the largest auxiliary speed on the orbit to one.
At a point attaining this maximum, the second derivative of the
speed will bound both the acceleration and the coefficient in
the orbit equation. Set
$m=\max_{\rm orbit}|\Z|_{g_+}$, start at a maximum, and
parametrize by $\gamma(t)=\Phi_{t/m}^\Z p$. Write
$V=\dot\gamma=\Z/m$ for its velocity and $A=\nabla_VV$ for its
acceleration, and set $\nu=|d\rho|_q/m$. Then
\begin{equation}\label{action:orbit-ode}
 \begin{aligned}
 \nabla_VV&=A,\\
 \nabla_VA&=\nu^2V,
 \end{aligned}
\end{equation}
with $|V|_{g_+}\le1$ and $|V(0)|_{g_+}=1$.
At the maximum, differentiating the fixed auxiliary metric
first gives
\[
 \frac d{dt}g_+(V,V)=(\nabla_Vg_+)(V,V)+2g_+(A,V).
\]
We differentiate the metric and acceleration terms separately,
using $\nabla_VV=A$ and $\nabla_VA=\nu^2V$:
\begin{align*}
 \frac d{dt}\bigl[(\nabla_Vg_+)(V,V)\bigr]
 &=(\nabla^2g_+)(V,V,V,V)+(\nabla_Ag_+)(V,V)
       +2(\nabla_Vg_+)(A,V),\\
 \frac d{dt}\bigl[2g_+(A,V)\bigr]
 &=2(\nabla_Vg_+)(A,V)+2\nu^2g_+(V,V)+2g_+(A,A).
\end{align*}
At the maximum this yields
\[
 \begin{aligned}
 0&\ge\left.\frac{d^2}{dt^2}g_+(V,V)\right|_{t=0}\\
 &=2|A(0)|_{g_+}^2+2\nu^2
   +(\nabla^2g_+)(V,V,V,V)
   +(\nabla_Ag_+)(V,V)+4(\nabla_Vg_+)(A,V).
 \end{aligned}
\]
Here $\nabla^2g_+$ is the second covariant derivative of the
auxiliary metric formed with the connection of $h$. The constants in this
proof depend only on $K,h,g_+$ and the assumed length bound $L_K$.
Boundedness of the metric derivatives on $K$ gives
\begin{align*}
2|A(0)|_{g_+}^2+2\nu^2
 &\le C+C|A(0)|_{g_+}\le |A(0)|_{g_+}^2+C,\\
\shortintertext{and therefore}
|A(0)|_{g_+}^2+\nu^2&\le C.
\end{align*}
Since $\nu=|d\rho|_q/m$, this also gives
\[
 |d\rho|_q\le Cm.
\]
Let $P_t:T_{\gamma(0)}Q\to T_{\gamma(t)}Q$ be $h$-parallel
transport. Its operator norm uses $g_+$ at the two endpoints. The
two equations in \eqref{action:orbit-ode} have the exact solution
\[
 \begin{pmatrix}P_t^{-1}V(t)\\P_t^{-1}A(t)\end{pmatrix}
 =\exp\left(t\begin{pmatrix}0&1\\\nu^2&0\end{pmatrix}\right)
   \begin{pmatrix}V(0)\\A(0)\end{pmatrix}.
\]
For $\nu>0$ the two entries are
\[
\begin{aligned}
P_t^{-1}V(t)&=\cosh(\nu t)V(0)+\frac{\sinh(\nu t)}\nu A(0),\\
P_t^{-1}A(t)&=\nu\sinh(\nu t)V(0)+\cosh(\nu t)A(0).
\end{aligned}
\]
At $\nu=0$ these expressions have continuous limits
$V(0)+tA(0)$ and $A(0)$, so the estimates are uniform
as $\nu\to0$.
For any nonzero parallel vector $Y$, differentiation gives
\[
 \left|\frac d{dt}\log|Y|_{g_+}\right|
 =\left|\frac{(\nabla_Vg_+)(Y,Y)}{2g_+(Y,Y)}\right|
 \le C|V|_{g_+}.
\]
Integration, with $|V|_{g_+}\le1$, gives
$\|P_t\|_{g_+}\le e^{C|t|}$. The initial bound and the matrix
formula therefore bound $V,A$ on every fixed time interval,
uniformly over all complete orbits.

We can now prove the bound for $m/\rho$. Suppose it fails and
take complete orbits over $K$ with $\rho_j/m_j\to0$, where the
subscript $j$ denotes the quantities just defined on the $j$th
orbit. Their normalized velocities approach the null cone,
while the preceding estimate bounds the accelerations and
coefficients needed to pass to the orbit equation. Pass to a subsequence of the bounded
initial data with $\nu_j\to\nu\ge0$. In a coordinate chart
\eqref{action:orbit-ode} is the fixed system
\[
\begin{aligned}
\dot\gamma^i&=V^i,\\
\dot V^i&=A^i-\Gamma^i_{k\ell}V^kV^\ell,\\
\dot A^i&=\nu^2V^i-\Gamma^i_{k\ell}V^kA^\ell,\\
\dot\nu&=0,
\end{aligned}
\]
where $i,k,\ell=1,2,3$ and the Christoffel symbols are those
of the fixed metric $h$.
For every fixed $R$, the preceding estimates put the solutions
on $[-R,R]$ in one compact subset of the bundle of
$(\gamma,V,A,\nu)$. Cut off the smooth system outside that set.
Continuous dependence on the convergent initial data gives
uniform convergence on $[-R,R]$. The limit stays over $K$ and
therefore solves the original system. Uniqueness identifies the
limits as $R$ increases, giving a solution on all of $\RR$.
The identities
\[
\begin{aligned}
h(V_j,V_j)&=-\rho_j^2/m_j^2,\\
h(V_j,A_j)&=0,\\
h(A_j,A_j)&=\rho_j^2\nu_j^2/m_j^2
\end{aligned}
\]
pass to the limit. Since $|V(0)|_{g_+}=1$, the vector $V(0)$
is nonzero and null. Complete it to a null frame
$V(0),\underline V,E$ at $\gamma(0)$, with
$h(V(0),\underline V)=-1$ and $h(E,E)=1$.
For some $b\in\RR$, the two limiting orthogonality relations give
\[
\begin{aligned}
 A(0)&=bV(0)+h(A(0),E)E,\\
 0&=h(A(0),A(0))=h(A(0),E)^2,
\end{aligned}
\]
so $A(0)=bV(0)$.
Reverse time if necessary
to have $b\ge0$. The matrix formula for $h$-parallel transport
$P_t:T_{\gamma(0)}Q\to T_{\gamma(t)}Q$ now gives
\[
 V(t)=
 \begin{cases}
 [\cosh(\nu t)+(b/\nu)\sinh(\nu t)]P_tV(0),&\nu>0,\\
 (1+bt)P_tV(0),&\nu=0.
 \end{cases}
\]
Both scalar factors are at least one for $t\ge0$, and
$\nabla_VV=A\in\RR V$. We have therefore obtained a nonzero null curve on
$[0,\infty)$ which is a geodesic up to reparametrization. Its image stays in $K$, so the assumed length
bound gives
\[
 \ell(t)=\int_0^t|V(s)|_{g_+}\,ds\le L_K.
\]
On the other hand, integrating the logarithmic derivative
of the parallel vector $P_tV(0)$ gives
\[
 \log\frac{|P_tV(0)|_{g_+}}{|V(0)|_{g_+}}
 \ge -C\int_0^t|V(s)|_{g_+}\,ds=-C\ell(t).
\]
The displayed formula for $V$ therefore implies
\begin{align*}
|V(t)|_{g_+}&\ge|P_tV(0)|_{g_+}\ge e^{-C\ell(t)}\ge c_K,\\
\shortintertext{and hence}
\ell(t)&\ge c_Kt,
\end{align*}
a contradiction. Hence $m+|d\rho|_q\le C\rho$.

It remains to control the derivative of the normalized field.
For $X=X^\phi\partial_\phi+X^A\partial_{x^A}$, the Christoffel
symbols computed at the start of the proof give
\begin{align*}
 \nabla_X\Z
 &=X^A\Gamma^\phi_{A\phi}\partial_\phi
       +X^\phi\Gamma^A_{\phi\phi}\partial_{x^A}\\
 &=\rho^{-1}X(\rho)\Z+X^\phi\rho\nabla\rho\\
 &=(X\log\rho)\Z-h(X,\Z)\nabla\log\rho.
\end{align*}
Differentiating $\rho^{-1}\Z$ now cancels the first term:
\begin{align*}
 \nabla_X\left(\frac{\Z}{\rho}\right)
 &=-\rho^{-1}(X\log\rho)\Z+\rho^{-1}\nabla_X\Z\\
 &=-h\left(X,\frac{\Z}{\rho}\right)\nabla\log\rho.
\end{align*}
We have bounded $\Z/\rho$ in the auxiliary norm, with
$h(\Z/\rho,\Z/\rho)=-1$. To make the metric comparison
uniform in the field, observe that
\[
 \inf\left\{h(X,X):
 \begin{array}{l}
 p\in K,\ u,X\in T_pQ,\ h(u,u)=-1,\ |u|_{g_+}\le C,\\
 h(u,X)=0,\ |X|_{g_+}=1
 \end{array}\right\}>0.
\]
Indeed, the conditions define a compact set, and a nonzero
vector orthogonal to a timelike vector is spacelike. Since
$h(\nabla\log\rho,\Z/\rho)=\rho^{-1}\Z(\log\rho)=0$,
we can apply this lower bound to obtain
\[
\begin{aligned}
 |\nabla\log\rho|_{g_+}^2
 &\le C h(\nabla\log\rho,\nabla\log\rho)
   =C|d\log\rho|_q^2\le C,\\
 |d\log\rho|_{g_+}
 &=|h(\nabla\log\rho,\cdot)|_{g_+}\le C.
\end{aligned}
\]
The derivative identity above yields
\[
 \left|\nabla\left(\frac{\Z}{\rho}\right)\right|_{g_+}
 \le |h|_{g_+}\left|\frac{\Z}{\rho}\right|_{g_+}
                   |\nabla\log\rho|_{g_+}\le C,
\]
which proves \eqref{action:intrinsic-bound}.
\end{proof}

We apply this lemma to the metric $h$ of
Definition~\ref{def:stationary-metrics}. By
\eqref{local:null-lift}, its null geodesics lift horizontally to
zero-energy spacetime null geodesics. On a compact subset of
$\{\g(\T,\T)>0\}$ this lifting map is smooth, and the
$g_+$-norm of a vector in $TQ$ is uniformly equivalent to the
$\mathbf g_+$-norm of its lift. Parametrize an $h$-null curve
by $g_+$-arclength. In the
stationary product coordinates its horizontal velocity is
\[
 v^i\partial_{x^i}
 -\frac{\g(\T,\partial_{x^i})v^i}{\g(\T,\T)}\T.
\]
Its bounded stationary component gives a lift on every finite
interval. Lemma~\ref{action:escape} consequently supplies $L_K$.
Circularity gives hypersurface orthogonality of the projected
rotational field, and its complete periodic flow supplies the
periodic orbits required above.

\begin{proposition}\label{action:relative}
For every compact $K\Subset\{\g(\T,\T)>0\}\subset Q$ there is
$C\ge1$ such that, along every complete $\Z$-orbit in $U$
whose stationary projection lies in $K$,
\begin{equation}\label{action:relative-est}
\begin{aligned}
\frac{|\Z|_{g_+}}{\rho}
 +\left|\nabla(\Z/\rho)\right|_{g_+}
 +|d\log\rho|_{g_+}+|d\log\rho|_q&\le C,\\
C^{-1}\g(\Z,\Z)&\le W\le C\g(\Z,\Z),\\
|\Z|_{\mathbf g_+}&\le C\rho.
\end{aligned}
\end{equation}
Here $W=\rho^2=-h(\Z,\Z)$ on $Q$, and $q$ is the Riemannian
metric in \eqref{action:reduction}. The constant depends only
on $K$, the stationary geometry, and $g_+,\mathbf g_+$, and is
independent of $U,\Z$. Thus $\Z/\rho$ on these projected orbits
lies in a compact subset of $\{v\in TQ:h(v,v)=-1\}$ over $K$
independent of $U,\Z$.
\end{proposition}
\begin{proof}
Lemma~\ref{action:intrinsic-orbits} gives the quotient estimates,
including the derivative bound. We transfer them to spacetime.
The closed orbit has an extremum of $\vartheta$, where
$\Z\in\ker d\vartheta$, a spacelike plane whose projection
to $TQ$ is an isomorphism. After translating this point to
the fixed stationary section,
uniform metric comparison gives
\[
\begin{aligned}
c|\Z|_{g_+}^2&\le\g(\Z,\Z)\le C|\Z|_{g_+}^2,\\
\rho^2&=-h(\Z,\Z)\le C|\Z|_{g_+}^2,\\
|\Z|_{g_+}^2&\le C\rho^2.
\end{aligned}
\]
The last inequality is the bound just proved, and the middle
one uses only bounded coefficients of the fixed metric $h$.
Thus
\[
\begin{aligned}
C^{-1}\rho^2&\le\g(\Z,\Z)\le C\rho^2,\\
|\g(\T,\Z)|&\le C|\Z|_{g_+}\le C\rho.
\end{aligned}
\]
These contractions and $\rho$ are constant along the orbit.
The $\T$-orthogonal decomposition of $\Z$ then yields
$|\Z|_{\mathbf g_+}\le C(|\Z|_{g_+}+|\g(\T,\Z)|)\le C\rho$ everywhere.
Together with the intrinsic estimate, this proves
\eqref{action:relative-est}. The equalities
$h(\Z/\rho,\Z/\rho)=-1$ and $|\Z/\rho|_{g_+}\le C$ give the
last assertion.
\end{proof}

\subsection{Verification in Kerr}\label{subsec:nontrapping}

We verify NT in subextremal Kerr by separating the zero-energy
null geodesic equations. The resulting radial acceleration is
strictly negative on every compact exterior interval, which lets
us exclude confinement even when the geodesic has radial turning
points.
Write $M>0$ and $a$,
$|a|<M$, for the mass and rotation parameters. When $a=0$,
$\T$ is timelike throughout the exterior, so there are no
nonzero null covectors of zero stationary energy. Suppose $a\ne0$.
In Boyer--Lindquist coordinates $(t,r,\theta,\phi)$,
$\T=\partial_t$. Write
$p=p_r\,dr+p_\theta\,d\theta+p_\phi\,d\phi$ for a nonzero
null covector of zero stationary energy, so $p_t=0$. The inverse Kerr metric gives
\[
 0=(r^2+a^2\cos^2\theta)\g^{-1}(p,p)
   =\Delta p_r^2+p_\theta^2
        +\frac{p_\phi^2}{\sin^2\theta}-\frac{a^2p_\phi^2}{\Delta},
\]
where
\[
 \Delta=r^2-2Mr+a^2.
\]
Here primes on $\Delta$ denote differentiation in $r$.
The rotation axis has $\g(\T,\T)=-\Delta/(r^2+a^2)<0$,
so it contains no such nonzero null covector.
We use the separation found by
Carter~\cite{Carter68}.
For the calculation, multiply the null Hamiltonian
$\tfrac12\g^{-1}(p,p)$ by $r^2+a^2\cos^2\theta$, or
equivalently use the parameter $s$
with $d\lambda/ds=r^2+a^2\cos^2\theta$, where $\lambda$ is
affine. The radial and angular parts separate, so
\[
 \mathcal K=p_\theta^2+\frac{p_\phi^2}{\sin^2\theta}
\]
is constant along the geodesic, as is the momentum $p_\phi$.
All Hamilton equations below use $s$, with $p_\theta,p_r$ still
the canonical momenta of the affine null geodesic. The angular
equations show the cancellation
directly:
\begin{align*}
\frac{d\theta}{ds}&=p_\theta,\\
\frac{dp_\theta}{ds}&=\frac{p_\phi^2\cos\theta}{\sin^3\theta},
\end{align*}
which imply
\[
 \frac{d\mathcal K}{ds}
 =2p_\theta\frac{p_\phi^2\cos\theta}{\sin^3\theta}
      -\frac{2p_\phi^2\cos\theta}{\sin^3\theta}p_\theta=0.
\]
The radial Hamilton equations give
\[
 \begin{aligned}
 \frac{dr}{ds}&=\Delta p_r,\\ 
 \frac{dp_r}{ds}
   &=-\frac{\Delta'}2\left(p_r^2+\frac{a^2p_\phi^2}{\Delta^2}\right),\\
 \left(\frac{dr}{ds}\right)^2&=a^2p_\phi^2-\Delta\mathcal K,\\
 \frac{d^2r}{ds^2}
   &=\frac{\Delta'}2
        \left(\Delta p_r^2-\frac{a^2p_\phi^2}{\Delta}\right)
     =-(r-M)\mathcal K.
 \end{aligned}
\]
Because we derived the acceleration from Hamilton's equations,
the same formula applies at radial turning points. Since $\Delta>0$ for
$r>r_+=M+\sqrt{M^2-a^2}$, a nonzero null covector has
$\mathcal K>0$ and $p_\phi\ne0$.

If the stationary projection of a maximal null geodesic lay in a compact
set, then $r_+<r_1\le r\le r_2$ and
\[
\begin{aligned}
\sin^2\theta&\ge\frac{p_\phi^2}{\mathcal K}>0,\\
|p_\theta|&\le\sqrt{\mathcal K},\\
|p_r|&=\frac{|dr/ds|}{\Delta}\le C.
\end{aligned}
\]
The position modulo stationary translation and the covector
therefore remain in a compact set away from the coordinate axis.
The smooth Hamilton equations, including the stationary-coordinate
equation, continue the geodesic for all $s\in\RR$. But
\[
 \frac{dr}{ds}(s)
 \le\frac{dr}{ds}(0)-(r_1-M)\mathcal K\,s,\qquad s\ge0,
\]
contradicts its bounded radial velocity. This proves NT in Kerr.
\section{A maximal domain of rotational symmetry}
\label{sec:discrete-continuation}
\label{sec:limit-geometry}

We now extend the rotational action as far as possible while
retaining a geometric condition on the region where it is not yet
defined. We require this remaining region to contain every sufficiently
short zero-energy null segment whose endpoints it contains. The
strict Hessian inequality \eqref{local:convex} preserves this
condition when the removed sublevel portion is compactly contained
in its coordinate neighborhood, as proved in
Lemma~\ref{discrete:admissible-update}.
It is also preserved by intersections, which lets us choose a maximal
domain directly.

The main point is to understand the boundary of this domain. On a
disk transverse to the rotation, short null segments become geodesics
of the optical metric. The second fundamental form of an invariant
hypersurface, evaluated on its two null tangent directions, is
determined by one geodesic curvature. We will show that a
boundary curve either gives a strictly pseudoconvex supporting
hypersurface or is itself a geodesic. Maximality excludes the first
possibility. The second produces the zero-energy null geodesics to
which we will apply NT in Section~\ref{sec:remaining}.

\subsection{Supporting hypersurfaces}

Boundaries and interiors of subsets of $Q$ are taken in $Q$,
unless explicitly called relative. Let $H\subset Q$ be closed
and $p\in\partial H$.
A smooth hypersurface supports $H$ at $p$ if it has
a defining function $f$ on a neighborhood $V$ such that
\[
 \begin{aligned}
 f(p)&=0\\
 df(p)&\ne0\\
 H\cap V&\subset\{f\ge0\}.
 \end{aligned}
\]
We apply the extension theorem from the side
$V\cap\{f<0\}\subset Q\setminus H$. Throughout this construction we
work in $\{\g(\T,\T)>0\}$, and we lift defining functions on $Q$
by $\pi$ when applying Proposition~\ref{local:extension}. We choose
the orientation by the side on which the field is already defined.
With this choice, the inequality \eqref{local:convex} prevents a short
zero-energy null segment with projected endpoints in $\{f\ge0\}$
from projecting into $\{f<0\}$. This is the relation between local extension and the
convexity property used below.

\subsection{Convexity along short null geodesics}

Recall from Proposition~\ref{prep:residual} that $C_+$ is the
initial open neighborhood carrying the rotational field,
$H_0=\{\g(\T,\T)\ge0\}\setminus C_+$, and
$\mathcal D_0=\{\g(\T,\T)\ge0\}\setminus C_-$, where
$\overline{C_-}\subset C_+$. We use the stationary auxiliary
Riemannian metric $\mathbf g_+$ on $\E$ and the fixed metric $g_+$
on $Q$. All remaining sets will be compact subsets of $H_0$.
We remove points only where $\g(\T,\T)>0$.
The circle action will therefore be constructed throughout the
ergoregion before we complete the extension at the ergosurface.

\begin{definition}\label{discrete:null-convexity}
Fix $\epsilon_0>0$. We impose the following condition on a
closed set $H\subset\mathcal D_0$: for every null geodesic segment
$\gamma:[a,b]\to\E$ satisfying
\[
 \begin{aligned}
 \g(\T,\dot\gamma)&=0\\
 \int_a^b|\dot\gamma|_{\mathbf g_+}\,ds&<\epsilon_0
 \end{aligned}
\]
we have
\[
 \pi(\gamma(a)),\pi(\gamma(b))\in H
 \quad\Longrightarrow\quad
 \pi(\gamma([a,b]))\subset H.
\]
\end{definition}

We measure these segments with the fixed Riemannian metric
$\mathbf g_+$ on spacetime, which remains smooth at $\g(\T,\T)=0$. In the region $\{\g(\T,\T)>0\}$ we can also work with the Lorentzian quotient metric $h$
of Definition~\ref{def:stationary-metrics}: on each compact subset,
stationary horizontal lifting gives a uniform comparison with the
auxiliary length of the corresponding $h$-null segments.

On a neighborhood carrying a hypersurface-orthogonal Killing
field $\Z$ which is timelike for $h$ and preserves $H$, this
condition has an ordinary convexity interpretation.
In the local product coordinates of Lemma~\ref{local:optical-curvature},
let $I_\phi$ be the coordinate interval along $\Z$ and let $C$
be the intersection of $H$ with the transverse surface. Invariance
then gives $H=I_\phi\times C$ within this chart. If $c$ is a unit-speed geodesic
of $q_{\mathrm{opt}}$,
then \eqref{local:optical-geodesics} gives the null lift
$s\mapsto(\phi_0+s,c(s))$; conversely, every product null geodesic
projects to a geodesic of $q_{\mathrm{opt}}$. For sufficiently short segments
in this chart, the condition on $H$ is therefore equivalent to local
geodesic convexity of $C$ for $q_{\mathrm{opt}}$. The spacetime formulation above lets us retain
the condition before such a rotational neighborhood has been
constructed at the boundary.

\begin{lemma}\label{discrete:null-convex-cuts}
The set $H_0$ has the property in
Definition~\ref{discrete:null-convexity} for some $\epsilon_0>0$.
For this fixed $\epsilon_0$, the property is preserved by intersections
and by the following operation. Let $H$ have this property,
let $V\Subset Q$ be open, and let $C\subset V$ be relatively closed. Suppose every zero-energy null
segment contained over $V$, of $\mathbf g_+$-length less than $\epsilon_0$
and with projected endpoints in $C$, projects into $C$. If
\[
 \overline{H\cap(V\setminus C)}\Subset V,
\]
then $H'=H\setminus(V\setminus C)$ has the same property.
\end{lemma}
\begin{proof}
We first verify the property for the initial set. Near its relative
boundary we use the Gaussian coordinate $r$ of
Proposition~\ref{prep:residual}, whose Hessian satisfies \eqref{local:convex}
on a neighborhood of the compact annulus $r=\delta_+$ and its
nearby levels. If a zero-energy null segment
with both projected endpoints in $H_0$ enters $\{r<\delta_+\}$ and
remains in this neighborhood, then at an interior minimum below
$\delta_+$ we have
\begin{equation}\label{discrete:null-minimum}
 \begin{aligned}
 (r\circ\pi\circ\gamma)'&=0,\\
 0\le(r\circ\pi\circ\gamma)''
   &=\D^2(r\circ\pi)(\dot\gamma,\dot\gamma)<0.
 \end{aligned}
\end{equation}
This contradiction uses an affine parameter.
The sign of the second derivative at a critical point is unchanged
by reparametrization.
To make the length restriction uniform, let $A$ be the compact
initial annulus and choose a relatively compact neighborhood $V_0$ of
$A$ on which $r$ is defined and satisfies \eqref{local:convex}.
Stationarity and compactness give $L<\infty$ with
\[
 |d\pi(X)|_{g_+}\le L|X|_{\mathbf g_+}
 \quad\text{over }\overline{V_0}.
\]
Choose $\epsilon_0>0$ so that
$L\epsilon_0<\operatorname{dist}_{g_+}(A,Q\setminus V_0)$.
If the projected segment meets $A$, its first exit from $V_0$ would give
\[
 \operatorname{dist}_{g_+}(A,Q\setminus V_0)
 \le\int|d\pi(\dot\gamma)|_{g_+}\,ds
 <L\epsilon_0,
\]
a contradiction. The preceding minimum argument therefore applies to
every segment of the required length that crosses the initial annulus.
A nonzero null vector orthogonal to $\T$ cannot occur where $\T$ is
timelike. Thus such a segment stays in $\{\g(\T,\T)\ge0\}$,
and crossing the annulus is the only way it could leave $H_0$.

We next check that the removal preserves the condition of
Definition~\ref{discrete:null-convexity}. Let $\gamma$
have projected endpoints in $H'$, so that its entire projection lies
in $H$. If its projection enters $V\setminus C$, let $(a_0,b_0)$
be a component of $(\pi\circ\gamma)^{-1}(V\setminus C)$.
Compact containment of the removed part excludes an endpoint on
$\partial V$, and consequently
\[
 \begin{aligned}
 \pi(\gamma(a_0)),\pi(\gamma(b_0))&\in C\\
 \pi(\gamma([a_0,b_0]))&\subset V.
 \end{aligned}
\]
The subsegment has length less than $\epsilon_0$. Its assumed convexity
contradicts $\pi(\gamma((a_0,b_0)))\subset V\setminus C$.
An intersection has the property by applying its defining implication
to each of the intersected sets.
\end{proof}

For a compact $H\subset H_0$ define the open stationary set
\begin{equation}\label{discrete:exact-domain}
 U_H=\pi^{-1}(C_+)\ \cup\
 \pi^{-1}\bigl(\{\g(\T,\T)>0\}\cap(\mathcal D_0\setminus H)\bigr).
\end{equation}
It is open because $H\subset H_0$ and a fixed neighborhood of
$\partial\mathcal D_0\cap\{\g(\T,\T)>0\}$ is contained in $C_+$.
We retain $\Z$ on all of $U_H$. Thus, when we use a boundary
neighborhood to remove part of $H$, we require its extension to agree
with $\Z$ wherever the two fields are defined. Formula
\eqref{discrete:exact-domain} then determines the enlarged domain
directly from the new remaining set. In particular,
\[
 \partial H\cap\{\g(\T,\T)>0\}\subset\partial\pi(U_H).
\]

\begin{definition}\label{discrete:admissible}
A pair $(H,\Z)$ is required to satisfy the following conditions.
The set $H\subset H_0$ is compact, $U_H$ in
\eqref{discrete:exact-domain} is connected, and $\Z$ is a Killing
field on $U_H$ whose flow is complete with effective period $2\pi$.
It agrees on $\pi^{-1}(C_+)$ with the field of
Proposition~\ref{prep:residual} and satisfies
\[
 [\T,\Z]=0.
\]
We require that $H$ satisfy
Definition~\ref{discrete:null-convexity}, with the fixed
$\epsilon_0$ of Lemma~\ref{discrete:null-convex-cuts}, and that
\[
 H_0\cap E_0\subset H,
\]
where $E_0=\{\g(\T,\T)=0\}\subset Q$ is the projected ergosurface.
\end{definition}

The initial pair is supplied by Proposition~\ref{prep:residual}
and Lemma~\ref{discrete:null-convex-cuts}. To enlarge its domain,
we need a local extension which agrees with the old field on
their entire overlap and has a complete rotational action.
Lemma~\ref{local:complete-flows} then identifies the actions on
that overlap. With an invariant defining function, the removed
region is invariant as well, and the following lemma verifies
that the enlarged pair has all the required properties.

\begin{lemma}\label{discrete:admissible-update}
Let $(H,\Z)$ satisfy Definition~\ref{discrete:admissible} and let
$V\Subset\operatorname{Int}(\mathcal D_0)\cap\{\g(\T,\T)>0\}$
be open. Suppose a Killing field $\widetilde\Z$ on $\pi^{-1}(V)$
commutes with $\T$, has complete period-$2\pi$ flow there, and
\[
 \widetilde\Z=\Z\quad\text{on }\pi^{-1}(V)\cap U_H.
\]
Let $f\in C^\infty(V)$ satisfy $\widetilde\Z f=0$ on $V$,
where the field is projected to $Q$. Suppose its stationary
pullback satisfies $\g^{-1}(\pi^{\ast}df,\pi^{\ast}df)>0$ and
\[
 \D^2(f\circ\pi)(X,X)<0
 \quad\bigl(X\ne0,\ X(f\circ\pi)=\g(X,X)=\g(\T,X)=0\bigr)
\]
for $X\in T\pi^{-1}(V)$.
For $\varepsilon\in\RR$, assume
\[
 \overline{H\cap V\cap\{f<\varepsilon\}}\Subset V
\]
and that every component of $V\cap\{f<\varepsilon\}$ meets
$\pi(U_H)$. Then
\[
 H'=H\setminus\bigl(V\cap\{f<\varepsilon\}\bigr)
\]
and the glued field on $U_{H'}$ satisfy Definition~\ref{discrete:admissible}.
Under the same assumptions on $V,\widetilde\Z$, the conclusion
also holds for $H'=H\setminus V$ when $H\cap\partial V=\varnothing$
and every component of $V$ meets $\pi(U_H)$.
\end{lemma}
\begin{proof}
Let $A=V\cap\{f<\varepsilon\}$. It is open, and
$H'=H\setminus A$ is compact. Since $\overline V$ lies in
$\{\g(\T,\T)>0\}$, we retain $H_0\cap E_0\subset H'$.
For a zero-energy null segment contained over $V$, with projected
endpoints in $\{f\ge\varepsilon\}$, an interior minimum below
$\varepsilon$ contradicts \eqref{discrete:null-minimum}, applied with
$r$ replaced by $f$. Thus every such segment projects into
$\{f\ge\varepsilon\}$. Lemma~\ref{discrete:null-convex-cuts}
and the compact-containment hypothesis give
the condition of Definition~\ref{discrete:null-convexity} for $H'$.

Write $\Phi_t$ and $\widetilde\Phi_t$ for the complete flows
on $U_H$ and $\pi^{-1}(V)$. By Lemma~\ref{local:complete-flows},
\begin{align*}
 \widetilde\Phi_t(U_H\cap\pi^{-1}(V))
     &=U_H\cap\pi^{-1}(V),\\
 \widetilde\Phi_t(\pi^{-1}(H\cap V))
     &=\pi^{-1}(H\cap V).
\end{align*}
The second equality follows by taking the complement inside
$\pi^{-1}(V)$. Since
\[
 \frac d{dt}(f\circ\pi\circ\widetilde\Phi_t)
       =(\widetilde\Z f)\circ\pi\circ\widetilde\Phi_t=0,
\]
the set $\pi^{-1}(A)$ is also invariant. Consequently the fields
glue on the enlarged domain
\[
 U_{H'}=U_H\cup\pi^{-1}(A).
\]
Each component of $A$ meets $\pi(U_H)$, so this domain is
connected. On $U_H$ its flow is $\Phi_t$, and on $\pi^{-1}(A)$
it is $\widetilde\Phi_t$. The overlap identity therefore gives
a complete flow satisfying
\[
 \Phi_{2\pi}^{\Z}=\operatorname{Id}_{U_{H'}}.
\]
A smaller global period would restrict to a smaller period on
$\pi^{-1}(C_+)$, contrary to the initial normalization.
The glued field is Killing, commutes with $\T$, and agrees
with the initial field, proving all the required conditions.

For $H'=H\setminus V$, compactness and the inclusion
$H_0\cap E_0\subset H'$ follow in the same way. A null segment with projected
endpoints in $H'$ already projects into $H$. If its projection
entered $V$, it would meet $H\cap\partial V$, contrary to the
hypothesis. Its projection therefore stays in $H'$. Now take $A=V$ in the preceding flow
and domain argument.
\end{proof}

\subsection{The maximal domain and its boundary}

We order the pairs of Definition~\ref{discrete:admissible} by extension:
\[
 (H_1,\Z_1)\preceq(H_2,\Z_2)
 \quad\Longleftrightarrow\quad
 H_2\subset H_1\ \text{ and }\ 
 \Z_2|_{U_{H_1}}=\Z_1.
\]
Thus the order records agreement of the fields, in addition
to inclusion of their domains.

\begin{lemma}\label{discrete:maximal-action}
The pairs of Definition~\ref{discrete:admissible} have a maximal
element under $\preceq$.
\end{lemma}
\begin{proof}
The initial pair exists by Proposition~\ref{prep:residual} and
Lemma~\ref{discrete:null-convex-cuts}. For a nonempty chain
$\mathcal C$, put
\[
 H_{\mathcal C}=\bigcap_{(H',\Z')\in\mathcal C}H'.
\]
This compact set contains $H_0\cap E_0$ and retains the condition
of Definition~\ref{discrete:null-convexity}.
Formula~\eqref{discrete:exact-domain} gives
\[
 U_{H_{\mathcal C}}=\bigcup_{(H',\Z')\in\mathcal C}U_{H'}.
\]
The union is connected, since its members are nested and contain
$\pi^{-1}(C_+)$. The order identifies the fields on their overlaps,
so they define a smooth commuting Killing field $\Z_{\mathcal C}$
on this union. For $p\in U_{H'}$, its entire old orbit stays in
$U_{H'}$, and uniqueness of integral curves gives
\[
 \Phi_t^{\Z_{\mathcal C}}(p)=\Phi_t^{\Z'}(p)
 \qquad(t\in\RR).
\]
Consequently the glued flow is complete and has period $2\pi$.
A smaller effective period would restrict to a smaller period on
$\pi^{-1}(C_+)$, contrary to the initial normalization. Thus
$(H_{\mathcal C},\Z_{\mathcal C})$ is an upper bound for the chain.
Zorn's lemma gives a maximal pair.
\end{proof}

Fix a maximal pair $(H,\Z)$ for the rest of the proof. In particular,
$H\subset H_0$ is compact, it is closed under the short null segments
of Definition~\ref{discrete:null-convexity}, and the rotation
has complete effective period $2\pi$ on $U_H$. We use $\nabla$ for
the connection of the Lorentzian quotient metric $h$ on the open
ergoregion and $g_+$ for the fixed auxiliary Riemannian metric on $Q$.
Spacetime fields and their derivatives are evaluated on the stationary
section $\vartheta=0$ when we regard their components as functions on
$Q$.

\begin{proposition}\label{discrete:stalled-boundary}
For the maximal pair of Lemma~\ref{discrete:maximal-action}, every
point of
\[
 \partial H\cap\{\g(\T,\T)>0\}
\]
has a neighborhood in which this boundary is a smooth timelike
hypersurface of $(Q,h)$. For either unit spacelike normal $n$, its
second fundamental form $B(X,Y)=-h(\nabla_Xn,Y)$ satisfies
\[
 B(X,X)=0
 \qquad(X\in T\partial H,\ h(X,X)=0).
\]
Every smooth null curve with nowhere zero tangent in this hypersurface is an $h$-geodesic
up to reparametrization. Locally the remaining set is either one side
of the hypersurface or the hypersurface itself.
\end{proposition}

Thus this boundary is locally a photon surface of $(Q,h)$:
every null geodesic initially tangent to it stays in it locally.
Claudel, Virbhadra, and Ellis~\cite[Definition~2.1]{CVE01}
introduced this local notion for nowhere-spacelike hypersurfaces.
In Section~\ref{sec:remaining} we lift the null geodesics
of this quotient surface to spacetime null geodesics orthogonal to
$\T$ and continue them at the ergosurface.

The difficulty is that $\Z$ is initially defined only outside $H$.
To use the optical metric at the boundary, we first obtain boundary
values of $\Z$ and $\D\Z$ from exterior normal segments. We must also
compare these values from distinct exterior components: a local
extension is useful only if it agrees with the field on their entire
overlap. All estimates below are local on compact subsets of the
open ergoregion.

\subsection{Exterior normal segments and Killing transport}

\begin{definition}[Reach]\label{def:reach}
For a closed set $H\subset Q$, its reach in $g_+$ is the
supremum of $R\ge0$ for which every $x$ with
$d_{g_+}(x,H)<R$ has a unique nearest point in $H$.
Positive reach near $q\in\partial H$ means that this property
holds for all such $x$ in a neighborhood of $q$, for some $R>0$.
\end{definition}

This is the Riemannian version of the reach introduced by
Federer~\cite{Federer59}. Since points of $H$ are their own nearest
points, we use this condition on the exterior. In particular, a minimizing normal
segment can be continued with the same nearest point up to any
radius smaller than the reach and the injectivity radius, as long
as it stays in the neighborhood where nearest points are unique.
To see this, put
$d=d_{g_+}(\,\cdot\,,H)$. The nearest point depends continuously
on the exterior point. The first variation of distance therefore
gives $d\in C^1$ on the exterior tube and $|\nabla^{g_+}d|_+=1$.
An integral curve of $\nabla^{g_+}d$ satisfies
\[
 \operatorname{Length}_{g_+}(\gamma|_{[s,t]})
 =d(\gamma(t))-d(\gamma(s))
 \le d_{g_+}(\gamma(t),\gamma(s)).
\]
Equality of length and distance makes it a minimizing geodesic.
Its initial point and velocity determine that geodesic, so an
exterior normal segment continues with the same nearest point
throughout the tube. This is the local Riemannian form of the
argument of Federer~\cite[Theorem~4.8(3)--(6)]{Federer59}.

\begin{definition}[Proximal normals]\label{def:proximal-conormal}
Let $H\subset Q$ be closed and $q\in\partial H$.
We use the proximal normal cone of
Bernicot and Venel~\cite[Definition~3.6]{BV15} and denote its
$g_+$-metric dual by $N_H^P(q)$. Thus $N_H^P(q)$ consists of the covectors $\alpha\in T_q^{\ast}Q$ for which there is $s_0>0$ such that
\[
 d_{g_+}\bigl(\exp_q^{g_+}(s\alpha^{\sharp_{g_+}}),H\bigr)
       =s|\alpha|_+\qquad(0<s<s_0).
\]
Here $\alpha^{\sharp_{g_+}}$ is the metric dual of $\alpha$ in $g_+$.
\end{definition}

We shall use a quadratic characterization of these normal covectors.
The Euclidean normal inequality and normal-ray description are due to
Federer~\cite[Theorem~4.8(7) and~(12)]{Federer59}; the corresponding
Riemannian characterization is given by
Bernicot and Venel~\cite[Remark~3.11 and Proposition~3.18]{BV15}.
We give the local calculation, including its uniform constants,
because we will change the auxiliary Riemannian metric below. If $H$ has positive reach, the
normal-segment argument lets us follow every unit
$\alpha\in N_H^P(q)$ up to a fixed radius smaller than the reach
and the injectivity radius.

Fix $R$ below the common reach and injectivity radii, and put
$c=\exp_q^{g_+}(R\alpha^{\sharp_{g_+}})$. In relatively compact
coordinates, with Euclidean norm $|\cdot|$, first variation gives
\[
 \begin{aligned}
 0&\le d_{g_+}(c,x)^2-d_{g_+}(c,q)^2\\
  &=-2R\alpha(x-q)
   +\int_0^1(1-s)\partial_{ij}d_{g_+}(c,q+s(x-q))^2
                         (x-q)^i(x-q)^j\,ds\\
  &\le-2R\alpha(x-q)+C_R|x-q|^2\qquad(x\in H).
 \end{aligned}
\]
The segment is restricted to a coordinate ball avoiding the cut
locus of $c$. Consequently
\begin{equation}\label{normal:conormal-inequality}
 \alpha(x-q)\le C|x-q|^2
 \qquad(x\in H\text{ near }q,\quad
        \alpha\in N_H^P(q),\quad|\alpha|_+=1).
\end{equation}
The constant is uniform on compact sets with a common reach
radius. Conversely, use $g_+$-normal coordinates $\xi$ at $q$.
For $|\alpha|_+=1$ and small $s>0$,
\[
 d_{g_+}(\exp_q^{g_+}(s\alpha^{\sharp_{g_+}}),x)^2-s^2 =-2s\alpha(\xi)+|\xi|^2+O(s|\xi|^2)
 \ge[1-(2C+C_0)s]|\xi|^2.
 \]
The assumed quadratic bound gives the last inequality. For
$s<(2C+C_0)^{-1}$, the point $q$ is therefore nearest among the
points of $H$ in this coordinate ball. Points outside the ball stay
farther away by the triangle inequality after reducing $s$, which
proves the nearest-point property. We can repeat this expansion for
any fixed smooth Riemannian metric.
Since in each case the covectors in $N_H^P(q)$ are characterized by the
same local quadratic inequality, the subset $N_H^P(q)\subset T_q^\ast Q$ is independent of the
metric. Moreover, when the bound is uniform on unit covectors, norm
equivalence and the converse estimate extend all unit normal rays
to one fixed radius $R>0$, reduced if necessary for the new metric.
We recover uniqueness of nearest points from this extension property.
Indeed, if $p,q$ are nearest to $x$ at distance
$s<R$, extend the minimizing segment from $p$ through $x$ to $c$
at distance $R$ from $p$. Then
\[
 R=d(c,H)\le d(c,q)\le d(c,x)+d(x,q)=(R-s)+s=R.
\]
Equality makes the broken segment from $c$ through $x$ to $q$
a minimizing geodesic. Its continuation from $x$ is unique, so
$q=p$. Here all segments are chosen below the injectivity radius
and inside the fixed coordinate neighborhood.
For normal covectors $\alpha,\beta$ and $a,b\ge0$, the same inequalities give
\[
 (a\alpha+b\beta)(x-q)
 \le (aC_\alpha+bC_\beta)|x-q|^2.
\]
The converse therefore puts $a\alpha+b\beta$ in $N_H^P(q)$.
For unit normal covectors at varying base points the constant is uniform,
so passage to the limit in \eqref{normal:conormal-inequality}
also proves closedness.
Every boundary point has a nonzero normal covector: project exterior
points tending to it onto $H$, and take a convergent subsequence
of the resulting unit normal covectors.

We next obtain positive reach before extending $\Z$ to $H$.
The input is closure under short null geodesics and a complete
timelike exterior flow with a bounded first derivative.

\begin{lemma}\label{closure:null-reach}
Let $K\subset O\Subset\{\g(\T,\T)>0\}$, where $K$ is compact
and $O\subset Q$ is open, and let $H\subset Q$ be closed.
For some $r>0$, assume that every $h$-null geodesic segment
contained in $O$, of $g_+$-length less than $r$, with endpoints
in $H$ lies in $H$.
Suppose an open set $U\subset\{\g(\T,\T)>0\}\setminus H$
contains $O\setminus H$ and carries a complete smooth field $V$
such that, for some $C<\infty$, on $O\setminus H$,
\[
 \begin{aligned}
 h(V,V)&=-1,\\
 |V|_{g_+}+|\nabla V|_{g_+}&\le C.
 \end{aligned}
\]
Then $H$ has a uniform positive reach radius near $K$, depending
only on $K,O,r,C$ and the fixed metrics.
\end{lemma}
\begin{proof}
Suppose that an exterior point has two nearest points in $H$. We first
extend the exterior flow to these points by taking limits from the
ball, then move them in opposite timelike directions until the
connecting geodesic becomes null. The first-order displacements
cancel at its midpoint, while their difference changes the causal
character of the connecting geodesic. The remaining error will
be smaller than the difference between the ball radius and the
distance of its chord midpoint from the center.

Choose a fixed smooth $h$-unit timelike field $V_0$ locally and put
$\widehat g_+=h+2h(V_0,\,\cdot\,)\otimes h(V_0,\,\cdot\,)$.
It suffices to prove
a reach bound in this metric, by the change-of-metric calculation following
Definition~\ref{def:proximal-conormal}.
The field $V_0$ and all coordinate constants below depend only on
the fixed geometry.
Suppose an exterior point $c$ has distinct nearest points $p,q$ at
the same sufficiently small distance $R$. In $\widehat g_+$-normal coordinates
centered at $c$, choose the initial basis with first vector $V_0(c)$.
Then $\widehat g_+(c)=I$, $h(c)=\operatorname{diag}(-1,1,1)$, and the
coordinate ball $B_R(c)$ is the geodesic ball disjoint from $H$.
In this proof $|\cdot|$ and $\cdot$ denote the Euclidean coordinate
norm and inner product, and $d=|q-p|$.
The open coordinate chord from $p$ to $q$ lies in $B_R(c)$,
where there is one smooth exterior field. The bound on $\nabla V$
and the fixed connection coefficients give $|\partial V|\le C$ there.
The boundary values are its limits from this ball; hence
\[
 \begin{aligned}
 V(q)-V(p)&=\int_0^1\partial V(p+s(q-p))(q-p)\,ds,\\
 |V(q)-V(p)|&\le Cd.
 \end{aligned}
\]
All values $V(p),V(q)$ in this proof denote limits from the
same ball $B_R(c)$. Write $\Phi_t$ for the complete flow on $U$.
Choose $K\subset O_1\Subset O$ and then $t_0>0$ with
\[
 Ct_0<\tfrac12\operatorname{dist}_{g_+}(\overline{O_1},Q\setminus O).
\]
The speed bound keeps every trajectory starting in $O_1\setminus H$
inside $O$ for $|t|\le t_0$. Restrict all the small balls to $O_1$.
In coordinates, the variational equation and Gronwall's inequality
give
\begin{align*}
 \frac d{dt}d\Phi_t&=(\partial V)_{\Phi_t}d\Phi_t,\\
 |d\Phi_t|&\le C e^{C|t|}.
\end{align*}
For $x,y$ in the ball, apply this bound along their coordinate
chord, whose image remains in the exterior:
\begin{align*}
 d_{g_+}(\Phi_t(x),\Phi_t(y))
 &\le\int_0^1\bigl|(d\Phi_t)_{x+s(y-x)}(y-x)\bigr|_{g_+}\,ds\\
 &\le C e^{C|t|}|x-y|.
\end{align*}
Thus $\Phi_t$ has a continuous limit from the ball at $p$ and $q$,
uniformly for these times. Denote the resulting curves by $p(t)$
and $q(t)$. They lie in $\partial H$. Indeed they are limits of
exterior points, and if, for example, $p(t)$ were exterior, the
inverse flow would give
\[
 p=\lim_{x\to p,\ x\in B_R(c)}\Phi_{-t}(\Phi_t(x))
   =\Phi_{-t}(p(t))\notin H.
\]
The coordinate acceleration of every exterior trajectory satisfies
\begin{align*}
 \frac{d^2}{dt^2}\Phi_t^i(x)
   &=\partial_jV^i(\Phi_t(x))V^j(\Phi_t(x)),\\
 \left|\frac{d^2}{dt^2}\Phi_t(x)\right|&\le C.
\end{align*}
Taking limits in its Taylor formula proves that $p'(0)=V(p)$
and $q'(0)=V(q)$, with uniform quadratic remainders.
The distance to $c$ has a minimum along these two-sided boundary
curves. Gauss' lemma therefore gives
\[
 V(p)\cdot(c-p)=V(q)\cdot(c-q)=0,
\]
so
\[
 |(q-p)\cdot V(p)|
 =|(c-q)\cdot(V(q)-V(p))|\le CRd.
\]
Write $V(p)=(v^0,v')$ in these coordinates. The hypotheses and
$h(p)-h(c)=O(R)$ imply
$(v^0)^2-|v'|^2\ge1/2$ and $|V(p)|\le C$ for small $R$.
If $w\cdot V(p)=0$, then $w^0=-w'\cdot v'/v^0$ and
\[
 h(c)(w,w)\ge\frac{|w'|^2}{2(v^0)^2}\ge c_0|w|^2.
\]
Apply this inequality to
\[
 w=q-p-\frac{(q-p)\cdot V(p)}{|V(p)|^2}V(p).
\]
The bound $|(q-p)\cdot V(p)|\le CRd$ gives
\begin{align*}
 |w-(q-p)|&\le CRd,\\
 h(c)(q-p,q-p)&\ge c_0|w|^2-CRd^2
                       \ge\tfrac12c_0d^2
\end{align*}
after reducing $R$.
The boundary curves just constructed satisfy
\[
 \begin{aligned}
 p(-t)&=p-tV(p)+O(t^2)\\
 q(t)&=q+tV(q)+O(t^2).
 \end{aligned}
\]
Work inside one convex-normal neighborhood for $h$. For
$0\le t\le Md$, let $\gamma_t:[0,1]\to Q$ be the geodesic from
$p(-t)$ to $q(t)$, with affine parameter on $[0,1]$.
To estimate the connecting geodesics, integrate
their coordinate equation twice with the two endpoint values fixed:
\[
 \begin{aligned}
 &\gamma_t^i(s)-(1-s)p(-t)^i-sq(t)^i\\
 &\quad=\int_0^1\min\{s,u\}\bigl(1-\max\{s,u\}\bigr)
       \Gamma^i_{kl}(\gamma_t(u))\dot\gamma_t^k(u)
                                      \dot\gamma_t^l(u)\,du.
 \end{aligned}
\]
Here $\Gamma^i_{kl}$ are the Christoffel symbols of $h$ in these
coordinates. The endpoint distance is $O_M(d)$; smoothness of the exponential
map and its inverse gives $|\dot\gamma_t|\le C_Md$.
The integral and its first $s$-derivative are consequently
$O_M(d^2)$. Together with the orbit expansions, this yields
\[
 \begin{aligned}
 \dot\gamma_t&=q-p+t\bigl(V(p)+V(q)\bigr)+O_M(d^2),\\
 h(\dot\gamma_t,\dot\gamma_t)
 &=h(c)\bigl(q-p+t(V(p)+V(q)),q-p+t(V(p)+V(q))\bigr)
       +O_M(Rd^2).
 \end{aligned}
\]
For the second equality, $h-h(c)=O_M(R)$, and $d\le2R$
absorbs the velocity error $O_M(d^3)$.
At $t=0$ the constant squared speed is at least $c_0d^2/4$.
Since $V(p)+V(q)=2V(p)+O(d)$, we have
$h(c)(V(p)+V(q),V(p)+V(q))\le-2$ for small $R$. Hence
\[
 h(\dot\gamma_{Md},\dot\gamma_{Md})
 \le \bigl(C+CM-2M^2+C_MR\bigr)d^2<0
\]
if $M$ is fixed sufficiently large and then $R$ sufficiently small.
Choose $R$ also so that $2MR<t_0$, hence $Md<t_0$.
To exclude a constant connecting geodesic, let $P$ be the
Euclidean orthogonal projection onto $V(p)^\perp$. The vector
$w$ above equals $P(q-p)$, so, uniformly for $0\le t\le Md$,
\begin{align*}
 P(q(t)-p(-t))
   &=w+tP(V(q)-V(p))+O(t^2)=w+O_M(d^2),\\
 |P(q(t)-p(-t))|
   &\ge (1-CR)d-C_Md^2\ge d/2>0
\end{align*}
after reducing $R$, since $d\le2R$.
Continuity therefore gives a nonconstant null connecting
geodesic for some $0<t_{\ast}<Md$.
Its endpoints belong to $H$, it stays in the prescribed larger
neighborhood, and its auxiliary length is at most $C_Md<r$.

We now compare the midpoint of this null geodesic with the original
chord midpoint. Write $z$ for its affine midpoint and $m=(p+q)/2$.
Using the integral formula at $s=1/2$ and the two orbit expansions,
we obtain
\[
 \begin{aligned}
 z-m&=\tfrac12t_{\ast}(V(q)-V(p))+O(t_{\ast}^2+d^2),\\
 |z-m|&\le C(t_{\ast}d+t_{\ast}^2+d^2)\le C_Md^2.
 \end{aligned}
\]
But
\[
 \begin{aligned}
 |m-c|^2&=R^2-d^2/4,\\
 |z-c|^2
 &\le |m-c|^2+2|m-c||z-m|+|z-m|^2\\
 &\le R^2-\tfrac14d^2+2RC_Md^2+C_M^2d^4\\
 &\le R^2-\bigl(\tfrac14-2C_MR-4C_M^2R^2\bigr)d^2
 \le R^2-\tfrac18d^2<R^2.
 \end{aligned}
\]
Here we used $d\le2R$ and reduced $R$ so that
$2C_MR+4C_M^2R^2\le1/8$, after fixing $M$.
No lower bound for $d/R$ is needed: both the chord deficit and
the error are compared at order $d^2$.
The midpoint lies in $B_R(c)$, contrary to the required inclusion
of the null segment in $H$. All sufficiently near exterior points
therefore have a unique nearest point, with a uniform radius.
\end{proof}

On a domain carrying the circle action, put $V=\Z/\rho$,
where $\rho^2=-h(\Z,\Z)$. Proposition~\ref{action:relative} gives
exactly the bounds needed above:
\[
 \begin{aligned}
 h(V,V)&=-1,\\
 |V|_{g_+}+|\nabla V|_{g_+}&\le C.
 \end{aligned}
\]
The constants are uniform for complete orbits confined to a fixed
compact subset of the ergoregion.
We apply the lemma to the maximal remaining set $H$. Fix a chart near its
boundary in the ergoregion and choose $\delta>0$ below
$\g(\T,\T)$ on a larger chart. The complete action on
$U_{H}$ preserves $\g(\T,\T)$ and the initial region
$C_-$. Flow uniqueness prevents an orbit starting outside $C_-$
from crossing its invariant boundary. Every orbit under consideration
is therefore confined to the fixed compact set
\[
 \mathcal D_0\cap\{\g(\T,\T)\ge\delta\}.
\]
Proposition~\ref{action:relative} gives the displayed bounds on
the exterior of $H$. Since $\Z\rho=0$, the normalized flow
is complete on each orbit as well. The open set
$\pi(U_{H})\cap\{\g(\T,\T)>0\}$ is invariant under
this complete normalized flow and contains the local exterior.
The condition of Definition~\ref{discrete:null-convexity} is part
of the choice of $(H,\Z)$, and horizontal lifting
compares the quotient and spacetime lengths on these fixed charts.
Thus $H$ has locally uniform positive reach in the ergoregion.

We can now pass from geometry to the Killing transport equations.
Every nearby exterior point lies on a normal segment reaching a
fixed compact subset of the exterior, so parallel transport along
these segments bounds $(\Z,\D\Z)$ up to the boundary. The construction
is illustrated in Figure~\ref{fig:null-reach-transport}. To identify
the limits from different directions, we connect the normal rays
by exterior paths whose lengths tend to zero. This works whenever
the cone at the boundary point is not a line. If it is a line, its two opposite rays
must be treated separately because their normal segments may lie in
different exterior components.

\begin{lemma}\label{normal:parallel-boundary}
Let $H\subset Q$ be closed with positive reach near
$q\in\partial H$. Let $\mathcal W\subset Q$ be an open neighborhood
of $q$, put $U=\mathcal W\setminus H$, and let $\DD$ be a smooth
connection on a vector bundle over $\mathcal W$.
There exist a neighborhood $V\Subset \mathcal W$ of $q$ and a compact set
$K\Subset U$ such that every smooth section $\mathcal J$ on $U$ with
$\DD \mathcal J=0$ satisfies, in a fixed local bundle frame,
\[
 \|\mathcal J\|_{C^m(U\cap V)}\le C_m\sup_K|\mathcal J|\qquad(m\ge0),
\]
where the norms use the chosen coordinates, bundle frame and a fixed
positive bundle metric, and $C_m$ is independent of $\mathcal J$.
If $N_H^P(q)$ is not a line, the limit $\mathcal J(q)=\lim_{U\ni x\to q}\mathcal J(x)$
exists. These limits are continuous at boundary points with
this property. If a smooth parallel section $\widetilde{\mathcal J}$
near $q$ satisfies $\widetilde{\mathcal J}(q)=\mathcal J(q)$, then
$\widetilde{\mathcal J}=\mathcal J$ on $U\cap V$ after shrinking $V$.
\end{lemma}
\begin{proof}
We first use the normal segments to transport the section from
points at a fixed distance outside $H$. Fix a smooth positive bundle metric for
$|\mathcal J|$ and a local bundle frame in a coordinate neighborhood with compact closure
in $\mathcal W$, where the nearest-point property holds. Take $R>0$
below the local reach and injectivity radii. After reducing the
boundary neighborhood, all its normal segments of length $R$
stay in this coordinate neighborhood. For $p\in\partial H$
and a unit $\alpha\in N_H^P(p)$, put
\[
 x_R=\exp_p^{g_+}(R\alpha^{\sharp_{g_+}}).
\]
As $p$ ranges over a smaller closed neighborhood, the unit normal
directions form a closed set, so the corresponding endpoints $x_R$
form a compact set $K$. The preceding extension of normal segments
gives $d_{g_+}(x_R,H)=R$, so $K$ is compactly contained in $U$.
For $x=\exp_p^{g_+}(s\alpha^{\sharp_{g_+}})$, $0<s<R$,
$\DD$-parallel transport $P$ along this segment gives
\begin{equation}\label{normal:radial-transport}
\begin{aligned}
\mathcal J(x)&=P_{x_R\to x}\mathcal J(x_R),\\
|\mathcal J(x)|&\le e^{CR}\sup_K|\mathcal J|,
\end{aligned}
\end{equation}
where the bound follows from \eqref{local:parallel-transport-bound}.
Every sufficiently near exterior point has this form.
Successive coordinate derivatives of
$\partial_i\mathcal J=-\Gamma_i^{\DD}\mathcal J$ give the asserted
$C^m$ bounds from this zeroth-order bound. Transport along each
normal segment extends smoothly to its boundary endpoint and gives,
in the fixed bundle frame,
\[
 \bigl|\mathcal J(\exp_p^{g_+}(s\alpha^{\sharp_{g_+}}))
            -P_{x_R\to p}\mathcal J(x_R)\bigr|
 \le Cs\sup_K|\mathcal J|.
\]

To prove that the radial limits give one boundary value, fix $q$
where $N_H^P(q)$ is not a line. Any two unit directions in this convex cone can
be joined within its unit section by spherical arcs of total
length at most $2\pi$. If the directions are opposite, use
any third direction not on their line. At distance $s>0$,
the normal exponential map sends these arcs to a path in $U$
of length at most $Cs$. If $\eta_s$ is such a path, the coordinate
parallel equation and the bound already proved give
\[
 \begin{aligned}
 \mathcal J(\eta_s(1))-\mathcal J(\eta_s(0))
 &=-\int_0^1\Gamma_i^\DD(\eta_s(t))\dot\eta_s^i(t)
                                      \mathcal J(\eta_s(t))\,dt,\\
 |\mathcal J(\eta_s(1))-\mathcal J(\eta_s(0))|
 &\le C\operatorname{Length}_{g_+}(\eta_s)\sup_K|\mathcal J|.
 \end{aligned}
\]
In particular,
\[
 \left|\mathcal J(\exp_q^{g_+}(s\alpha_1^{\sharp_{g_+}}))
       -\mathcal J(\exp_q^{g_+}(s\alpha_2^{\sharp_{g_+}}))\right|
       \le Cs\sup_K|\mathcal J|.
\]
The limit along each normal segment exists as $s\downarrow0$.
By this inequality it is independent of the direction, and we
denote it by $\mathcal J(q)$.

If $x_j\in U$ tends to $q$, its nearest points $p_j$ tend
to $q$. Pass to a subsequence of its unit normal covectors
converging to $\alpha\in N_H^P(q)$, using the chosen coordinates
to identify nearby cotangent spaces. The extended endpoints
satisfy
\begin{align*}
x_{R,j}&\longrightarrow\exp_q^{g_+}(R\alpha^{\sharp_{g_+}}),\\
\mathcal J(x_j)&=P_{x_{R,j}\to x_j}\mathcal J(x_{R,j})\longrightarrow \mathcal J(q).
\end{align*}
This proves convergence as $x\in U$ tends to $q$. At a sequence
of boundary points tending to $q$, use their radial expressions
$P_{x_R\to p}\mathcal J(x_R)$ and pass to a convergent subsequence
of unit normal covectors. The same limit proves continuity of the boundary
values wherever they are defined.

It remains to compare a smooth extension with the section on the
entire nearby exterior. The equality
$\widetilde{\mathcal J}(q)=\mathcal J(q)$ and uniqueness of parallel
transport first give $\widetilde{\mathcal J}=\mathcal J$ on each short normal segment
from $q$. The endpoints at a fixed length $R$ form a compact
subset of $U$. At each endpoint, transport within a coordinate
ball extends the equality to that ball. If equality failed at $x_j\in U$ tending to $q$,
extend its normal segment to length $R$. A subsequence of
the extended endpoints converges to the preceding compact
set, where equality holds on an open neighborhood. Transport along such a segment gives
\[
 |\mathcal J-\widetilde{\mathcal J}|(x_j)
 \le e^{CR}|\mathcal J-\widetilde{\mathcal J}|(x_{R,j})=0,
\]
which is a contradiction. All these segments lie
in the domain of $\widetilde{\mathcal J}$ after initially reducing
$R$ and the neighborhood.
\end{proof}

\begin{figure}[htbp]
\centering
\begin{tikzpicture}[x=1cm,y=1cm,font=\normalsize,
                   line cap=round,line join=round]
\path[use as bounding box] (0,-2.35) rectangle (15.4,2.85);
\begin{scope}[shift={(2.45,.1)}]
  \node[align=center,text width=4.7cm] at (0,2.4)
       {(a) Auxiliary ball\\and nearest points};
  \fill[black!5] (0,0) circle[radius=1.2];
  \draw[thick] (0,0) circle[radius=1.2];
  \coordinate (rc) at (0,0);
  \coordinate (rp) at (-.8,.8944);
  \coordinate (rq) at (.8,.8944);
  \coordinate (rm) at (0,.8944);
  \draw[thin] (rp)--(rc)--(rq);
  \draw[densely dashed] (rp)--(rq);
  \foreach \pt in {rc,rp,rq,rm} \fill (\pt) circle[radius=1.4pt];
  \node[below] at (rc) {$c$};
  \node[above left] at (rp) {$p$};
  \node[above right] at (rq) {$q$};
  \node[below=3pt,fill=white,inner sep=1pt] at (rm) {$m$};
  \node[left] at (-.47,.36) {$R$};
  \node[align=center,text width=4.75cm] at (0,-1.98)
       {$B_R(c)\cap H=\varnothing$\\[3pt]
        $|m-c|^2=R^2-\frac14|p-q|^2$};
\end{scope}
\begin{scope}[shift={(7.55,.1)}]
  \node[align=center,text width=4.8cm] at (0,2.4)
       {(b) Timelike motion\\and a null geodesic};
  % Perspective projection of three independent coordinate directions.
  % The contacts lie in z=0; their timelike motions leave that plane.
  \begin{scope}[x={(1cm,0cm)},y={(.3cm,.24cm)},z={(0cm,1cm)}]
    \fill[black!3] (0,0,0) circle[radius=1.2];
    \draw[gray!70] (0,0,0) circle[radius=1.2];
    \draw[densely dashed,gray!70] (-1.2,0,0)--(1.2,0,0);
    \draw[densely dashed,gray!70] (0,-1.2,0)--(0,1.2,0);
    \coordinate (sc) at (0,0,0);
    \coordinate (sp) at (-.8,.8944,0);
    \coordinate (sq) at (.8,.8944,0);
    \coordinate (spp) at (-.8,.8944,-.9);
    \coordinate (sqq) at (.8,.8944,.9);
    \coordinate (sz) at (.045,.86,.02);
    \draw[densely dashed,gray] (sp)--(sc)--(sq);
    \draw[densely dashed] (sp)--(sq);
    \draw[thin] (-.8,.8944,.36)--(-.8,.8944,-1.12);
    \draw[thin] (.8,.8944,-.3)--(.8,.8944,1.13);
    \draw[->,thick] (sp)--(spp);
    \draw[->,thick] (sq)--(sqq);
    \draw[very thick] (spp) .. controls (-.45,.87,-.50)
                     and (-.19,.85,-.21) .. (sz)
                     .. controls (.29,.87,.28)
                     and (.55,.90,.56) .. (sqq);
  \end{scope}
  % Sphere silhouette; the horizontal ellipse above is its contact plane.
  \draw[gray!70] (0,0) ellipse[x radius=1.26,y radius=1.2];
  \foreach \pt in {sc,sp,sq,spp,sqq,sz}
       \fill (\pt) circle[radius=1.35pt];
  \node[below left,fill=white,inner sep=1pt] at (sc) {$c$};
  \node[left=3pt,fill=white,inner sep=1pt] at (sp) {$p$};
  \node[right=3pt,fill=white,inner sep=1pt] at (sq) {$q$};
  \node[below left=1pt,fill=white,inner sep=1pt] at (spp) {$p(-t_{\ast})$};
  \node[above right=1pt,fill=white,inner sep=1pt] at (sqq) {$q(t_{\ast})$};
  \node[above left=3pt,fill=white,inner sep=1pt] at (sz) {$z$};
  \node[rotate=48,fill=white,inner sep=1pt] at (.6,.63) {null};
  \node[align=center,text width=4.8cm] at (0,-1.98)
       {$|z-m|\le C|p-q|^2$\\[3pt]
        $z\in B_R(c)$};
\end{scope}
\begin{scope}[shift={(12.85,.1)}]
  \node[align=center,text width=4.85cm] at (0,2.4)
       {(c) Exterior segments\\and parallel transport};
  % A convex corner shows that no smooth-boundary assumption is needed.
  \fill[black!9] (-1.85,-1.2)--(-1.85,-.57)--(-.65,.03)
      --(.2,.21)--(1.65,-.056)--(1.65,-1.2)--cycle;
  \draw[thick] (-1.85,-.57)--(-.65,.03)--(.2,.21)--(1.65,-.056);
  \node at (-1.2,-.83) {$H$};
  \node at (1.32,.50) {$U$};
  % K is a compact positive-distance subset of the exterior, not a boundary.
  \fill[black!4] (-1.93,.82)
     .. controls (-.8,1.55) and (.15,1.81) .. (1.7,1.38)
     --(1.7,1.11)
     .. controls (.1,1.52) and (-.8,1.30) .. (-1.93,.56)--cycle;
  \draw[densely dashed] (-1.93,.82)
     .. controls (-.8,1.55) and (.15,1.81) .. (1.7,1.38);
  \draw[densely dashed] (-1.93,.56)
     .. controls (-.8,1.30) and (.1,1.52) .. (1.7,1.11);
  \node[fill=white,inner sep=1pt] at (.95,1.55) {$K\Subset U$};
  \coordinate (tp) at (.2,.21);
  \coordinate (tx) at (.2,.72);
  \coordinate (tr) at (.2,1.46);
  \draw[thick] (-1.05,-.17)--(-1.61,.95);
  \draw[thick] (1.0,.063)--(1.226,1.293);
  \draw[thick] (tp)--(tr);
  \draw[->,thick] (.2,1.19)--(.2,.87);
  \foreach \pt in {tp,tx,tr} \fill (\pt) circle[radius=1.4pt];
  \node[below right] at (tp) {$p$};
  \node[right=3pt,fill=white,inner sep=1pt] at (tx) {$x$};
  \node[above left=2pt,fill=white,inner sep=1pt] at (tr) {$x_R$};
  \node[align=center,text width=4.85cm] at (0,-1.98)
       {$\mathcal J(x)=P_{x_R\to x}\mathcal J(x_R)$\\[3pt]
        $d_{g_+}(x_R,H)=R$};
\end{scope}
\end{tikzpicture}
\caption{Short null geodesics and Killing transport.
(a) Two nearest points $p,q\in H$ to the center of a ball disjoint from $H$ give
a chord midpoint $m=(p+q)/2$ strictly inside it.
(b) In the three-dimensional stationary quotient, opposite
timelike boundary motions produce a short null geodesic whose
endpoints lie in $H$ but whose midpoint $z$ remains inside the
ball, contradicting the required inclusion of the null segment in $H$.
The schematic ball is measured in the auxiliary Riemannian metric.
(c) The resulting unique exterior normal segments reach a fixed
compact set $K\Subset U$ at distance $R$ from $H$.
Parallel transport $P$ carries the controlled values of $\mathcal J$
on $K$ to nearby exterior points, uniformly up to the possibly
nonsmooth boundary. For the connection $\DD$ defined after
\eqref{local:killing-connection},
$\mathcal J=(\Z^\flat,\D(\Z^\flat))$.
See Lemmas~\ref{closure:null-reach} and~\ref{normal:parallel-boundary}.}
\label{fig:null-reach-transport}
\end{figure}

\begin{lemma}\label{closure:normal-limits}
For the maximal pair $(H,\Z)$ fixed above,
put $U=U_{H}$. The set $H$ has locally
uniform positive reach in the ergoregion.
For every compact set
\[
 K\subset\partial H\cap\{\g(\T,\T)>0\},
\]
there are a
neighborhood $O\subset Q$ and constants $C_m,c>0$ such that,
on the section $\vartheta=0$,
\[
 \|(\Z,\D\Z)\|_{C^m(\{\vartheta=0\}\cap U\cap\pi^{-1}(O))}
       \le C_m\qquad(m\ge0),
\]
where the $C^m$ norm uses coordinate derivatives on the stationary
section and the positive tensor norm induced by $\mathbf g_+$.
Also $c\le\rho\le c^{-1}$ there, where
$\rho=\sqrt{W[\T,\Z]}=\sqrt{-h(\Z,\Z)}$.
For every $p\in K$ and every $\alpha\in N_H^P(p)$, we have
\[
 h^{-1}(\alpha,\alpha)\ge c|\alpha|_{g_+}^2.
\]
For $p\in K$, if $N^P_{H}(p)$ is not a line, the pair
$(\Z,\D\Z)$ has one limit as $x\in\pi(U)$ tends to $p$,
with values evaluated on $\vartheta=0$. The projected limit of $\Z$ is timelike for
$h$ and annihilated by every $\alpha\in N^P_{H}(p)$.
\end{lemma}
\begin{proof}
Having obtained positive reach for $H$, we apply the
boundedness part of Lemma~\ref{normal:parallel-boundary} to the
pullback of $\DD$ to $\vartheta=0$. Recall from
\eqref{local:killing-connection} that its parallel section is
$\mathcal J=(\Z^\flat,\D(\Z^\flat))$, with second component
$(\D(\Z^\flat))_{\alpha\beta}=\D_\alpha\Z_\beta$.
Let $K_R\Subset\pi(U)$ be the compact set of exterior normal
endpoints at the fixed distance $R$ supplied by that lemma. Equation~\eqref{normal:radial-transport} bounds the section from
its values on $K_R$; the coordinate parallel equation then gives
the asserted $C^m$ bounds.
For the lower bound, Proposition~\ref{action:relative} gives
$|d\log\rho|_{g_+}\le C$ along the same exterior normal segments.
If $x_R\in K_R$ is the endpoint corresponding to $x$, then
\[
 |\log\rho(x)-\log\rho(x_R)|
 \le\int_x^{x_R}|d\log\rho|_{g_+}\,ds\le CR.
\]
Since $K_R$ is compactly contained in the old domain and $\rho>0$
there, we obtain
\[
 e^{-CR}\min_{K_R}\rho\le\rho(x)\le e^{CR}\max_{K_R}\rho.
\]
For the normal covectors we use the exterior flow of
$V=\Z/\rho$, where $h(V,V)=-1$, as in the reach proof.
Fix a sufficiently small $g_+$-ball disjoint from $H$
and tangent to it at $p$. Limits of the normalized
flow from this ball give a two-sided boundary curve $p(t)$ with
\begin{align*}
 p(t)&=p+tV(p)+O(t^2),\\
 h(V(p),V(p))&=-1,\\
 |V(p)|_{g_+}&\le C.
\end{align*}
Here $V(p)$ is the limit from this ball. For every unit
$\alpha\in N^P_{H}(p)$, the quadratic inequality
\eqref{normal:conormal-inequality} gives, for both signs of $t$,
\begin{align*}
 \alpha(p(t)-p)&\le C|p(t)-p|^2,\\
 t\alpha(V(p))&\le C t^2.
\end{align*}
It follows that $\alpha(V(p))=0$. In a local $h$-orthonormal
frame, write the components of $V(p)$ as $V^0,V^1,V^2$ and
the dual components of $\alpha$ as $\alpha_0,\alpha_1,\alpha_2$.
The identities $h(V(p),V(p))=-1$ and $\alpha(V(p))=0$ give
\begin{align*}
 (V^0)^2&=1+(V^1)^2+(V^2)^2,\\
 \alpha_0&=-\frac{\alpha_1V^1+\alpha_2V^2}{V^0}.
\end{align*}
Therefore
\[
 \begin{aligned}
 h^{-1}(\alpha,\alpha)
 &=\alpha_1^2+\alpha_2^2
     -\frac{(\alpha_1V^1+\alpha_2V^2)^2}{(V^0)^2}\\
 &\ge\frac{\alpha_1^2+\alpha_2^2}{(V^0)^2}
 \ge c|\alpha|_{g_+}^2.
 \end{aligned}
\]
For the last inequality we used the uniform bound for $|V(p)|_{g_+}$,
$\alpha_0^2\le\alpha_1^2+\alpha_2^2$, and equivalence of the
coordinate and $g_+$ norms on a finite frame cover of $K$.
The constant is independent of the exterior ball.
If $p_j\to p\in K$ and $\alpha_j\in N_H^P(p_j)$ converge to
$\alpha$, closedness gives $\alpha\in N_H^P(p)$, so the same
inequality also holds for every such limit.

If $N^P_{H}(p)$ is not a line,
Lemma~\ref{normal:parallel-boundary} gives one limit
$\mathcal J(p)$ of $(\Z^\flat,\D(\Z^\flat))$ as exterior points tend to $p$.
The positive lower bound for $\rho$ shows that every preceding
limit from a ball disjoint from $H$ is the projection of this same field divided by
$\rho(p)$. Hence
\begin{align*}
 \alpha(\Z(p))&=0
                 \quad(\alpha\in N^P_{H}(p)),\\
 h(\Z(p),\Z(p))&=-\rho(p)^2<0.
\end{align*}
\end{proof}

We next choose coordinates in which the exterior components are
explicit. Recall that a convex cone is pointed if its intersection
with its negative is $\{0\}$. At a point $p$ where $N^P_{H}(p)$ is pointed,
the bound \eqref{normal:conormal-inequality} gives a direction in which each nearby
line crosses the boundary exactly once. When the cone is a line,
we obtain a lower and an upper boundary that can coincide. In the
second case we will construct an auxiliary field from each exterior
side, using the limits of $(\Z,\D\Z)$ from that component.

We use the semiconvexity convention of Rataj and
Zaj\'i\v{c}ek~\cite[Definition~2.1]{RZ17}, with a rescaled constant:
$v$ is semiconvex if, locally, $v(x)+C|x|^2$ is convex for some
$C\ge0$. Semiconcavity of $v$ means semiconvexity of $-v$.

\begin{lemma}\label{closure:graph-band}
Let $p\in\partial H\cap\{\g(\T,\T)>0\}$.
If $N^P_{H}(p)$ is pointed, there are local coordinates
$(x,s)$, with $x\in\mathbb R^2$, in which
\[
 H=\{s\ge v(x)\},
\]
with $v$ Lipschitz and semiconvex. If this cone is a line, there
are coordinates of the same form, centered at $p$, in which
\[
 \begin{aligned}
 H&=\{a(x)\le s\le b(x)\},\\
 a(0)&=b(0)=0,\\
 Da(0)&=Db(0)=0,
 \end{aligned}
\]
where $a$ is semiconvex and $b$ semiconcave.
\end{lemma}
\begin{proof}
We use the Euclidean metric of these coordinates to identify
vectors and covectors, and normalize the normal covectors in this metric. The change-of-metric calculation following
Definition~\ref{def:proximal-conormal} gives
\eqref{normal:conormal-inequality} with a uniform constant $C$ for
these unit covectors. Its closedness consequence makes the unit normal
cones upper semicontinuous: if $q_k\to p$ and unit
$\alpha_k\in N_{H}^P(q_k)$ converge to $\alpha$, then
$\alpha\in N_{H}^P(p)$.

In the pointed case choose a unit vector $E$ and $c>0$ with
$\alpha(E)>2c$ for every unit normal covector at $p$. After reducing the
neighborhood, every unit normal covector there satisfies $\alpha(E)>c$.
For a boundary point $q_0$, put
$z=q_0-tE+w$, where $|w|\le\epsilon t$ and $\epsilon<c/2$.
If $z\notin H$, let $q$ be its nearest point and
$\alpha=(z-q)/d$, where $d=|z-q|\le t+|w|$. Then
\[
 d+ct-|w|\le\alpha(q_0-q)
       \le C|q_0-q|^2
       \le C(d+t+|w|)^2.
\]
The left side is at least $ct/2$ and the right side is $O(t^2)$.
Thus this cone lies in $H$ for a uniform small range of $t$.
For $z=q_0+tE+w$, any unit normal covector at $q_0$ instead gives
$\alpha(z-q_0)\ge ct/2>C|z-q_0|^2$, so this opposite cone lies
outside $H$.

Choose a coordinate product of a disk with an interval whose end disks lie in these two cones at $p$.
The inner disk lies in $H$ and the outer disk lies outside $H$.
Every line segment parallel to $E$ therefore meets the boundary. Along
this segment, the two cones put an interval of $H$ below every boundary point and an exterior interval above it.
Two boundary crossings would require a reentry with the opposite
orientation. Thus each segment has exactly one crossing. The same
cones imply that its height difference between two such segments is at
most $\epsilon^{-1}$ times their separation: a larger difference
would place one boundary point in the other's interior or exterior
cone. Reverse $E$ to obtain the inward coordinate $s$ and the graph
$H=\{s\ge v(x)\}$. A unit normal covector at $(x,v(x))$ has
the decomposition $\alpha=\alpha_i\,dx^i+\alpha_s\,ds$, with
$i\in\{1,2\}$ and $\alpha_x=(\alpha_1,\alpha_2)$.
Its height component satisfies $\alpha_s\le-c$, and hence
\[
 v(z)\ge v(x)-\frac{\alpha_x}{\alpha_s}\cdot(z-x)
       -\frac{C(1+\epsilon^{-2})}{c}|z-x|^2.
\]
After adding a fixed positive quadratic function to $v$, this
inequality gives an affine function touching it from below at
every point. This proves semiconvexity.

Suppose now that the cone at $p=0$ is $\mathbb R\,ds$.
Upper semicontinuity puts every nearby unit normal covector in small
neighborhoods of $ds$ and $-ds$. In particular,
\[
 |\alpha_s|\ge c>0.
\]
The two normal covectors at $p$ give
$|s|\le C(|x|^2+s^2)$ on $H$. Choose a closed cylinder
$\overline{B_R^2}\times[-\delta,\delta]$ so small that the regions near its two ends
$|s|\ge\delta/2$ miss $H$.

Every vertical segment with $|x_0|<R/4$ meets $H$ in this cylinder.
Otherwise minimize $|x-x_0|^2$ over its compact intersection
with $H$. Since $0\in H$, the minimum is at most $|x_0|^2$;
the minimizer $q=(x_q,s_q)$ cannot lie on a lateral face, and
the empty end strips exclude the other faces. If that segment missed $H$, $\nu=(x_0-x_q,0)$ would be nonzero and
\[
 \nu(z-q)\le\tfrac12|x_z-x_q|^2\le\tfrac12|z-q|^2
 \qquad(z\in H\text{ near }q).
\]
The quadratic characterization puts $\nu$ in $N_{H}^P(q)$. Since
$\nu_s=0$, this contradicts $|\alpha_s|\ge c$.

We now show that each vertical segment meets $H$ in one interval. For this we
first determine the orientation of a normal covector at an endpoint of a
possible gap. Take $q\in H$ with $q+t\partial_s\notin H$ for small
$t>0$. Let $q_t$ be the nearest point of $q+t\partial_s$,
put $d_t=|q+t\partial_s-q_t|$, and let
$\alpha_t=(q+t\partial_s-q_t)/d_t$. Since $|q-q_t|\le2t$,
\[
 t\alpha_t(\partial_s)
 =d_t-\alpha_t(q-q_t)\ge-4Ct^2.
\]
A limiting unit normal covector at $q$ therefore has nonnegative
$s$-component, and that component is at least $c$.
For a gap of length $d$ above $q$, evaluation at its upper endpoint
would give $cd\le Cd^2$. Take $2\delta<c/C$ to exclude every gap.
Thus its intersection with $H$ is one closed interval $[a(x),b(x)]$.
The same projection argument, with the sign of $t$ reversed,
gives a normal covector with $\alpha_s\le-c$ at its lower endpoint;
at its upper endpoint there is one with $\alpha_s\ge c$.

Compare the lower endpoints over $x$ and $\eta$.
Use a normal covector with $\alpha_s<0$ at the higher of the two lower endpoints.
The quadratic inequality \eqref{normal:conormal-inequality} gives
\[
 c|a(\eta)-a(x)|\le |\eta-x|
                  +C\bigl(|\eta-x|^2+|a(\eta)-a(x)|^2\bigr).
\]
Since $|a(\eta)-a(x)|\le2\delta$, reducing $\delta$ absorbs the last
term and proves a uniform Lipschitz bound. The upper endpoints
satisfy the same estimate, using normal covectors with positive $s$-component. If $L$ is a
common Lipschitz constant, the original inequalities now give
\[
 \begin{aligned}
 a(\eta)&\ge a(x)-\frac{\alpha_x}{\alpha_s}\cdot(\eta-x)
                  -\frac{C(1+L^2)}c|\eta-x|^2,
                       &&\alpha_s\le-c,\\
 b(\eta)&\le b(x)-\frac{\beta_x}{\beta_s}\cdot(\eta-x)
                  +\frac{C(1+L^2)}c|\eta-x|^2,
                       &&\beta_s\ge c.
 \end{aligned}
\]
These are the asserted semiconvexity and semiconcavity estimates.
For $H$, the two normal covectors at $p$ moreover give
$|a(x)|+|b(x)|\le C'|x|^2$, so $Da(0)=Db(0)=0$.
\end{proof}

\subsection{Optical convexity at the boundary}

For a rotationally invariant hypersurface,
\eqref{local:optical-null-curvature} expresses both null curvatures
as $\kappa_{\mathrm{opt}}/(2\rho)$. Thus positive geodesic curvature
gives the strict null Hessian needed
for Killing extension, while zero geodesic curvature makes both
null tangent directions geodesic. To use this identity at a rough
boundary, we extend the optical metric smoothly from one exterior
component. The short-null-segment condition then gives convexity for
this metric: a geodesic leaving the corresponding epigraph would
have an exterior excursion whose null lift has both endpoints in $H$.
This argument only uses the Killing equation on the exterior excursion.

Fix a boundary graph $\Sigma$ from Lemma~\ref{closure:graph-band}
and one adjacent exterior component. We orient its transverse
coordinate so that this component lies below the graph.

\begin{lemma}\label{local:optical-reduction}
Near a point of $\Sigma$ there is a smooth stationary field
$\widetilde\Z$, equal to $\Z$ on the chosen exterior, whose
projection is timelike for $h$ and whose local flow preserves
$\Sigma$. A transverse disk $S$ has coordinates $(y,s)$ in which
\[
 \Sigma\cap S=\{s=v(y)\},
\]
with $v$ Lipschitz and semiconvex and the chosen exterior given by
$s<v(y)$. Put $\rho^2=-h(\widetilde\Z,\widetilde\Z)$. The metric
\[
 q_{\mathrm{opt}}(X,Y)
 =\rho^{-2}h(X,Y)
       +\rho^{-4}h(\widetilde\Z,X)h(\widetilde\Z,Y),
 \qquad X,Y\in TS,
\]
is smooth and positive definite on $S$. It equals the optical metric on the chosen exterior. The Killing
and circularity identities for $\widetilde\Z$ hold on that exterior
and on $\Sigma$.
\end{lemma}
\begin{proof}
Choose graph coordinates $(t,y,s)$ on $Q$, centered at the point
under consideration. After reversing $s$ for an upper exterior
component, the chosen exterior is $s<a(t,y)$ with $a$ Lipschitz
and semiconvex. On the fixed stationary section
put $\mathcal J=(\Z^\flat,\D(\Z^\flat))$ and choose a smooth
bundle frame. Write $\Gamma_s^\DD$ for the matrix in this frame
of the Killing transport connection $\DD$, pulled back to the stationary
section, in the $s$ direction.
Apply the transport construction \eqref{local:graph-transport}
from a strip about $s=s_-$ in the chosen exterior, solving
\[
 \begin{aligned}
 \partial_s\widetilde{\mathcal J}(t,y,s)
    &=-\Gamma_s^\DD(t,y,s)\widetilde{\mathcal J}(t,y,s),\\
 \widetilde{\mathcal J}(t,y,s_-)&=\mathcal J(t,y,s_-).
 \end{aligned}
\]
The smooth coefficients and strip data
give a smooth solution on the whole coordinate rectangle. Below
the boundary graph, each vertical segment stays in the exterior,
so uniqueness gives $\widetilde{\mathcal J}=\mathcal J$ there.
We take $\widetilde\Z$ to be the metric dual of the first
component and extend it stationarily. It is timelike for $h$ near
the boundary by Lemma~\ref{closure:normal-limits}. In the rectangle
we use the smooth field supplied by this transport equation. On the
chosen exterior and its boundary we also have the Killing and circularity identities,
by equality with $\Z$ and continuity.

We claim that the local flow of $\widetilde\Z$ preserves this boundary.
Reduce the cylinder and the time interval so that the trajectories
remain in the original cylinder. On the chosen exterior component this flow equals the
complete flow of $\Z$. A trajectory cannot change exterior
components without meeting $H$. For exterior points
$p_k\to p\in\partial H$, smooth dependence gives
\[
 \Phi_\tau^{\widetilde\Z}(p)
     =\lim_k\Phi_\tau^\Z(p_k)
       \in\overline{Q\setminus H}.
\]
If this limit were exterior, the inverse flow would give
$p=\Phi_{-\tau}^\Z(\Phi_\tau^{\widetilde\Z}(p))\notin H$,
a contradiction. Thus the flow preserves the boundary graph.

Writing $\widetilde\Z^i$ for the projected components in $(t,y,s)$,
the Lipschitz bound for $a$ along a boundary trajectory gives
\[
 |\widetilde\Z^s|
 \le C\bigl(|\widetilde\Z^t|+|\widetilde\Z^y|\bigr).
\]
Since $\widetilde\Z$ is timelike, its two base components cannot
both vanish. After a linear change of $(t,y)$ and a restriction of the chart,
we have $\widetilde\Z^t\ne0$. Take the transverse disk $S=\{t=0\}$, with coordinates $(y,s)$. The flow map
\[
 (\tau,y,s)\longmapsto\Phi_\tau^{\widetilde\Z}(0,y,s)
\]
has, at $\tau=0$, the nonzero differential determinant
\[
 \det\bigl(\widetilde\Z,\partial_y,\partial_s\bigr)|_S
                 =\widetilde\Z^t|_S\ne0.
\]
It therefore gives a smooth local coordinate system. On the disk,
put $v(y)=a(0,y)$, which is Lipschitz and semiconvex by restriction.
Boundary invariance and the inverse flow then give $s=v(y)$ for
$\Sigma$ in these coordinates, with the chosen exterior $s<v(y)$.
For the metric in the statement we have
\[
 q_{\mathrm{opt}}(X,X)
 =\rho^{-2}h\bigl(X+\rho^{-2}h(\widetilde\Z,X)\widetilde\Z,
                  X+\rho^{-2}h(\widetilde\Z,X)\widetilde\Z\bigr).
\]
The vector inside the pairing is the $h$-orthogonal projection of
$X$ onto $\widetilde\Z^{\perp_h}$. This plane is spacelike, and
that projection is injective on the transverse plane $TS$.
Thus $q_{\mathrm{opt}}$ is positive definite. On the exterior its
formula is exactly orthogonal projection to $\Z^{\perp_h}$ followed
by rescaling by $\rho^{-2}$, as in \eqref{action:reduction}.

\end{proof}

\begin{lemma}\label{local:optical-convexity}
In the transverse disk of Lemma~\ref{local:optical-reduction}, the
closed epigraph
\[
 C=\{(y,s):s\ge v(y)\}
\]
is locally geodesically convex for $q_{\mathrm{opt}}$.
\end{lemma}
\begin{proof}
Use the local flow of $\widetilde\Z$ to give coordinates
$(\phi,y,s)$. On the chosen exterior, the Killing equation and
circularity give
\[
 h=\rho^2\bigl(-(d\phi+\beta)^2+q_{\mathrm{opt}}\bigr),
\]
where the coefficients are independent of $\phi$ and
$\beta=-\rho^{-2}h(\widetilde\Z,\cdot)|_{TS}$. To check its
closedness on the exterior, circularity gives
\[
 \begin{aligned}
 0&=h(\Z,\cdot)\wedge d[h(\Z,\cdot)]\\
  &=\rho^4(d\phi+\beta)\wedge d\beta
   =\rho^4d\phi\wedge d\beta.
 \end{aligned}
\]
The last equality uses the two-dimensional disk, so $d\beta=0$.
We retain $\beta$, since the chosen disk need not be orthogonal to
the rotation.

Let $c$ be a sufficiently short affinely parametrized
$q_{\mathrm{opt}}$-geodesic with endpoints in $C$. If it leaves
$C$, take a component $(u_-,u_+)$ of its exterior excursion. Its
endpoints lie on $s=v(y)$. Write $k=|\dot c|_{q_{\mathrm{opt}}}$,
which is constant, and lift this excursion by
\[
 \dot\phi=k-\beta(\dot c).
\]
Locally choose $\chi$ with $d\chi=\beta$ and put
$\theta=\phi+\chi(c)$. Then $\dot\theta=k$, so
\eqref{local:optical-geodesics} identifies the lift as a null geodesic
of $-d\theta^2+q_{\mathrm{opt}}$ and hence, after reparametrization,
of $h$. These local descriptions agree on overlaps,
since the equation for $\phi$ is fixed.

There is no loss of the geodesic equation at the two endpoints of
the excursion. Let $\widehat c=(\phi,c)$ denote the lift just
constructed and $X$ the velocity of its stationary horizontal
spacetime lift, with the original parameter $u$. The conformal
connection formula and \eqref{local:null-lift} give
\[
 \begin{aligned}
 \nabla_{\dot{\widehat c}}\dot{\widehat c}
     &=2\frac d{du}\log\rho\,\dot{\widehat c},\\
 \D_XX
     &=\frac d{du}\log\!\left(\frac{\rho^2}{\g(\T,\T)}\right)X.
 \end{aligned}
\]
In the second line both sides are horizontal, so the quotient
identity determines the spacetime identity. Therefore the change
of parameter
\[
 \frac{d\lambda}{du}
       =\frac{\rho^2}{\g(\T,\T)}>0
\]
makes the spacetime lift affinely parametrized. On a smaller
flowbox this derivative and its reciprocal are bounded. The
coefficients of the optical metric, $\beta$, and the horizontal
lifting equation are smooth up to the boundary, and
$|\dot\phi|\le Ck$, $|X|_{\mathbf g_+}\le Ck$.
The lift and its velocity thus extend to both endpoints; the
affine geodesic equation holds there by continuity. Its auxiliary
length is at most $C\operatorname{Length}_{q_{\mathrm{opt}}}(c)$.
Choose the chord short enough that its entire lift stays in the
flowbox and this length is less than $\epsilon_0$.
Boundary invariance gives
\[
 (\phi(u_-),c(u_-)),\ (\phi(u_+),c(u_+))\in\Sigma\subset H,
\]
whereas the open arc lies in the chosen exterior. This contradicts
Definition~\ref{discrete:null-convexity}. Every short chord therefore
stays in $C$.
\end{proof}

\begin{lemma}\label{local:weak-cap}
Let $v,q_{\mathrm{opt}}$ be as in
Lemma~\ref{local:optical-reduction}. If a smooth curve $s=w(y)$
touches $v$ from below and has $\kappa_{\mathrm{opt}}>0$ toward
increasing $s$ at contact, its images under the local flow of
$\widetilde\Z$ form a supporting hypersurface for $H$ whose
stationary lift is strictly $\T$-conditionally pseudoconvex.
\end{lemma}
\begin{proof}
Let $f=s-w(y)$ in the auxiliary flow coordinates. Since $w\le v$,
its entire negative side lies in the chosen exterior. At contact, the Killing and circularity identities follow by
continuity from that exterior. Let $n$ point toward increasing $s$,
and take the normalized null tangents $e_3,e_4$ of
Lemma~\ref{local:optical-curvature}. Since $nf>0$, equations
\eqref{local:optical-null-curvature} and \eqref{flow:null-extension} give
\[
 \D^2(f\circ\pi)(X,X)
  =-(nf)B(d\pi X,d\pi X)<0
 \quad\bigl(0\ne X,\ X(f\circ\pi)=\g(X,X)=\g(\T,X)=0\bigr).
\]
The same calculation holds if the touching curve meets the boundary
only at the contact point: shift it into the exterior, apply the
smooth identity there, and take the limit. Its first two derivatives
and the smooth metric coefficients converge at contact.
Since $df$ annihilates the timelike vector $\widetilde\Z$, its metric
dual is spacelike. The stationary conormal is therefore spacelike as well. This verifies
the hypotheses of Proposition~\ref{local:extension}.

The extension agrees with $\Z$ on the entire chosen exterior in a
smaller neighborhood. Indeed, after applying the local theorem,
choose a rectangular flowbox inside the extension neighborhood and
a strip below both $w$ and $v$. The two fields and their first
derivatives agree on this strip. Every point below $v$ joins it by
a vertical segment in that exterior, and uniqueness for
\eqref{local:killing-connection} gives agreement along the segment.
\end{proof}

We will also need the opposite exterior when the two boundary
graphs meet. Their optical metrics may differ, so the following
argument uses only the optical metric extended smoothly from that
opposite exterior component.

\begin{lemma}\label{discrete:weak-null-tests}
Suppose locally
\[
 H=\{(t,y,s):a(t,y)\le s\le b(t,y)\},
\]
with the boundary graphs of Lemma~\ref{closure:graph-band}.
If $a(p)=b(p)$, no smooth strictly $\T$-conditionally pseudoconvex
hypersurface supports $H$ at $p$.
\end{lemma}
\begin{proof}
Center the graph coordinates at $p$, with common height zero.
Semiconvexity of $a$ and semiconcavity of $b$ supply linear functions
$\ell_a,\ell_b$ with
\[
 \ell_a(x)-C|x|^2\le a(x)\le b(x)
                         \le\ell_b(x)+C|x|^2.
\]
Taking $x=rX$ with both signs of $r$ shows $\ell_a=\ell_b$.
Both graphs therefore have one common differential and a quadratic
remainder at $p$. Any smooth supporting hypersurface has this tangent
plane. Suppose such a hypersurface is strictly pseudoconvex and,
after reversing $s$ if necessary, take its defining function $f$
with $f\ge0$ on the upper graph and $df_p$ pointing toward increasing $s$. We use the auxiliary
field of Lemma~\ref{local:optical-reduction} constructed from the
upper exterior. In this proof all occurrences of $\widetilde\Z$
and $q_{\mathrm{opt}}$ refer to this choice. The upper graph is
invariant under its local flow. As $f$ has a minimum along the
boundary orbit through $p$,
\[
 \begin{aligned}
 \widetilde\Z f(p)&=0,\\
 \widetilde\Z^2f(p)&\ge0.
 \end{aligned}
\]
Take a transverse disk $S$ tangent at $p$ to
$\widetilde\Z^{\perp_h}$, and extend $f|_S$ constantly by this
local flow to a function $\widetilde f$. At $p$ the two differentials
agree. We calculate the connection terms from the upper exterior
before taking their boundary limits. For fields
$A,B\perp_h\widetilde\Z$, Killing's equation and circularity give
\[
 2h(\nabla_A\widetilde\Z,B)
 =h(\nabla_A\widetilde\Z,B)-h(\nabla_B\widetilde\Z,A)
 =-h(\widetilde\Z,[A,B])=0.
\]
Together with
$h(\nabla_A\widetilde\Z,\widetilde\Z)=-\rho A\rho$,
this gives
$\nabla_A\widetilde\Z=(A\log\rho)\widetilde\Z$.
Killing's equation also gives
$\nabla_{\widetilde\Z}\widetilde\Z
=-\tfrac12\operatorname{grad}_h h(\widetilde\Z,\widetilde\Z)
=\rho\operatorname{grad}_h\rho$.
These first-derivative identities pass to $p$ by smoothness.
Put $e_0=-\widetilde\Z/\rho$ and choose
$E\in T_pS\cap\ker df_p$ with $h(E,E)=1$, extending $E$
by the auxiliary flow. At $p$ we consequently have
\[
 \begin{aligned}
 \nabla_{e_0}e_0&=\rho^{-1}\operatorname{grad}_h\rho,\\
 \nabla_{e_0}E&=(E\log\rho)e_0.
 \end{aligned}
\]
Since $[\widetilde\Z,E]=0$,
$\widetilde\Z\widetilde f=0$, and $df_p=d\widetilde f_p$,
\begin{align*}
 \Hess_h\widetilde f(e_0,e_0)
   &=\Hess_hf(e_0,e_0)-\rho^{-2}\widetilde\Z^2f,\\
 \Hess_h\widetilde f(e_0,E)&=0,\\
 \Hess_h\widetilde f(E,E)&=\Hess_hf(E,E).
\end{align*}
For $e_4=(e_0+E)/\sqrt2$ and $e_3=(e_0-E)/\sqrt2$, we have
\[
 \begin{aligned}
 2\Hess_hf(e_4,e_4)
 &=\Hess_hf(e_0,e_0)+2\Hess_hf(e_0,E)+\Hess_hf(E,E),\\
 2\Hess_hf(e_3,e_3)
 &=\Hess_hf(e_0,e_0)-2\Hess_hf(e_0,E)+\Hess_hf(E,E).
 \end{aligned}
\]
The mixed terms cancel in their sum. Thus each null Hessian of the
invariant function is
\begin{equation}\label{local:invariant-hessian}
 \Hess_h\widetilde f(e_a,e_a)
 =\frac{\Hess_hf(e_3,e_3)+\Hess_hf(e_4,e_4)}2
             -\frac{\widetilde\Z^2f}{2\rho^2}<0,
 \qquad a=3,4.
\end{equation}
Boundary invariance gives $\widetilde f\ge0$ on the upper graph.
Its zero curve in $S$ therefore touches that graph from below, and
\eqref{local:invariant-hessian} makes its geodesic curvature strictly
positive toward increasing $s$. On the other hand,
Lemma~\ref{local:optical-convexity}, applied from the upper exterior,
says that the hypograph of this boundary is geodesically convex.
We verify the contradiction directly.

Use $q_{\mathrm{opt}}$-normal coordinates $(y,s)$ at contact,
with common tangent $s=0$. Write the upper boundary as $s=b(y)$
and the smooth lower curve as $s=w(y)$. The common differential
and quadratic remainder proved above give
\[
 b(y)=b(0)+w'(0)y+O(y^2)=O(y^2).
\]
Positive curvature gives $w''(0)>0$, so
\[
 \frac{b(\eta)+b(-\eta)}2
 \ge\frac{w(\eta)+w(-\eta)}2
 =\tfrac12w''(0)\eta^2+o(\eta^2).
\]
The short geodesic joining $(\pm\eta,b(\pm\eta))$ has coordinate
velocity $O(\eta)$ and Christoffel symbols $O(\eta)$. Its affine
midpoint $(y_m,s_m)$ consequently satisfies
\[
 \begin{aligned}
 y_m&=O(\eta^3),\\
 s_m&=\tfrac12\{b(\eta)+b(-\eta)\}+O(\eta^3).
 \end{aligned}
\]
Since $b(y_m)=O(\eta^6)$, this midpoint lies strictly above the
upper boundary for small $\eta$. This contradicts geodesic
convexity of its hypograph. Reversing $s$ treats a hypersurface
supporting from the upper exterior.
\end{proof}

\begin{lemma}\label{local:geodesic-stalemate}
Let $\Sigma,v,q_{\mathrm{opt}}$ be as in
Lemma~\ref{local:optical-reduction}. If no point of $\Sigma$ admits
a smooth strictly $\T$-conditionally pseudoconvex supporting
hypersurface from the chosen exterior, then $\Sigma\cap S$ is a
smooth $q_{\mathrm{opt}}$-geodesic and
\[
 B_{33}=B_{44}=0\quad\text{on }\Sigma.
\]
The null curves tangent to $\Sigma$ are $h$-geodesics after
reparametrization.
\end{lemma}
\begin{proof}
By Lemma~\ref{local:optical-convexity}, the epigraph of $v$ is
locally geodesically convex. We first exclude a corner. If the two
one-sided slopes at $y_0$ differ, choose a slope $\mu$ strictly between
them. Semiconvexity gives $c>0$ and $C<\infty$ such that
\[
 v(y)-v(y_0)-\mu(y-y_0)
       \ge c|y-y_0|-C|y-y_0|^2.
\]
For any fixed sufficiently large $A$, the smooth curve
\[
 w(y)=v(y_0)+\mu(y-y_0)+A(y-y_0)^2
\]
lies below $v$ near $y_0$ and has positive geodesic curvature at
contact. Lemma~\ref{local:weak-cap} contradicts the hypothesis.
Thus $v$ is differentiable. The subdifferential of the convex
function $v+C y^2$ is a singleton at every point and is upper
semicontinuous, so $v'$ is continuous.

Join two sufficiently close boundary points by their short
$q_{\mathrm{opt}}$-geodesic. Convexity puts the chord in the
closed epigraph. Since the boundary is $C^1$, Fermi coordinates
$(y,n)$ about this chord express the intervening boundary as
$n=v(y)$, with
\[
 v(y_-) = v(y_+) = 0 \quad\text{and}\quad v\le0.
\]
In these coordinates, with $y$ the arclength on the chord,
\[
 q_{\mathrm{opt}}=dn^2+J(y,n)^2dy^2.
\]
The geodesic equation of the chord gives
\[
 J(y,0)=1 \quad\text{and}\quad J_n(y,0)=0.
\]
The only possibly nonzero Christoffel symbols are
\[
 \Gamma^n_{yy}=-JJ_n,\qquad
 \Gamma^y_{yy}=J_y/J,\qquad
 \Gamma^y_{yn}=\Gamma^y_{ny}=J_n/J.
\]
They vanish on $n=0$, since $J_y(y,0)=0$, and satisfy
$|\Gamma^k_{ij}(y,n)|\le C|n|$ on a smaller strip.
The constant is uniform for these short chords: their initial
points and unit tangents lie in a compact set, and the normal
exponential map depends smoothly on these data.

Suppose that $v<0$ somewhere. Let $d=y_+-y_-$, let $y_c$ be the
midpoint, and fix $K>2C$, where $C$ bounds the coefficients in
the curvature expression below. After reducing $d$, the function
$\eta(y)=1-K(y-y_c)^2$ satisfies $1/2\le\eta\le1$. Put
\[
 \lambda=\max_{[y_-,y_+]}\frac{-v}{\eta}>0.
\]
The curve $w=-\lambda\eta$ lies below $v$, with an interior contact. The local Lipschitz bound
for the boundary gives $\lambda\le 2\|v\|_\infty\le Cd$, and
\[
 \begin{aligned}
 |w|&\le\lambda,\\
 |w'|&\le\lambda Kd,\\
 w''&=2K\lambda.
 \end{aligned}
\]
Write $Y=\partial_y+w'\partial_n$. The unit normal pointing
toward increasing $n$ has metric dual
$J(dn-w'\,dy)/(J^2+(w')^2)^{1/2}$. Using the three Christoffel
symbols above, the signed geodesic curvature is
\[
 \kappa_{\mathrm{opt}}
 =\frac{J}{(J^2+(w')^2)^{3/2}}
       \bigl\{w''-JJ_n-(J_y/J)w'-2(J_n/J)(w')^2\bigr\}.
\]
All coefficients here are evaluated at $(y,w(y))$.
Reduce $d$ further, keeping $K$ fixed, so that
$|w'|\le\lambda Kd\le1$. Since $|w|\le\lambda$, the expression
in braces is bounded below by
\[
 2K\lambda-C\lambda\{1+|w'|+(w')^2\}
       \ge (2K-3C)\lambda>0.
\]
This proves strict positivity with the required orientation.
Lemma~\ref{local:weak-cap} again gives a forbidden supporting
hypersurface. Therefore $v=0$: every sufficiently short chord
between boundary points lies in the boundary, and the boundary
is the smooth geodesic determined by those endpoints.

The optical curvature identity now gives $B_{33}=B_{44}=0$.
We use the null-tangent argument of Foertsch, Hasse, and
Perlick~\cite{FHP03}. For a nonzero tangent null field $X$ on $\Sigma$, the Gauss formula
and the induced connection $\nabla^\Sigma$ give
\begin{equation}\label{local:null-surface-acceleration}
 \nabla_X X=\nabla_X^\Sigma X+B(X,X)n
                   =\nabla_X^\Sigma X.
\end{equation}
Furthermore
$h(\nabla_X^\Sigma X,X)=\tfrac12Xh(X,X)=0$.
The orthogonal complement of $X$ in the Lorentzian tangent plane
is $\RR X$, so $\nabla_XX\in\RR X$, as required.
\end{proof}

Figure~\ref{fig:optical-alternative} separates the two uses of the
optical metric. An exterior excursion contradicts the condition
on short null segments.
Once the chord lies in the epigraph, the quadratic curve in Fermi
coordinates detects strict curvature wherever the boundary differs
from the chord.

\begin{figure}[htbp]
\centering
\begin{tikzpicture}[x=1cm,y=1cm,font=\normalsize,
 line cap=round,line join=round,>=stealth]
 \path[use as bounding box] (0,-.55) rectangle (15.8,4.15);
 \begin{scope}[shift={(2.35,1.35)},x=1.45cm,y=1.8cm]
  \node at (0,1.33) {(a) Exterior excursion};
  \path[fill=black!6]
    plot[domain=-1.2:1.2,samples=71]
      (\x,{.8*(1-\x*\x/1.44)})
      --(1.2,-.25)--(-1.2,-.25)--cycle;
  \draw[densely dotted,thick] (-1.2,0)--(1.2,0);
  \draw[very thick] plot[domain=-1.2:1.2,samples=71]
      (\x,{.8*(1-\x*\x/1.44)});
  \fill (0,0) circle[radius=1.1pt];
  \fill (-1.2,0) circle[radius=1pt];
  \fill (1.2,0) circle[radius=1pt];
  \draw[->,thick] (0,-.42)--(0,-.07);
  \node[fill=white,inner sep=1pt] at (.83,.27) {$v$};
  \node[above,fill=white,inner sep=1pt] at (.48,0) {$c$};
  \node at (0,-.65) {null lift leaves $H$};
 \end{scope}
 \begin{scope}[shift={(7.55,1.35)},x=1.45cm,y=1.8cm]
  \node at (0,1.33) {(b) Strict curvature};
  \path[fill=black!6]
    plot[domain=-1.2:1.2,samples=71] (\x,{.625*\x*\x})
      --(1.2,-.25)--(-1.2,-.25)--cycle;
  \draw[densely dotted,thick] (-1.2,.9)--(1.2,.9);
  \draw[very thick] plot[domain=-1.2:1.2,samples=71]
      (\x,{.625*\x*\x});
  \draw[dashed,thick] plot[domain=-1.2:1.2,samples=71]
      (\x,{.27*\x*\x});
  \fill (0,0) circle[radius=1.1pt];
  \fill (-1.2,.9) circle[radius=1pt];
  \fill (1.2,.9) circle[radius=1pt];
  \node[above] at (0,.91) {$n=0$};
  \node[fill=white,inner sep=1pt] at (-.78,.52) {$v$};
  \node[fill=white,inner sep=1pt] at (.74,-.10) {$-\lambda\eta$};
  \node at (0,-.65) {$\kappa_{\mathrm{opt}}>0$};
 \end{scope}
 \begin{scope}[shift={(12.5,.90)},x=.88cm,y=.88cm]
  \node at (.15,3.25) {(c) Equality};
  \path[fill=black!6] (-1.6,0)--(1.6,0)--(2.05,2.4)--(-1.15,2.4)--cycle;
  \draw[thin] (-1.6,0)--(-1.15,2.4)--(2.05,2.4)--(1.6,0);
  \begin{scope}
   \clip (-1.6,0)--(1.6,0)--(2.05,2.4)--(-1.15,2.4)--cycle;
   \foreach \c in {-1.2,0,1.2}{
    \draw[black!55,thin] ({\c-1.2},0)--({\c+1.65},2.4);
    \draw[black!55,thin] ({\c+1.2},0)--({\c-.75},2.4);
   }
   \draw[->,thick] (.225,1.2)--(1.02,1.87);
   \draw[->,thick] (.225,1.2)--(-.42,1.87);
  \end{scope}
  \fill (.225,1.2) circle[radius=1.3pt];
  \node[fill=white,inner sep=1pt] at (1.1,2.03) {$e_4$};
  \node[fill=white,inner sep=1pt] at (-.54,2.03) {$e_3$};
  \draw[very thick] (-1.6,0)--(1.6,0);
  \node[below=3pt] at (0,0) {$v=0$};
  \draw[->,thick] (2.08,.30)--(2.36,1.75);
  \node[right] at (2.25,1.04) {$-\Z$};
  \node[fill=white,inner sep=1pt] at (-.89,.36) {$\Sigma$};
  \node at (.20,-.83) {$B_{33}=B_{44}=0$};
 \end{scope}
\end{tikzpicture}
\caption{Optical convexity and Killing extension.
(a) A short geodesic $c$ of $q_{\mathrm{opt}}$ with an exterior excursion would
lift to a null geodesic with endpoints in $H$ and interior outside
$H$. The shaded region is the exterior.
(b) In Fermi coordinates about a chord $n=0$, a boundary lying
strictly below it admits the lower touching curve
$-\lambda\eta$, of positive geodesic curvature. Its rotational
images give a hypersurface across which the Killing field extends.
(c) When this possibility is excluded, the boundary equals its
geodesic chord. Its rotational images form a timelike photon surface
$\Sigma$, and both families tangent to $e_3,e_4$ are ambient null
geodesics.}
\label{fig:optical-alternative}
\end{figure}

\subsection{Extension around a complete orbit}

We now show that a strict supporting hypersurface permits a larger
pair in Definition~\ref{discrete:admissible}. To do this we must
extend around a complete orbit, since the orbit estimate requires
a complete rotational action on every domain. We transport the
hypersurface by the limiting old isometries over one period. The
boundary values of $\Z$ and $\D\Z$ identify the local extensions,
and Lemma~\ref{local:compatibility} recovers the complete action.

The defining function of the supporting hypersurface can be made invariant using
\eqref{local:invariant-hessian}. A positive quadratic perturbation
then localizes the portion to be removed. We first record how two
independent normal covectors produce a strict supporting
hypersurface.

\Needspace{8\baselineskip}
\begin{lemma}\label{closure:corner-cap}
If $p\in\partial H\cap\{\g(\T,\T)>0\}$ and
$\dim N^P_{H}(p)=2$, there are a neighborhood $V$ of $p$
and $f\in C^\infty(V)$ with spacelike nonzero differential such that
\begin{align*}
 f(p)&=0,\\
 H\cap V&\subset\{f\ge0\}
\end{align*}
and
\[
 \Hess_hf(X,X)<0
 \qquad\bigl(0\ne X\in\ker df(p),\ h(X,X)=0\bigr).
\]
\end{lemma}
\begin{proof}
By Lemma~\ref{closure:normal-limits}, $N^P_{H}(p)$ lies in the
plane of covectors annihilating the common timelike field, on which
$h^{-1}$ is positive definite. Take independent normal covectors $\alpha,\beta$ and
$h$-normal coordinates $x$ centered at $p$. Regard the covectors
as linear functions in these coordinates. Equation~\eqref{normal:conormal-inequality} gives
\[
 \max\{\alpha(x),\beta(x)\}\le C|x|^2
 \qquad(x\in H).
\]
The line $\ker\alpha\cap\ker\beta$ is timelike. Consequently
\[
 \min_{\substack{|X|=1,\ h(X,X)=0\\(\alpha+\beta)(X)=0}}
                 |(\alpha-\beta)(X)|>0.
\]
Choose $M$ sufficiently large and put
\[
 f(x)=-(\alpha+\beta)(x)+(2C+1)|x|^2
                                  -M(\alpha-\beta)(x)^2.
\]
The differential at $p$ is $-(\alpha+\beta)\ne0$, and is spacelike.
Since the Christoffel symbols vanish at $p$,
\[
 \Hess_hf(X,X)=2(2C+1)|X|^2
                         -2M\bigl((\alpha-\beta)(X)\bigr)^2<0
\]
on its nonzero null tangents.

For $x\in H$, the same two normal inequalities give
\begin{align*}
 |(\alpha-\beta)(x)|
 &\le-(\alpha+\beta)(x)+2C|x|^2,\\
 M(\alpha-\beta)(x)^2
 &\le MC_1|x|\bigl(-(\alpha+\beta)(x)+2C|x|^2\bigr).
\end{align*}
Here $C_1$ bounds the norm of $\alpha-\beta$.
Reduce the coordinate ball so that $MC_1|x|\le1/2$. Then
\begin{align*}
 f(x)&\ge(1-MC_1|x|)(-(\alpha+\beta)(x))
           +(2C+1-2CMC_1|x|)|x|^2\\
 &\ge\bigl[-2C(1-MC_1|x|)+2C+1-2CMC_1|x|\bigr]|x|^2
 =|x|^2.
\end{align*}
Thus its entire negative side is in the existing domain.
\end{proof}

\begin{proposition}\label{closure:orbit-tube}
Let $p\in\partial H\cap\{\g(\T,\T)>0\}$ admit a
smooth strictly $\T$-conditionally pseudoconvex supporting
hypersurface. There is a connected stationary invariant neighborhood
$V\subset\E$ of $\pi^{-1}(p)$ carrying a commuting Killing extension
$\widetilde\Z$ of complete effective period $2\pi$, with
\[
 \widetilde\Z=\Z\quad\text{on }V\cap U_{H}.
\]
Its flow preserves $H\cap\pi(V)$.
\end{proposition}
\begin{proof}
Put $U=U_{H}$. Let $f$ define the supporting hypersurface,
with $\{f<0\}\subset\pi(U)$.
Since $f(p)=0$ and $f\ge0$ on $H$ near $p$, Taylor's
formula in a fixed coordinate chart gives
\[
 -df_p(x-p)\le C|x-p|^2
 \qquad(x\in H\text{ near }p).
\]
Thus $-df_p\in N^P_{H}(p)\setminus\{0\}$.
Lemma~\ref{closure:normal-limits} and
Definition~\ref{def:stationary-metrics} give
\[
 \g^{-1}(\pi^{\ast}df,\pi^{\ast}df)
   =\g(\T,\T)h^{-1}(df,df)>0.
\]
Thus the conormal is spacelike, as required in
Proposition~\ref{local:extension}. If the normal cone $N_H^P(p)$
were a line, Lemma~\ref{closure:graph-band} would
give $H=\{a\le s\le b\}$ with $a(p)=b(p)$, and
Lemma~\ref{discrete:weak-null-tests} would exclude the strict
support. Thus Lemma~\ref{closure:normal-limits} gives one limit
of $(\Z,\D\Z)$ from all nearby exterior components.

To carry the supporting hypersurface around the orbit, we first extend the old flow
maps smoothly to the boundary. We work on $Q$ and write $\Phi_t$
for the projected old flow. The complete orbits stay in a fixed compact subset of the ergoregion, by the confinement argument following
Lemma~\ref{closure:null-reach}. A finite cover of its boundary part
by Lemma~\ref{closure:normal-limits} bounds every derivative of $\Z$
there. The remaining exterior part has compact closure in $\pi(U)$,
where $\Z$ is smooth. We thus have uniform coefficient bounds
along the full orbit interval in the variational equations.
Here $D$ denotes coordinate differentiation of the projected field
and its flow:
\begin{align*}
 \frac d{dt}D\Phi_t&=(D\Z)_{\Phi_t}D\Phi_t,\\
 \frac d{dt}D^2\Phi_t
 &=(D^2\Z)_{\Phi_t}(D\Phi_t,D\Phi_t)
                         +(D\Z)_{\Phi_t}D^2\Phi_t.
\end{align*}
Gronwall's inequality and successive differentiation therefore give
\[
 \sup_{|t|\le2\pi,\ x\in\{f<0\}}
       \bigl(|D^k\Phi_t(x)|+|(D\Phi_t(x))^{-1}|\bigr)\le C_k
 \qquad(k\ge1)
\]
on a smaller smooth half-neighborhood of $p$.
In coordinates $(z,r)$ flattening $f=0$, these bounds imply
\[
 |D^k\Phi_t(z,-r)-D^k\Phi_t(z,-s)|\le C_{k+1}|r-s|.
\]
Thus the maps and all their derivatives have smooth one-sided
limits on $f=0$, jointly with $t$. We use $\Phi_t(x)$ and
$d\Phi_t|_x$ also for these limits.

Let $n$ be the spacelike unit normal toward $f>0$, and let
$\exp^h$ be the exponential map of $h$. Define a map on a full
neighborhood of $f=0$ by
\begin{equation}\label{closure:extended-flow-map}
 I_t\bigl(\exp_x^h(rn_x)\bigr)
   =\exp_{\Phi_t(x)}^h\bigl(r\,d\Phi_t|_x n_x\bigr).
\end{equation}
For $r<0$, isometry invariance of the geodesic equation in
\eqref{closure:extended-flow-map} gives
\[
 I_t=\Phi_t\quad\Longrightarrow\quad I_t^{\ast}h=h.
\]
At $r=0$ we have $dI_t=d\Phi_t$, which is invertible.
The inverse function theorem gives a uniform smaller neighborhood
on which $I_t$ is a diffeomorphism for $|t|\le2\pi$.
Set $\gamma(t)=\Phi_t(p)$ and $f_t=f\circ I_t^{-1}$.
We use the isometry identity on the negative side to compute the
second derivatives of $I_t$ at the boundary, where we must preserve the
strict Hessian inequality. Write $\Gamma^k_{ij}$ for the Christoffel
symbols of $h$. Preservation of the connection on the negative side
reads
\[
 \partial_i\partial_j I_t^k
 +\Gamma^k_{ab}(I_t)\partial_iI_t^a\partial_jI_t^b
 -\Gamma^a_{ij}\partial_aI_t^k=0.
\]
Smoothness gives the same equation at $p$, or
$(\nabla dI_t)_p=0$, where the covariant derivative uses $h$
on both the source and target. Differentiating $f=f_t\circ I_t$ twice,
we obtain
\begin{align*}
 \Hess_hf(X,X)|_p
 &=\Hess_hf_t(dI_tX,dI_tX)|_{\gamma(t)}
       +df_t\bigl((\nabla dI_t)(X,X)\bigr)|_p\\
 &=\Hess_hf_t(dI_tX,dI_tX)|_{\gamma(t)}.
\end{align*}
Moreover, $dI_t$ preserves $h$ at $p$ and
\[
 \{f_t<0\}=I_t(\{f<0\})\subset\pi(U).
\]
Thus the null tangents and their strict Hessian inequality are
transported together, and $\{f_t=0\}$ is a strictly pseudoconvex supporting
hypersurface at $\gamma(t)$.

For $p_j\in\{f<0\}$ approaching $p$, we have
$\Phi_t(p_j)\to\gamma(t)$.
If $\gamma(t)\in\pi(U)$, continuity of the old inverse flow gives
\[
 p=\lim_j\Phi_{-t}(\Phi_t(p_j))
                       =\Phi_{-t}(\gamma(t))\in\pi(U),
\]
a contradiction. Thus
\begin{align*}
 \gamma([0,2\pi])&\subset\partial H,\\
 \gamma(2\pi)&=\gamma(0)=p.
\end{align*}
We can now apply the local extension theorem all along the compact
curve. Each point has such a supporting hypersurface, so
$N^P_{H}(\gamma(t))$ is not a line. Combining the local theorem with
Lemma~\ref{normal:parallel-boundary}, we obtain local Killing fields with
\[
 (\widetilde\Z,\D\widetilde\Z)|_{\gamma(t)}
   =\lim_{\substack{x\to\gamma(t)\\x\in\pi(U)}}(\Z,\D\Z)(x),
\]
where the limits of the spacetime field and its covariant derivative
are taken on the fixed stationary section.
Each field agrees with $\Z$ throughout its intersection with $U$.
Apply the first part of Lemma~\ref{local:compatibility} to the
compact set on $\vartheta=0$ above $\gamma([0,2\pi])$.
It gives one stationary Killing extension near that set, agreeing
with $\Z$ on the whole old overlap. The curve $\gamma$ is an
integral curve of its projection, since differentiation of the
one-sided flow limits gives
\[
 \dot\gamma(t)=\lim_{x\to p,\ f(x)<0}\Z(\Phi_t(x))
                           =\widetilde\Z(\gamma(t)).
\]
Together with $\gamma(2\pi)=\gamma(0)$, this proves invariance
of the compact projected curve. The second part of
Lemma~\ref{local:compatibility} therefore provides a stationary
invariant neighborhood $V$ with complete period $2\pi$.
A smaller effective period would be the identity on a nonempty
open subset of the connected old domain and, by isometry
uniqueness, throughout it, contrary to the initial normalization.

It remains to show that the complete circle action preserves the remaining
set. Lemma~\ref{local:complete-flows}, applied to the complete
flows on $U$ and $V$, gives
\[
 \Phi_t^{\widetilde\Z}(U\cap V)=U\cap V.
\]
Choose $V$ over $\operatorname{Int}(\mathcal D_0)$ in the ergoregion. Since
\[
 V\setminus U=V\cap\pi^{-1}(H),
\]
the inverse maps prove the asserted preservation of $H$.
\end{proof}

\begin{lemma}\label{local:invariant-cut}
If $p\in\partial H\cap\{\g(\T,\T)>0\}$ admits a smooth strictly
$\T$-conditionally pseudoconvex supporting hypersurface, there is
an open set $A\ni p$, with $\overline A$ contained in the open
ergoregion, for which $H'=H\setminus A$ and a field extending
$\Z$ on $U_H$ satisfy Definition~\ref{discrete:admissible}.
\end{lemma}
\begin{proof}
We first make the defining function invariant under the extended
circle action, and then localize the removal in a fixed neighborhood.
By Proposition~\ref{closure:orbit-tube}, the projected orbit $O\subset Q$
through $p$ has a complete invariant neighborhood whose inverse image
carries a Killing extension of $\Z$. The union of this spacetime
neighborhood with $U_{H}$ is connected, so we may apply
Lemma~\ref{action:area} to obtain
\begin{align*}
 \T^\flat\wedge\Z^\flat\wedge d\T^\flat&=0,\\
 \T^\flat\wedge\Z^\flat\wedge d\Z^\flat&=0.
\end{align*}
Thus the spacelike distribution $\{\T,\Z\}^{\perp_\g}$ is
integrable throughout the extension neighborhood. We work on its
stationary quotient and write $\Psi_t$ for the projected flow.
The generator is timelike for $h$, so its stabilizer is finite.
If $t$ has order $m$ in this stabilizer, then for a spacetime point
$q$ above $p$ there is $a\in\RR$ with
\[
 \Phi_t^\Z(q)=\Phi_a^\T(q)
 \quad\Longrightarrow\quad
 q=\Phi_{mt}^\Z(q)=\Phi_{ma}^\T(q).
\]
The stationary action is free, so $a=0$.
Thus the stabilizer fixes $q$. Its differential fixes $\T_q,\Z_q$
and preserves the spacetime orientation, so it acts by rotations on
the spacelike plane $\{\T,\Z\}^{\perp_\g}$.
The representation is faithful: a trivial rotation fixes this plane
as well as $\T,\Z$, so the isometry fixes $q$ and has identity differential there. It
would give a smaller period on the old domain.

The stabilizer preserves $N^P_{H}(p)$ in the dual of
the transverse plane.
A nontrivial finite rotation cannot preserve a nonzero proper convex
cone unless that cone is a line. The line case was excluded in the
proof of Proposition~\ref{closure:orbit-tube}. Thus either the orbit
is free or the cone is the whole dual of the transverse plane.

Suppose first that this cone is proper, so the orbit is free.
Let $f$ define the supporting hypersurface at $p$. We shall make
$f$ invariant while preserving its strict Hessian inequality on
null tangent vectors.
Since $O\subset H$ and $f\ge0$ there near $p$,
\begin{align*}
 \Z f(p)&=0,\\
 \Z^2f(p)&\ge0.
\end{align*}
Take a horizontal disk $S$ through $p$ for the static expression
$h=-\rho^2d\phi^2+q$, where $q=h|_{TS}$ and
$\rho^2=-h(\Z,\Z)$.
A neighborhood of the entire orbit is $S^1\times S$.
Extend $f|_S$ constantly along the circles and call the resulting
function $\widetilde f$. Then
\begin{align*}
 d\widetilde f(p)&=df(p),\\
 \widetilde f|_{H}&\ge0.
\end{align*}

Equation~\eqref{local:invariant-hessian}, applied with the extended
Killing field $\Z$, gives
$\Hess_h\widetilde f(e_a,e_a)<0$ for $a=3,4$, since
$\Z^2f(p)\ge0$. The circle isometries transport this inequality
around $O$.
Thus circularity and the minimum of $f$ on the orbit give an
invariant supporting hypersurface satisfying the strict Hessian inequality. We now localize the portion to be removed,
keeping it away from the edge of one fixed orbit neighborhood.
Choose coordinates $(s,y)$ on $S$ with $\widetilde f=s$ and
$p=(0,0)$. For $\delta>0$ sufficiently small put
\[
 f_\delta=s+\delta(s^2+y^2).
\]
On $O$ its differential is unchanged, and its Hessian differs from that
in \eqref{local:invariant-hessian} by
$\delta\Hess_h(s^2+y^2)$. To make the choice uniform, consider
the compact set
\[
 \{(x,X):x\in O,\ |X|_{g_+}=1,\ h(X,X)=0,\ ds(X)=0\}.
\]
On this set $-\Hess_hs(X,X)$ has a positive minimum $\mu$,
and $|\Hess_h(s^2+y^2)(X,X)|$ has a finite maximum $C_0$.
Choose $\delta C_0<\mu/4$. Since $df_\delta=ds$ on $O$,
\[
 -\Hess_h f_\delta(X,X)\ge 3\mu/4
       \qquad(x\in O,\ |X|_{g_+}=1,\ h(X,X)=df_\delta(X)=0).
\]
The conormal remains spacelike near $O$. Continuity of the two
null tangent lines then gives, on a smaller invariant neighborhood,
\[
 -\Hess_h f_\delta(X,X)\ge(\mu/2)|X|_{g_+}^2
       \qquad(h(X,X)=df_\delta(X)=0).
\]
Choose this neighborhood as
\[
 V=S^1\times\{|s|<R,\ |y|<R\},\qquad \delta R<1/8,
\]
with closure in the extension neighborhood. Since $s\ge0$ on
$H\cap V$,
\begin{equation}\label{local:cut-margin}
 f_\delta\ge\delta(s^2+y^2)
                         \quad\text{on }H\cap\overline V.
\end{equation}
Choose $0<\epsilon<\delta R^2/64$ and set
\[
 A=V\cap\{f_\delta<\epsilon\}.
\]
Its boundary inside $V$ is a graph, since
\begin{align*}
 \partial_s f_\delta&=1+2\delta s>1/2,\\
 f_\delta(-R,y)&<0<\epsilon<f_\delta(R,y).
\end{align*}
Thus $A$ is connected and meets the old domain, since its strip
$s=-R/2$, $|y|<R/2$ lies below the supporting graph and satisfies
\[
 f_\delta\le-R/2+\delta R^2/2<-R/4<\epsilon.
\]
On the remaining set, \eqref{local:cut-margin} gives directly
\[
 \overline{H\cap A}
 \subset\{s^2+y^2\le\epsilon/\delta\}\Subset V.
\]
The field on the invariant neighborhood agrees with the old field
on its entire overlap by Proposition~\ref{closure:orbit-tube}.
The invariant function $f_\delta$ and the preceding compact
containment therefore satisfy every hypothesis of
Lemma~\ref{discrete:admissible-update}. We obtain the required
pair on $U_{H\setminus A}$.

We finally treat the case where $N^P_{H}(p)$ is the
whole dual of the transverse plane. Here we will remove a tube around the entire
orbit, including the case of a finite stabilizer. Take a small
transverse disk through $p$, with coordinates $x$ and tangent space
$\Z_p^{\perp_h}$. Use the Euclidean coordinate norm on this disk and
its dual norm on covectors. Taking the supremum in its normal inequalities
and absorbing the quadratic coordinate error gives
\[
 |x|=\sup_{|\alpha|=1}\alpha(x)\le C|x|^2
                  \qquad(x\in H\cap S).
\]
Thus $H\cap S=\{p\}$ after shrinking. Taking the images
of the disk under the circle action
gives a connected invariant neighborhood $V$ with
\begin{align*}
 H\cap\overline V&=O,\\
 H\cap\partial V&=\varnothing.
\end{align*}
The transverse punctured disk lies in the old domain, so $V$
meets $\pi(U_H)$. The entire-component case of
Lemma~\ref{discrete:admissible-update}, with $A=V$, completes the
construction.
\end{proof}

We will also need to identify the two boundary surfaces when the
lower and upper graphs meet. Once their second fundamental forms vanish on null tangents,
we can do this using the two ambient null geodesic families: a
shared null geodesic gives common tangent planes, and the second
family then fills a common neighborhood. We record this argument
to complete the proof of Proposition~\ref{discrete:stalled-boundary}.

\begin{lemma}\label{closure:line-stalemate}
Let $\Sigma,\widetilde\Sigma$ be smooth timelike surfaces in
$(Q,h)$ whose second fundamental forms vanish on null tangents.
If they are locally ordered graphs tangent at $p$, they coincide
near $p$.
\end{lemma}
\begin{proof}
Write the graph heights as $v\le\widetilde v$. We use the
induced connections on the two surfaces and either choice of their
unit spacelike normals in the Gauss formula. We first follow a
common nonzero null tangent $L_p$ at $p$. By \eqref{local:null-surface-acceleration}, an affinely parametrized
null geodesic for either induced metric is also an ambient affinely
parametrized geodesic. Uniqueness therefore puts $\gamma(s)=\exp_p^h(sL_p)$ in both
surfaces. Along this curve,
\[
 \widetilde v-v=0\quad\Longrightarrow\quad D(\widetilde v-v)=0,
\]
because the difference is nonnegative. Their tangent planes agree.
Set $L(s)=\dot\gamma(s)$, and let $\underline L(s)$ be the
other common null vector, normalized by $h(L,\underline L)=-1$. The same application of \eqref{local:null-surface-acceleration}
shows that
\[
 \Psi(s,t)=\exp_{\gamma(s)}^h(t\underline L(s))
                        \in\Sigma\cap\widetilde\Sigma.
\]
At $(0,0)$ the differential columns are $L_p$ and $\underline L(0)$.
Their pairing $h(L_p,\underline L(0))=-1$ makes them independent,
so $\Psi$ has rank two and its image is an open surface in both
$\Sigma$ and $\widetilde\Sigma$. This proves local coincidence.
\end{proof}

\begin{proof}[Proof of Proposition~\ref{discrete:stalled-boundary}]
A strict supporting hypersurface at
$p\in\partial H\cap\{\g(\T,\T)>0\}$ would give, by
Lemma~\ref{local:invariant-cut}, a pair $(H',\Z')$ with
\[
 \begin{aligned}
 H'&=H\setminus A\subsetneq H,\\
 \Z'|_{U_H}&=\Z.
 \end{aligned}
\]
This contradicts maximality. Thus no boundary point in the open
ergoregion admits such a hypersurface.

If $N_H^P(p)$ is not a line, Lemma~\ref{closure:normal-limits}
puts it in the spacelike plane annihilating one timelike vector.
Lemma~\ref{closure:corner-cap} excludes dimension two. Since every
boundary point has a nonzero normal covector, its normal cone is
a ray or a line.

In the ray case, Lemma~\ref{closure:graph-band} gives a one-sided
boundary graph. Extend the optical metric smoothly from its exterior by
Lemma~\ref{local:optical-reduction}. The absence of strict supports
and Lemma~\ref{local:geodesic-stalemate} make its transverse curve
a smooth geodesic of the optical metric. Its rotational images are therefore a
smooth timelike hypersurface with
\[
 B_{33}=B_{44}=0.
\]

In the line case, write $H=\{a\le s\le b\}$ as in
Lemma~\ref{closure:graph-band}. We use the lower and upper exterior
components separately. Where the graphs differ, a supporting
hypersurface for either face is a supporting hypersurface for $H$
in a smaller neighborhood. At a point where they meet, the
calculation from the opposite exterior in Lemma~\ref{discrete:weak-null-tests}
excludes a strict support. Thus
Lemma~\ref{local:geodesic-stalemate} applies from each exterior,
with its own optical metric, and makes both faces smooth timelike
photon surfaces. They are ordered and tangent at $p$, so
Lemma~\ref{closure:line-stalemate} identifies them near $p$.

Every null tangent in a Lorentzian two-plane is a multiple of
$e_3$ or $e_4$. The displayed vanishing therefore holds for every
null tangent, proving the proposition.
\end{proof}
\section{Nontrapping and continuation of the symmetry}
\label{sec:remaining}

Let $(H,\Z)$ be a maximal pair from
Lemma~\ref{discrete:maximal-action}. Thus $H\subset H_0$ is compact,
and $\Z$ generates the complete rotational action on the connected
exterior $U_H$ defined in~\eqref{discrete:exact-domain}.
Proposition~\ref{discrete:stalled-boundary} shows that
$\partial H\cap\{\g(\T,\T)>0\}$ is a smooth timelike surface
of the stationary quotient, whose second fundamental form
vanishes on null tangent vectors. Its null curves lift to spacetime null
geodesics orthogonal to $\T$. To prove
Theorem~\ref{discrete:continuation}, we will continue one such
geodesic until its length contradicts Lemma~\ref{action:escape}.
The issue is its possible contact with the ergosurface.

\subsection{The geodesic equation at the ergosurface}

At an ergosurface point, a null vector orthogonal to $\T$ is a
multiple of $\T$, so the derivatives of $\g(\T,\T)$ and of
the twist potential $\psi$ along the geodesic vanish at contact. We use these invariants as
transverse coordinates. The common level surfaces carry a flat
Lorentzian metric, which supplies a smooth local frame even
where the rotational field has not yet been extended. Circularity
then removes the term linear in transverse velocity from the
geodesic equation. The resulting equation is smooth at the
ergosurface, and uniqueness forces the transverse curve to
retrace its incoming arc. We include the reconstruction of the
spacetime curve, since its transverse projection alone would
not establish continuation in $\partial H$.

\begin{lemma}\label{local:brake-reflection}
Let $(H,\Z)$ be the maximal pair above. Suppose that the
zero-energy null geodesic $\gamma:(s_-,s_0)\to\E$ is parametrized
by $\mathbf g_+$-arclength, that $s_0<\infty$, and that
\[
 \pi(\gamma(s))\in\partial H.
\]
If
\[
 \begin{aligned}
 \lim_{s\uparrow s_0}(\gamma(s),\dot\gamma(s))&=(p_0,v_0),\\
 \g(\T,\T)(p_0)&=0,
 \end{aligned}
\]
then $\gamma$ continues through $s_0$ with projection in
$\partial H$. Moreover, $\g(\T,\T)(\gamma(s))>0$ for all
sufficiently small $0<|s-s_0|$. In an affine parameter $\tau$
with $\tau(s_0)=0$, its transverse coordinates satisfy
\[
 (\g(\T,\T),\psi)(\gamma(\tau))
       =(\g(\T,\T),\psi)(\gamma(-\tau)).
\]
\end{lemma}
\begin{proof}
Equation~\eqref{action:normalized-flow} continues the geodesic
through $(p_0,v_0)$. Since $\T$ is nonzero and null at $p_0$,
we have $v_0=\lambda\T_{p_0}$ for a nonzero $\lambda$.
Since $\T[\g(\T,\T)]=0$ and
$\D_{\dot\gamma}\dot\gamma$ is proportional to $\dot\gamma$,
\eqref{local:ergosurface-convexity} gives
\begin{align*}
 \frac{d}{ds}\g(\T,\T)(\gamma(s))\big|_{s_0}&=0,\\
 \frac{d^2}{ds^2}\g(\T,\T)(\gamma(s))\big|_{s_0}
 &=\frac{\lambda^2}{2}|d[\g(\T,\T)]|_\g^2\big|_{p_0}>0.
\end{align*}
Consequently
\begin{equation}\label{local:isolated-contact}
 \g(\T,\T)(\gamma(s))
 =\frac{\lambda^2}{4}|d[\g(\T,\T)]|_\g^2\big|_{p_0}
       (s-s_0)^2+o((s-s_0)^2).
\end{equation}
In particular, the final incoming interval lies in the open
ergoregion. Choose an affine parameter $\tau$ with $\tau(s_0)=0$ and
$(d\tau/ds)(s_0)=1$, so the contact velocity remains $v_0$.
Below a dot denotes $d/d\tau$.

By Lemma~\ref{canon:regularity}, the gradients of
$x^1=\g(\T,\T)$ and $x^2=\psi$ span a spacelike plane in a
neighborhood of $p_0$. Their common level surfaces are therefore
Lorentzian. On any one of these surfaces, $\T$ is Killing and
has constant length. If $\nabla^C$ is its intrinsic connection,
then $A(X)=\nabla^C_X\T$ is skew-adjoint and
\[
 \g(AX,\T)=\tfrac12X\g(\T,\T)=0.
\]
Skew-adjointness then gives
\[
 A\T=0.
\]
A nonzero skew-adjoint endomorphism of a Lorentzian two-plane
is invertible. Hence $A=0$. Since the intrinsic curvature
annihilates the nonzero parallel field $\T$, the curvature
of the surface vanishes as well. Along a smooth transversal,
choose a second tangent vector to each level surface, independent
of $\T$, and extend it by intrinsic parallel transport. On a small
simply connected level-surface patch this gives a smooth field
$Y$, independent of $\T$, with
\[
 \begin{aligned}
 \nabla^C Y&=0,\\
 [\T,Y]&=0.
 \end{aligned}
\]
We thus obtain local coordinates $(x^1,x^2,\theta,\phi)$ with
$\T=\partial_\theta$ and $Y=\partial_\phi$. Their coordinate
domain is a product $B\times P$, where $B$ is a rectangle in the
$x$ variables and $P$ is a rectangle about $(0,0)$ in $(\theta,\phi)$, with $\phi(p_0)=0$.

On the part of $U_H$ with $x^1>0$, the actual field $\Z$
preserves $x^1,x^2$ by circularity and is independent of $\T$
by Lemma~\ref{action:area}. Its restriction to a level surface
is Killing. Commutation with the parallel field $\T$ implies
that its intrinsic skew derivative annihilates $\T$, so the
same two-dimensional argument gives
\[
 \Z=a(x)Y+b(x)\T,
 \qquad a(x)\ne0
\]
on each connected level-surface patch in $U_H$. Here and below the
coefficients are defined only where $\Z$ is already known.
Completeness of the action on $U_H$ makes each such patch
contain the whole coordinate rectangle $P$. Indeed, at a
point with coefficients $a_0,b_0$, the complete field
$(\Z-b_0\T)/a_0$ agrees with $Y$ along that level surface,
and its flow is explicitly
\[
 \Phi_t^{(\Z-b_0\T)/a_0}(q)
 =\Phi_{t/a_0}^{\Z}\circ\Phi_{-b_0t/a_0}^{\T}(q)
                   \qquad(t\in\RR).
\]
The right-hand side belongs to $U_H$ for every $t$.
A finite segment of its flow cannot leave $U_H$ while
remaining in the coordinate patch: at a first exit, completeness
would give a point of $U_H$ with the same limit, contrary to local
uniqueness. Applying this to $Y$ and to $\T$ proves that, for an
open set $O\subset B\cap\{x^1>0\}$,
\begin{equation}\label{local:plaque-invariance}
 U_H\cap\{x^1>0\}\cap(B\times P)=O\times P.
\end{equation}
In particular, $a,b$ are smooth functions on $O$. Taking boundaries
in $B\cap\{x^1>0\}$ and using \eqref{discrete:exact-domain} gives
\[
 \pi^{-1}(\partial H)\cap\{x^1>0\}\cap(B\times P)
 =\partial_{B\cap\{x^1>0\}}O\times P.
\]
Thus, if one point over $x$ lies in this boundary, every point
of $\{x\}\times P$ does. This is the property needed to reconstruct
the outgoing spacetime curve from its transverse coordinates.

We now compute the geodesic equation in this smooth local frame.
Use capital indices $A,B\in\{\theta,\phi\}$ for the level-surface
coordinates $y=(\theta,\phi)$ and lower-case indices $i,j\in\{1,2\}$
for $x$. Completing the square gives
\begin{equation}\label{local:regular-metric}
 \g=G_{AB}(x)(dy^A+\omega_i^A dx^i)(dy^B+\omega_j^B dx^j)
                  +\sigma_{ij}(x,\phi)dx^i dx^j.
\end{equation}
The matrix $G$ is the metric of the parallel frame $(\T,Y)$,
so it depends only on $x$. It is Lorentzian, and
\[
 W[\T,Y]=-\det G=\g(\T,Y)^2-x^1\g(Y,Y)>0.
\]
All the coefficients are smooth across $x^1=0$. With
$(G^{AB})=G^{-1}$, the remaining coefficients are determined by
\begin{align*}
 \omega_i^A&=G^{AB}\g(\partial_{y^B},\partial_{x^i}),\\
 \sigma_{ij}&=\g(\partial_{x^i},\partial_{x^j})
                         -G_{AB}\omega_i^A\omega_j^B.
\end{align*}
The matrix $\sigma$ is positive definite. Stationarity removes
all $\theta$ dependence, but $Y$ need not be an ambient Killing
field.

On $O\times P$, the isometry generated by $\Z$ preserves the
horizontal fields $E_i=\partial_{x^i}-\omega_i^A\partial_{y^A}$.
Expanding their commutator with $aY+b\T$ gives
\begin{align*}
 0=[\Z,E_i]
 &=-\bigl(a\partial_\phi\omega_i^\phi+\partial_i a\bigr)Y
   -\bigl(a\partial_\phi\omega_i^\theta+\partial_i b\bigr)\T,\\
 0=(\LL_\Z\g)(E_i,E_j)&=a\partial_\phi\sigma_{ij}.
\end{align*}
Circularity, in the form \eqref{action:circularity}, says that
the horizontal distribution is integrable. Thus
\begin{align*}
 [E_i,E_j]&=-\Omega_{ij}^A\partial_{y^A}=0,\\
 \Omega_{ij}^A
 &=\partial_i\omega_j^A-\partial_j\omega_i^A
   +\omega_j^\phi\partial_\phi\omega_i^A
   -\omega_i^\phi\partial_\phi\omega_j^A.
\end{align*}
Define a Riemannian metric $\sigma^0$ and a one-form $\alpha$ on
$B$ by evaluation on the fixed section $\phi=0$:
\[
 \begin{aligned}
 \sigma^0_{ij}(x)&=\sigma_{ij}(x,0),\\
 \alpha_i(x)&=-\partial_\phi\omega_i^\phi(x,0).
 \end{aligned}
\]
The commutator identity gives
$\partial_\phi\omega_i^\phi=-\alpha_i$ on $O\times P$.
Since the metric coefficients are smooth on $B\times P$, the
identities on $O\times P$ and their $x$ derivatives pass to
its closure. In particular, on that closure,
\begin{equation}\label{local:boundary-coefficients}
 \begin{aligned}
 \sigma_{ij}&=\sigma^0_{ij},\\
 \partial_k\sigma_{ij}&=\partial_k\sigma^0_{ij},\\
 \partial_\phi\sigma_{ij}&=0,\\
 \partial_\phi\omega_i^\phi&=-\alpha_i,\\
 \Omega_{ij}^A&=0.
 \end{aligned}
\end{equation}
These identities apply along the incoming boundary curve and
along every level-surface patch above it.

For the metric \eqref{local:regular-metric}, let $v^i=\dot x^i$ and put
\[
 p_A=\g(\partial_{y^A},\dot\gamma)
     =G_{AB}(\dot y^B+\omega_i^Bv^i).
\]
The stationary momentum is zero. Writing
$\ell=p_\phi=\g(Y,\dot\gamma)$, we obtain
\begin{equation}\label{local:brake-lift}
 \begin{aligned}
 \dot y^A&=G^{A\phi}\ell-\omega_i^Av^i,\\
 G^{\phi\phi}&=-\frac{x^1}{W[\T,Y]}.
 \end{aligned}
\end{equation}
The $\phi$ Euler--Lagrange equation for
$\mathcal L=\frac12\g(\dot\gamma,\dot\gamma)$ is
\begin{align*}
 \dot\ell
 &=\tfrac12\partial_\phi\sigma_{ij}v^iv^j
               +\ell\partial_\phi\omega_i^\phi v^i\\
 &=-\ell\alpha_i v^i.
\end{align*}
This momentum varies because $Y$ is only intrinsically parallel.
To compute the transverse equation, write
$w^A=\dot y^A+\omega_i^Av^i=G^{A\phi}\ell$. Then
\[
 -\partial_iG_{AB}w^Aw^B
 =-\ell^2G^{A\phi}(\partial_iG_{AB})G^{B\phi}
 =\ell^2\partial_iG^{\phi\phi}.
\]
The $x^i$ Euler--Lagrange equation is therefore
\begin{equation}\label{local:transverse-euler-lagrange}
\begin{aligned}
 0&=\frac d{d\tau}(\sigma_{ij}v^j+\ell\omega_i^\phi)
    -\frac12\partial_i\sigma_{jk}v^jv^k
    -\frac12\partial_iG_{AB}w^Aw^B
    -\ell\partial_i\omega_j^\phi v^j\\
 &=\sigma^0_{ij}(\nabla^{\sigma^0}_v v)^j
    +\frac{\ell^2}{2}\partial_iG^{\phi\phi}
    +\dot\ell\omega_i^\phi
    +\ell(\partial_j\omega_i^\phi-\partial_i\omega_j^\phi)v^j
    +\ell\partial_\phi\omega_i^\phi\dot\phi\\
 &=\sigma^0_{ij}(\nabla^{\sigma^0}_v v)^j
    +\frac{\ell^2}{2}\partial_iG^{\phi\phi}
    -\ell\Omega_{ij}^\phi v^j
    +\ell^2\partial_\phi\omega_i^\phi G^{\phi\phi}.
\end{aligned}
\end{equation}
The cancellation in the last line is
\[
\begin{aligned}
 \dot\ell\omega_i^\phi
    +\ell\partial_\phi\omega_i^\phi\dot\phi
 &=\ell\omega_i^\phi\partial_\phi\omega_j^\phi v^j
   +\ell\partial_\phi\omega_i^\phi
           (G^{\phi\phi}\ell-\omega_j^\phi v^j)\\
 &=\ell^2\partial_\phi\omega_i^\phi G^{\phi\phi}
   +\ell(\omega_i^\phi\partial_\phi\omega_j^\phi
          -\omega_j^\phi\partial_\phi\omega_i^\phi)v^j.
\end{aligned}
\]
Together with
$\ell(\partial_j\omega_i^\phi-\partial_i\omega_j^\phi)v^j$,
the last term is precisely $-\ell\Omega_{ij}^\phi v^j$.
It vanishes by circularity. Substituting
\eqref{local:boundary-coefficients} leaves the smooth system
\begin{equation}\label{local:mechanical-reflection}
 \begin{aligned}
 \nabla^{\sigma^0}_v v
 &=\frac{\ell^2}{2}\operatorname{grad}_{\sigma^0}
                 \left(\frac{x^1}{W[\T,Y]}\right)
   -\frac{\ell^2x^1}{W[\T,Y]}\alpha^\sharp,\\
 \dot\ell&=-\ell\alpha(v).
 \end{aligned}
\end{equation}
Here $\alpha^\sharp$ is taken with respect to $\sigma^0$.

At contact the initial data are
\[
 \begin{aligned}
 x(0)&=x(p_0),\\
 v(0)&=0,\\
 \ell(0)&=\lambda\g(\T,Y)(p_0)\ne0.
 \end{aligned}
\]
The last inequality follows because $\T,Y$ span a Lorentzian
plane and $\T$ is null there. The equations
in~\eqref{local:mechanical-reflection}, together with $\dot x=v$,
form a smooth first-order system for $(x,v,\ell)$ with $x\in B$. The coordinate
equation for $\dot v^i$ has right-hand side
\[
 -\Gamma(\sigma^0)^i_{jk}(x)v^jv^k
 +\frac{\ell^2}{2}(\sigma^0)^{ij}
               \partial_j\left(\frac{x^1}{W[\T,Y]}\right)
 -\frac{\ell^2x^1}{W[\T,Y]}(\sigma^0)^{ij}\alpha_j,
\]
which is even in $v$, whereas $\dot\ell=-\ell\alpha_i(x)v^i$
is odd in $v$. Thus $(x(-\tau),-v(-\tau),\ell(-\tau))$
solves the same first-order system with the same data at zero.
Uniqueness gives
\begin{equation}\label{local:brake-base}
 \begin{aligned}
 x(\tau)&=x(-\tau),\\
 \ell(\tau)&=\ell(-\tau).
 \end{aligned}
\end{equation}
Solve~\eqref{local:brake-lift} with initial point $p_0$ and the
reflected $x,\ell$. For small $\tau$ the solution stays in
$B\times P$. Its base retraces the incoming boundary arc, and
\eqref{local:plaque-invariance} consequently places its lift in
$\pi^{-1}(\partial H)$. Thus the coefficient identities
\eqref{local:boundary-coefficients} hold along the lift.
We can now check the full acceleration in the frame
$(\T,Y,E_1,E_2)$:
\begin{align*}
 \g(\D_{\dot\gamma}\dot\gamma,\T)&=0,\\
 \g(\D_{\dot\gamma}\dot\gamma,Y)
   &=\dot\ell-\tfrac12(\LL_Y\g)(\dot\gamma,\dot\gamma)
     =\dot\ell+\ell\alpha(v)=0.
\end{align*}
The two remaining contractions
$\g(\D_{\dot\gamma}\dot\gamma,E_i)$ vanish by
\eqref{local:transverse-euler-lagrange} and
\eqref{local:mechanical-reflection}. Hence
$\D_{\dot\gamma}\dot\gamma=0$. Its tangent at contact is
\[
 \dot\gamma(0)=\frac{\ell(0)}{\g(\T,Y)(p_0)}\T_{p_0}=v_0.
\]
Uniqueness for the spacetime geodesic equation identifies this
lift with the ambient continuation. Formula
\eqref{local:isolated-contact} proves the remaining assertion.
\end{proof}

\begin{figure}[tbp]
\centering
\begin{tikzpicture}[x=1cm,y=1cm,line cap=round,line join=round,
                    every node/.style={font=\normalsize},>=stealth]
 \path[use as bounding box] (0,-.7) rectangle (15.7,5.1);
 \begin{scope}
  \node at (3.7,4.9) {(a) The coordinates $(\g(\T,\T),\psi)$};
  \draw[densely dashed] (1.25,.45)--(1.25,4.45);
  \node[anchor=north] at (1.25,.3) {$\g(\T,\T)=0$};
  \node at (4.55,4.15) {$\g(\T,\T)>0$};
  \draw[line width=1pt,domain=0:1,samples=70,variable=\u]
   plot ({1.25+4.8*\u},{2.25+1.25*\u*\u});
  \draw[->,line width=1pt,domain=.36:.17,samples=15,variable=\u]
   plot ({1.25+4.8*\u},{2.25+1.25*\u*\u});
  \draw[->,line width=1pt,domain=.64:.82,samples=15,variable=\u]
   plot ({1.25+4.8*\u},{2.25+1.25*\u*\u});
  \node[anchor=north] at (2.48,2.16) {incoming};
  \node[anchor=south] at (4.80,3.14) {outgoing};
  \fill (1.25,2.25) circle (1.8pt);
  \node[anchor=east] at (1.07,2.25) {$x(p_0)$};
  \node[anchor=north west] at (1.40,1.65) {$\dot x(0)=0$};
  \fill (3.89,2.628) circle (1.8pt);
  \node[anchor=north west,align=left] at (3.92,2.46)
   {$x(-\varepsilon)$\\$=x(\varepsilon)$};
  \node[align=center] at (3.85,-.47) {One transverse arc, traversed twice};
 \end{scope}
 \begin{scope}[xshift=8cm]
  \node at (3.7,4.9) {(b) The coordinates $(\g(\T,\T),\theta)$};
  \draw[densely dashed,->] (1.25,.45)--(1.25,4.45);
  \node[anchor=east] at (1.08,4.10) {$\theta$};
  \node[anchor=north] at (1.25,.3) {$\g(\T,\T)=0$};
  \draw[line width=1pt,domain=-2:1.75,samples=90,variable=\u]
   plot ({1.25+.95*\u*\u},{2.4+\u+.08*\u*\u});
  \draw[->,line width=1pt,domain=-1.05:-.68,samples=15,variable=\u]
   plot ({1.25+.95*\u*\u},{2.4+\u+.08*\u*\u});
  \draw[->,line width=1pt,domain=.80:1.12,samples=15,variable=\u]
   plot ({1.25+.95*\u*\u},{2.4+\u+.08*\u*\u});
  \draw[dotted] (2.981375,1.1958)--(2.981375,3.8958);
  \fill (2.981375,1.1958) circle (1.8pt);
  \fill (2.981375,3.8958) circle (1.8pt);
  \node[anchor=west,fill=white,inner sep=2pt] at (3.16,1.1958)
   {$\gamma(-\varepsilon)$};
  \node[anchor=west,fill=white,inner sep=2pt] at (3.16,3.8958)
   {$\gamma(\varepsilon)$};
  \node[anchor=west,align=left] at (3.16,2.40)
   {Same transverse point,\\different spacetime events};
  \fill (1.25,2.4) circle (1.8pt);
  \draw[->,line width=.8pt] (1.25,2.48)--(1.25,3.12);
  \node[anchor=east] at (1.07,2.92) {$\T$};
  \node[anchor=east,align=right] at (1.04,2.07)
   {$\gamma(0)$};
  \node[align=center] at (3.85,-.47)
   {$0\ne\dot\gamma(0)\in\RR\T$};
 \end{scope}
\end{tikzpicture}
\caption{Continuation at the ergosurface, as in
Lemma~\ref{local:brake-reflection}.
In (a), the curve $x=(\g(\T,\T),\psi)$ reaches $x(p_0)$ with zero velocity
and retraces exactly the same arc, by \eqref{local:brake-base}.
The incoming and outgoing portions have $\tau<0$ and
$\tau>0$, respectively.
In (b), the same geodesic $\gamma$ is shown schematically in the regular coordinate
projection $(\g(\T,\T),\theta)$, with the other coordinates suppressed
and orientation chosen so that $\dot\theta(0)>0$.
Its velocity at contact is a nonzero multiple of $\T=\partial_\theta$.
The arrows follow increasing affine parameter, and the repeated
transverse point corresponds to distinct spacetime events.
The dashed line represents the ergosurface, and each panel depicts
the indicated local coordinate projection.}
\label{fig:ergosurface-continuation}
\end{figure}
\subsection{Exclusion of a remaining boundary}

We apply the length bound in Lemma~\ref{action:escape} to a
null geodesic whose projection remains in the compact set $H_0$.
Continuation at every finite parameter endpoint will give a segment
longer than that bound.

\begin{proof}[Proof of Theorem~\ref{discrete:continuation}]
Suppose $H\cap\{\g(\T,\T)>0\}\ne\varnothing$. Choose a connected
component of $\mathcal D_0$ meeting this set. By
Proposition~\ref{prep:residual}, its positive part is path connected
and meets $C_+$. There is consequently a path
$\eta:[0,1]\to\mathcal D_0\cap\{\g(\T,\T)>0\}$ with
$\eta(0)\notin H_0$ and $\eta(1)\in H$. Its first intersection
with the closed set $H$ lies in
$\partial H\cap\{\g(\T,\T)>0\}$.

Near this point, Proposition~\ref{discrete:stalled-boundary}
gives a smooth timelike surface $\Sigma\subset\partial H$
with $B(X,X)=0$ for every null tangent vector $X$.
By \eqref{local:null-surface-acceleration}, its null curves are
null geodesics of the Lorentzian stationary quotient $(Q,h)$
up to reparametrization. Since $\g(\T,\T)>0$ here,
\eqref{local:null-lift} lifts either null direction to a
zero-energy spacetime null geodesic with projection in $\partial H$.

Parametrize one lift $\gamma$ by $\mathbf g_+$-arclength, starting
at $s=0$. Let $s_+>0$ be the maximal forward time for which
\eqref{action:normalized-flow} has this property. If
$s_+<\infty$, compactness of $H_0$ and
Lemma~\ref{action:escape} give
\[
 \lim_{s\uparrow s_+}(\gamma(s),\dot\gamma(s))=(p_0,v_0).
\]
The limit satisfies
\begin{align*}
 |v_0|_{\mathbf g_+}&=1,\\
 \g(v_0,v_0)&=\g(\T,v_0)=0,\\
 \pi(p_0)&\in\partial H\subset H_0.
\end{align*}
If $\g(\T,\T)(p_0)>0$, the smooth boundary theorem and
\eqref{local:null-surface-acceleration} continue the lifted null curve through
$(p_0,v_0)$. If $\g(\T,\T)(p_0)=0$,
Lemma~\ref{local:brake-reflection} gives the same conclusion.
Uniqueness for the normalized equation contradicts the
maximality of $s_+$ in either case. Thus the curve is defined
for all $s\ge0$, with projection in $H_0$.

Let $C_{H_0}$ be the length bound of
Lemma~\ref{action:escape}. Applying it to the segment
$0\le s\le C_{H_0}+1$ gives
\[
 C_{H_0}+1
 =\int_0^{C_{H_0}+1}|\dot\gamma(s)|_{\mathbf g_+}\,ds
 \le C_{H_0},
\]
a contradiction. Hence $H\cap\{\g(\T,\T)>0\}=\varnothing$.
The definition~\eqref{discrete:exact-domain} of $U_H$ and
Lemma~\ref{discrete:maximal-action} now give the asserted
extension of the rotational action.
\end{proof}

\section{Completion of the proof}\label{exit:section}

The rotational field is now defined throughout the ergoregion.
We first extend it locally across the ergosurface and show that
it remains independent of $\T$ there. A timelike constant
combination of the two Killing fields then gives analytic
coordinates, which allow us to apply the global extension
theorem of Chru\'sciel~\cite[Theorem~1.1]{Chrusciel97}.
The uniqueness theorem will be applied only after this global
action has been obtained.

\begin{proposition}\label{exit:global}
Assume GR, SBS, NT, and $\T|_{S_0}\not\equiv0$.
The rotational field of Theorem~\ref{prop:continuation} extends
to $\Z\in C^\infty(T\E)$, with
\[
 \begin{aligned}
 \LL_\Z\g&=0,\\
 [\T,\Z]&=0.
 \end{aligned}
\]
Its flow is complete, has effective period $2\pi$, and together
with the stationary flow generates an effective isometric
$\RR\times U(1)$ action with nonempty rotation axis.
\end{proposition}
\begin{proof}
Fix $p\in\E$ with $\g(\T,\T)(p)=0$. The field $\Z$ is already
Killing and commutes with $\T$ on the entire local side
$\{\g(\T,\T)>0\}$. The strict inequality~\eqref{local:ergosurface-convexity}
and Proposition~\ref{local:extension} therefore
extend this same field to a connected neighborhood $V$ of $p$.

We claim that $\T,\Z$ remain independent at $p$.
By \eqref{canon:commuting-twist}, $d(\Z\psi)=0$ on $V$.
The side $\g(\T,\T)>0$ fixes this constant at zero, so $\Z\psi=0$
throughout $V$. Also $\Z[\g(\T,\T)]=0$ by the Killing
equation and commutation. Suppose $\Z(p)=b\T(p)$ for a real number $b$.
Set $Y=\Z-b\T$ and $A(X)=\D_XY|_p$. Since $Y(p)=0$, we have
\begin{align*}
 A\T&=[\T,Y]|_p=0,\\
 0=X\bigl(Y[\g(\T,\T)]\bigr)|_p
   &=\g\bigl(AX,\operatorname{grad}_\g\g(\T,\T)\bigr)|_p,\\
 0=X(Y\psi)|_p&=\g(AX,\operatorname{grad}_\g\psi)|_p.
\end{align*}
The endomorphism $A$ is skew-adjoint. It therefore annihilates
$\T$, $\operatorname{grad}_\g\g(\T,\T)$ and
$\operatorname{grad}_\g\psi$. These vectors are independent by
Lemma~\ref{canon:regularity}, so $\operatorname{rank}A\le1$.
The two-form $\g(A\cdot,\cdot)$ has even rank. Hence $A=0$,
and the Killing transport equations
\eqref{local:killing-connection} imply $Y=0$ on $V$. This
contradicts $W[\T,\Z]>0$ on its positive side, proved in
Lemma~\ref{action:area}.

At $p$, both $\T$ and $\Z$ belong to the Lorentzian plane
$\ker d[\g(\T,\T)]\cap\ker d\psi$. Since $\T$ is nonzero
and null, their independence gives $\g(\T,\Z)(p)\ne0$.
Choose a sufficiently small constant $\varepsilon$ with sign
opposite to this number. The Killing field $\T+\varepsilon\Z$
satisfies
\[
 \g(\T+\varepsilon\Z,\T+\varepsilon\Z)(p)
 =2\varepsilon\g(\T,\Z)(p)
       +\varepsilon^2\g(\Z,\Z)(p)<0.
\]
Thus a timelike Killing field exists near each ergosurface
point. At a point $p$ in the open ergoregion,
Lemma~\ref{action:area} instead gives the timelike Killing field
with constant coefficients
\[
 \T-\frac{\g_p(\T,\Z)}{\g_p(\Z,\Z)}\Z.
\]
Indeed,
\[
 \g_p\left(\T-\frac{\g_p(\T,\Z)}{\g_p(\Z,\Z)}\Z,
           \T-\frac{\g_p(\T,\Z)}{\g_p(\Z,\Z)}\Z\right)
 =-\frac{W[\T,\Z](p)}{\g_p(\Z,\Z)}<0.
\]
Where $\g(\T,\T)<0$, we use $\T$ itself.

The stationary vacuum analyticity theorem of
M\"uller zum Hagen~\cite{MH70}, in the form discussed by
Tod~\cite{Tod07}, supplies analytic metric
coordinates in each of these neighborhoods. Their transition
maps are analytic: a smooth local isometry $\varphi$ has the
normal-coordinate expression
\[
 \varphi=\exp_{\varphi(p)}\circ d\varphi_p\circ\exp_p^{-1}.
\]
Consequently $\E$ is an analytic vacuum spacetime.

We apply Chru\'sciel~\cite[Theorem~1.1]{Chrusciel97}
to $\E$ and the original rotational field near the horizon.
GR and Proposition~\ref{prep:quotient} give the complete
stationary field and the globally hyperbolic simply connected
domain of outer communications. Computed inside $\E$, the
domain of outer communications is again $\E$, since a causal
curve between points of the asymptotic end lies in the original
domain of outer communications. The asymptotic estimates
\eqref{exit:asymptotic-data} imply the decay required in
that theorem: one may take regularity order $3$ and decay
exponent $3/4$. They also give zero spatial
ADM momentum:
\[
 \frac1{8\pi}\lim_{R\to\infty}
 \int_{|x|=R}(k_{ij}-\operatorname{tr}_\gamma k\,\gamma_{ij})
                    \nu^j\,dS
 =\lim_{R\to\infty}O(R^{-1})=0.
\]
Here $\gamma,k$ are the induced asymptotic data,
$\nu$ is the Euclidean unit normal, and $dS$ is Euclidean area.
The ADM four-momentum is $(M,0,0,0)$, with $M>0$ by GR.
Since $\T|_{S_0}\not\equiv0$ and $S_0$ is spacelike,
$\g(\T,\T)>0$ at some point of $S_0$ and hence at nearby
points of $\pi^{-1}(C_+)$. At any such exterior point,
Lemma~\ref{action:area} gives $W[\T,\Z]>0$, so the original
field is not a constant multiple of $\T$.

The proof of Chru\'sciel's theorem gives a complete periodic
Killing field $Y$, commuting with $\T$, whose restriction to
$\pi^{-1}(C_+)$ has the form
\[
 Y=a\T+b\Z,
 \qquad a,b\in\RR,
 \qquad b\ne0.
\]
Let $P>0$ be a period of $Y$ and fix $q\in\pi^{-1}(C_+)$.
The complete commuting flows on $\pi^{-1}(C_+)$ give
\[
 \Phi_{nPa}^\T(q)=\Phi_{-nPb}^\Z(q)
                    \qquad(n\in\mathbb Z).
\]
Let $C=\{\Phi_\theta^\Z(q):0\le\theta\le2\pi\}$.
The equivariant function $\vartheta$ of
Proposition~\ref{prep:quotient} then gives
\[
 |nPa|
 =|\vartheta(\Phi_{-nPb}^\Z q)-\vartheta(q)|
 \le\max_C\vartheta-\min_C\vartheta
                    \qquad(n\in\mathbb Z).
\]
Letting $|n|\to\infty$ yields $a=0$. Thus $Y/b$ is a complete
global extension of the original $\Z$. Killing transport
identifies it with the field on the connected domain $U_H$.
Its time-$2\pi$ map is an isometry equal to the identity near
the horizon and hence everywhere, and effectiveness follows
by restriction to the original action.

If $\Phi_s^\T=\Phi_\theta^\Z$, iteration and the same function give
$|ns|\le\max_C\vartheta-\min_C\vartheta$ for all integers $n$.
Hence $s=0$, and effectiveness of the rotations
gives $\theta\in2\pi\mathbb Z$. The product action is therefore
effective. Finally, choose a zero $q$ of $\Z|_{S_0}$ and the
outward normal $e_1$ fixed in Lemma~\ref{prep:collar}. Its
rotation invariance gives
\[
 \Phi_\theta^\Z(\exp_q(re_1(q)))
 =\exp_{\Phi_\theta^\Z q}(r\,d\Phi_\theta^\Z e_1(q))
 =\exp_q(re_1(q))
                  \qquad(0<r<\delta).
\]
These exterior curves belong to the rotation axis.
\end{proof}

We use the smooth axisymmetric uniqueness theorem of
Chru\'sciel, Costa, and Heusler~\cite[Theorem~3.2]{CCH12},
with the nondegenerate conclusion of Chru\'sciel and
Costa~\cite[Corollary~6.3]{CC08}. We state its vacuum case in
the form needed here. The notion of $I^+$-regularity and its
verification from GR and SBS were given with
Lemma~\ref{exit:regular-completion}.

\begin{theorem}[Axisymmetric uniqueness]\label{exit:uniqueness}
Let a smooth four-dimensional stationary vacuum spacetime be
asymptotically flat and $I^+$-regular, with a connected
nondegenerate future horizon. Suppose its stationary and
rotational Killing fields generate an effective isometric
$\RR\times U(1)$ action with nonempty rotation axis.
Then its domain of outer communications is isometric to a
nonextremal Kerr exterior.
\end{theorem}

\begin{proof}[Proof of Theorem~\ref{thm:main}]
Suppose first that $\T|_{S_0}\not\equiv0$.
Theorem~\ref{prop:continuation} and Proposition~\ref{exit:global}
give the global rotational action. By Killing transport it
agrees with the original action on the horizon neighborhood
of Lemma~\ref{prep:collar}. The stationary extension of that
neighborhood in Lemma~\ref{exit:regular-completion} therefore
extends the same action to its $I^+$-regular completion.
Theorem~\ref{exit:uniqueness} applies and gives a nonextremal
Kerr exterior.

If $\T|_{S_0}=0$, Lemma~\ref{exit:regular-completion} supplies
an $I^+$-regular completion with connected nondegenerate
horizon generated by $\T$. We use the smooth staticity
argument in Chru\'sciel, Costa, and
Heusler~\cite{CCH12}. Its global preparation,
carried out by Chru\'sciel and Costa~\cite{CC08},
places the same exterior in a spacetime with a bifurcate
Killing horizon and a Cauchy section ending at the
bifurcation surface. The maximal-hypersurface theorem of
Chru\'sciel and Wald~\cite[Theorem~4.2]{CWMax94} then gives a
maximal asymptotically flat hypersurface with that boundary.
The exterior-preserving construction applies to the smooth
Killing horizon with nonzero surface gravity and future section
supplied by Lemma~\ref{exit:regular-completion}. It supplies the
strongly causal spacetime and closed acausal section required
by the maximal-hypersurface theorem: the section has compact
boundary where $\T=0$, one asymptotically flat end, and interior
domain of dependence $\E$. This domain contains no black or
white hole relative to that end. Vacuum supplies the timelike
convergence condition, so the stationary translates of the
maximal hypersurface foliate $\E$.

We use the vacuum divergence identity in the staticity argument
of Sudarsky and Wald~\cite[Theorem~1]{SW93}. Denote the resulting
maximal hypersurface by $\Sigma$ and its bifurcation boundary by $S$.
On $\Sigma$ write
$\T=N\mathbf n+X$, where $\mathbf n$ is the future unit normal,
$X$ is tangent, and $k_{ij}=-\g(\D_i\mathbf n,\partial_j)$.
Let $\gamma$ be the induced metric and $\nabla^\gamma$ its
Levi--Civita connection. In the integrals
below, $d\mu$ and $dS$ are the induced volume and area measures.
The lapse satisfies
\[
 \begin{aligned}
 \Delta_\gamma N&=|k|^2N,\\
 N|_S&=0,\\
 N&\longrightarrow1\quad(r\to\infty).
 \end{aligned}
\]
For sufficiently large $R$, the negative part
$N_-:=\max\{-N,0\}$ vanishes on $\{r=R\}$ and has zero trace
on $S$. Testing the lapse equation with $N_-$ on the portion
$\Sigma_R$ bounded by these two surfaces gives
\[
 \int_{\Sigma_R}|\nabla^\gamma N_-|^2\,d\mu
 =\int_{\Sigma_R}N_-\Delta_\gamma N\,d\mu
 =-\int_{\Sigma_R}|k|^2N_-^2\,d\mu.
\]
Thus $N\ge0$. The strong maximum principle and $N\to1$ give
$N>0$ in the interior. The stationary metric equation and the
vacuum momentum constraint give
\begin{align*}
 \nabla^\gamma_iX_j+\nabla^\gamma_jX_i&=2Nk_{ij},\\
 \nabla^\gamma_i k^{ij}&=0.
\end{align*}
Using symmetry of $k$, we obtain
\[
\begin{aligned}
 \nabla^\gamma_i(k^{ij}X_j)
 &=(\nabla^\gamma_i k^{ij})X_j
   +\tfrac12k^{ij}(\nabla^\gamma_iX_j+\nabla^\gamma_jX_i)\\
 &=N k^{ij}k_{ij}.
\end{aligned}
\]
If $\nu_R$ and $\nu_S$ are the outward unit normals of
$\Sigma_R$ along its two boundary components, integration gives
\[
 \int_{\Sigma_R}N|k|^2\,d\mu
 =\int_{r=R}k(X,\nu_R)\,dS+\int_S k(X,\nu_S)\,dS.
\]
On $S$, $\T=0$ implies $N=X=0$, so the second integral
vanishes. The asymptotics of the maximal hypersurface used in
Chru\'sciel and Costa~\cite{CC08} give
\[
 \begin{aligned}
 |k|&=O(r^{-2}),\\
 |X|&=O(r^{-1}),
 \end{aligned}
\]
and consequently
\[
 \left|\int_{r=R} k(X,\nu_R)\,dS\right|
 \le CR^2\sup_{r=R}|k|\sup_{r=R}|X|
 \le C R^{-1}\longrightarrow0.
\]
Since $N|k|^2\ge0$, exhaustion by $\Sigma_R$ yields
\[
 \int N|k|^2\,d\mu=0.
\]
It follows that $k=0$, and the stationary metric equation makes $X$ Killing
for the induced metric. On the bifurcation boundary, $X=0$ and
$\nabla^\gamma_VX=0$ for all tangent vectors $V$. The skew endomorphism $\nabla^\gamma X$
therefore has a two-dimensional kernel in a three-dimensional
space, so $\nabla^\gamma X=0$. Killing transport gives $X=0$ on the connected
maximal hypersurface. Its stationary translates cover $\E$ by
the maximal-hypersurface theorem, and hence
\[
 \T^\flat\wedge d\T^\flat=0\qquad\hbox{on }\E.
\]
Static uniqueness, due to Bunting and
Masood-ul-Alam~\cite{BuntingMasood87}, in the global form of
Chru\'sciel, Costa, and Heusler~\cite[Theorem~3.1]{CCH12},
identifies this exterior with Schwarzschild.
\end{proof}

\appendix
\section{Extension across the future horizon}\label{app:completion}

We prove the extension asserted in Lemma~\ref{exit:regular-completion}.
The local Hawking Killing field supplies the horizon generator, while
the stationary flow extends a small neighborhood of a future
section of the horizon. To preserve the given exterior, we
verify both that the gluing is Hausdorff and that future causal
curves cannot cross from the added region back into $\E$.

\begin{proof}[Proof of Lemma~\ref{exit:regular-completion}]
Lemma~\ref{prep:collar} and \eqref{prep:nonrotating-normalization}
give $\mathbf K=\T+\Omega_H\Z$ in the rotating case and
$\mathbf K=\T$ in the nonrotating case, with
\[
 \mathbf K=\kappa(V\partial_V-U\partial_U),
 \qquad \kappa>0.
\]
The exterior is $U,V>0$, and its future horizon is
$U=0,V>0$ in this local normal chart.

Use the temporal function $\tau$ constructed with these normal forms, choosing an invariant normal frame
in the rotating case and any oriented orthonormal normal frame
when $\mathbf K=\T$. Recall that $\tau$ is temporal, meaning
that its gradient is everywhere past timelike, and that
$\tau=(V-U)/2$ near $S_0$.
Outside a compact neighborhood its zero level is the
hypersurface $\Sigma^0$ chosen in the proof of Lemma~\ref{prep:collar}.
We deform the zero level by choosing a smooth cutoff $\eta=1$
near $S_0$, supported where $\tau=(V-U)/2$, and replacing
$\tau$ by $\tau-\varepsilon\eta$.
For small $\varepsilon>0$ this function is still temporal.
Indeed, choose $c>0$ with $\g^{-1}(d\tau,d\tau)\le-c$ on
the compact support of $d\eta$. Then
\[
 \g^{-1}(d\tau-\varepsilon d\eta,d\tau-\varepsilon d\eta)
 \le-c+2\varepsilon|\g^{-1}(d\tau,d\eta)|
                   +\varepsilon^2|\g^{-1}(d\eta,d\eta)|
 \le-c/2.
\]
The gradient is unchanged off this support. On a fixed
neighborhood of the compact portion being deformed, choose a
smooth future timelike field $Y$ with $Y\tau=1$. Its flow
identifies the nearby levels, and
\[
 Y(\tau-\varepsilon\eta)=1-\varepsilon Y\eta>0
\]
for small $\varepsilon$. The implicit function theorem therefore
identifies the new section with a compact deformation of the
old one. We take its connected exterior portion to be
$\mathcal S$. Its temporal construction makes it spacelike and acausal
in the causally convex development of $\Sigma^0$, and hence acausal
in $\E$. Its asymptotic end is unchanged, and its closure is diffeomorphic to
$[1,\infty)\times S^2$, with boundary
\[
 C=\{(0,2\varepsilon,p):p\in S_0\}.
\]
Figure~\ref{fig:future-horizon-section}(a) shows the exterior
portion of this level, with boundary $C$, and its return to
the unchanged asymptotic slice.

The deformed section will parametrize a neighborhood of the
horizon by the stationary flow. In the rotating case the section
is invariant under the local rotations near $C$. If $\T$ were
tangent there, then $\mathbf K=\T+\Omega_H\Z$ would also be
tangent. In the nonrotating case $\mathbf K=\T$.
Both cases contradict spacelikeness because $\mathbf K$ is nonzero
and null on $C$. Thus $\T$ is transverse there, and hence nearby.
At $C$, the null covector $d(UV)$ has
nonzero restriction to a spacelike three-plane, so we may use
$UV$ as a coordinate on the section near $C$.

In the rotating case the images of its small exterior portion under
the stationary flow are disjoint by Lemma~\ref{prep:collar}, using timelikeness
of $\mathbf K$ and periodicity of $\Z$. In the nonrotating case this follows
directly from acausality, since $\T=\mathbf K$ is timelike there.
Write $\iota:(-\delta,\delta)\times S_0\to\mathcal O$
for this part of the deformed section, using $UV$ as its first
coordinate and $\iota(0,p)=(0,2\varepsilon,p)$.
In the rotating case we choose $\iota$ equivariant under the
circle action on $S_0$: the section is $V-U=2\varepsilon$
near $C$, and both $U,V$ are rotation invariant.
On the product below, $UV$ denotes
this first section coordinate, independent of the stationary
parameter $t$; the separate functions $U,V$ belong to the original
normal chart.
On $UV>0$, the stationary flow gives an embedding into $\E$:
\[
 F(t,q)=\Phi_t^\T(\iota(q)),\qquad
 q\in(0,\delta)\times S_0.
\]
On $\mathbb R\times(-\delta,\delta)\times S_0$, for
$X,Y\in T_q(( -\delta,\delta)\times S_0)$ and $a,b\in\mathbb R$, define
\[
 \widehat{\g}_{(t,q)}((a,X),(b,Y))
 =\g_{\iota(q)}(a\T+d\iota(X),b\T+d\iota(Y)).
\]
Transversality makes the local maps
$F_{t_0}(t_0+s,q)=\Phi_s^\T(\iota(q))$ diffeomorphisms for
small $s$, on both sides of $UV=0$. Stationarity gives
\[
\begin{aligned}
F^{\ast}\g&=\widehat{\g}\quad(UV>0),\\
F_{t_0}^{\ast}\g&=\widehat{\g},\\
\LL_{\partial_t}\widehat{\g}&=0,\\
\Ric(\widehat{\g})&=F_{t_0}^{\ast}\Ric(\g)=0.
\end{aligned}
\]
Thus $\widehat{\g}$ is a smooth Lorentzian vacuum metric,
agreeing with the exterior metric on the identified region.
The field $\partial_t$ corresponds to $\T$ under $F$; we use
$\T$ for the resulting stationary field after gluing.
In the rotating case the maps
$(t,UV,p)\mapsto(t,UV,\Phi_\theta^\Z(p))$ preserve this metric
and agree with the original rotations on the overlap.
They define the lifted field $\Z$ on the product, with
$\Z t=\Z(UV)=0$, and extend the Hawking field by
$\mathbf K=\T+\Omega_H\Z$. In the nonrotating case we extend
$\mathbf K$ by $\T$.

For the gluing to be Hausdorff, we need to separate the added
horizon from compact subsets of $\E$. Suppose, to the contrary,
that $p_j$ in the deformed section tends to $C$ while
$q_j=\Phi_{t_j}^{\T}p_j$ remains in a compact subset of $\E$.
Apply \eqref{prep:acausal-time-bound} with the acausal section
$\mathcal S$ and a compact set containing all $q_j$.
Since $p_j=\Phi_{-t_j}^{\T}q_j\in\mathcal S$, the sequence
$t_j$ is bounded. Taking subsequences with
$t_j\to t$ and $q_j\to q\in\E$
would give $p_j\to\Phi_{-t}^{\T}q\in\E$, contrary to $p_j\to C$.

We now identify the domain of outer communications by proving
that $UV\le0$ is preserved along future causal curves in the
added neighborhood. We denote the glued metric by $\g$. On the
added product, the translates of the spacelike section have conormal $dt$,
so $\g^{-1}(dt,dt)<0$ on the added neighborhood. In the rotating
case $\Z t=0$, and $\mathbf K=\T+\Omega_H\Z$ gives
$\mathbf Kt=1$. This identity also holds when $\mathbf K=\T$.
Since $\mathbf K$ is future null on $C$, the metric gradient
of $t$ is past timelike there, and hence throughout a smaller
stationary product neighborhood.
In the original normal coordinates $(U,V,x^1,x^2)$, where
$x^A$ are local coordinates on $S_0$, the generator $\partial_V$
on $U=0$ is orthogonal to $\partial_V,\partial_{x^A}$ and
$\g(\partial_U,\partial_V)>0$. Inverting that row of the
metric gives
$\g^{UU}=\g^{UA}=0$ and $\g^{UV}=\frac1{\g(\partial_U,\partial_V)}$.
It follows that
\[
 \g^{\alpha\beta}\partial_\beta(UV)\partial_\alpha
       =\frac{V}{\g(\partial_U,\partial_V)}\partial_V
       =\frac{\mathbf K}{\kappa\g(\partial_U,\partial_V)}
       \quad\hbox{on }C.
\]
Since $\mathbf Kt=1$,
\[
 \g^{-1}(d(UV),dt)
       =\frac1{\kappa\g(\partial_U,\partial_V)}>0
                        \quad\hbox{on }C.
\]
Choose a closed interval $|UV|\le\delta_0<\delta$ in the
section. At $t=0$, the function
$\g^{-1}(d(UV),d(UV))$ vanishes on $C$, and its derivative
in the $UV$ coordinate is bounded on this compact interval
times $S_0$. The fundamental theorem of calculus gives
\[
 |\g^{-1}(d(UV),d(UV))|
 \le |UV|\sup_{|u|\le\delta_0,\,p\in S_0}
       \left|\partial_u\bigl[\g^{-1}(d(UV),d(UV))\bigr](0,u,p)\right|.
\]
The positive function $\g^{-1}(d(UV),dt)|_C$ has a positive
minimum. Shrinking $\delta_0$ and using stationarity therefore
gives constants $c_0,C_0>0$, independent of $t\in\RR$, such that
\[
\begin{aligned}
|\g^{-1}(d(UV),d(UV))|&\le C_0|UV|,\\
\g^{-1}(d(UV),dt)&\ge c_0,\\
\g^{-1}(dt,dt)&<0.
\end{aligned}
\]
Choose $C_1>C_0/(2c_0)$. Then
\[
 \begin{aligned}
 &\g^{-1}(d(UV)-C_1|UV|dt,d(UV)-C_1|UV|dt)\\
 &\quad=\g^{-1}(d(UV),d(UV))
       -2C_1|UV|\g^{-1}(d(UV),dt)
       +C_1^2(UV)^2\g^{-1}(dt,dt)\\
 &\quad\le(C_0-2C_1c_0)|UV|\le0,\\
 &\g^{-1}(d(UV)-C_1|UV|dt,dt)\ge c_0.
 \end{aligned}
\]
The metric dual of $d(UV)-C_1|UV|dt$ is future causal. Recall that $(UV)_+=\max\{UV,0\}$.
Every future causal curve in the added neighborhood, parametrized
by $t$, consequently satisfies
\begin{align}
 \frac{d(UV)}{dt}&\le C_1|UV|,\notag\\
 \shortintertext{and hence}
 \frac d{dt}(UV)_+&\le C_1(UV)_+
       \quad\hbox{almost everywhere}.\label{app:causal-barrier}
\end{align}
Integrating the inequality for $(UV)_+$ gives, for $t\ge t_0$,
\[
 (UV)_+(t)\le e^{C_1(t-t_0)}(UV)_+(t_0).
\]
Thus $UV\le0$ is preserved in the future direction. Since a
future causal path from the added region to the asymptotic end would have to
cross $UV=0$ toward $UV>0$, the domain of outer communications
remains $\E$. The same estimate keeps every causal path between
exterior points in the exterior, so the deformed section remains
acausal after gluing. A closed causal curve would have to stay
either in $\E$, where global hyperbolicity excludes it, or in
the added region, where $t$ is temporal. Thus the extended spacetime
is causal as well.
Figure~\ref{fig:future-horizon-section}(b) depicts this causal
property in the stationary coordinates $(t,UV)$.

\begin{figure}[htbp]
\centering
\begin{tikzpicture}[x=.92cm,y=.92cm,font=\normalsize,
                    line cap=round,line join=round]
 \begin{scope}
  \draw[densely dotted,black!60,line width=.55pt]
              (-.53,-.67) rectangle (3.0,1.48);
  \draw[line width=.85pt] (-.43,-.43)--(1.98,1.98);
  \draw[line width=.75pt] (-.67,.67)--(.82,-.82);
  \node[above left] at (1.78,1.78) {$\HH^+$};
  \node[right=3pt,fill=white,inner sep=1pt] at (.70,-.82) {$\HH^-$};
  \draw[dashed,line width=.75pt] (0,0)--(4.65,0);
  \draw[densely dashed,black!60,line width=.6pt] (-.34,.70)--(.70,.70);
  \draw[line width=1.2pt] (.70,.70)--(2.52,.70)
         .. controls (3.2,.70) and (3.45,0) .. (4.65,0)--(6.5,0);
  \filldraw[fill=white,line width=.65pt] (0,0) circle[radius=2pt];
  \fill (.70,.70) circle[radius=2pt];
  \node[anchor=east,fill=white,inner sep=1pt] at (-.19,-.17) {$S_0$};
  \node[above left=2pt] at (.70,.70) {$C$};
  \node[above=3pt,fill=white,inner sep=1pt] at (1.9,.70) {$V-U=2\varepsilon$};
  \node[below=3pt,fill=white,inner sep=1pt] at (2.0,0) {$V=U$};
  \node[align=center] at (5.25,.87) {unchanged\\asymptotic end};
  \node at (4.0,1.65) {$\E$};
  \node at (2.85,-1.47) {(a) Deformation of the section};
 \end{scope}
 \begin{scope}[xshift=10.2cm]
  \fill[black!13] (-1.30,-.83) rectangle (0,1.8);
  \fill[black!3] (0,-.83) rectangle (2.15,1.8);
  \draw[->,line width=.9pt] (0,-.83)--(0,1.92);
  \node[above=3pt] at (0,1.92) {$\HH^+:\ UV=0$};
  \draw[densely dashed,black!60,line width=.6pt] (-1.30,0)--(0,0);
  \draw[line width=1.2pt] (0,0)--(2.15,0);
  \fill (0,0) circle[radius=2pt];
  \node[above left=2pt] at (0,0) {$C$};
  \node[above=4pt] at (1.52,0) {$t=0$};
  \draw[->,line width=.9pt] (0,.35)--(0,.86);
  \node[anchor=west] at (.13,.70) {$\mathbf K$};
  \node[align=center] at (-.71,-.55) {$UV<0$};
  \node at (1.14,1.48) {$UV>0$};
  \draw[->,line width=.8pt] (1.6,-.67)
          .. controls (.90,-.45) and (.10,.39) .. (-.97,1.31);
  \draw[->,line width=.65pt] (-1.76,.44)--(-1.76,1.28);
  \node[above] at (-1.76,1.28) {$t$};
  \node at (.22,-1.47) {(b) Future causal crossing};
 \end{scope}
\end{tikzpicture}
\caption{The extension across the future horizon in Lemma~\ref{exit:regular-completion}.
In (a), the thick exterior section begins at
$C=\{U=0,V=2\varepsilon\}$ and rejoins the original section
outside the local chart (dotted box). Its dashed continuation
across the horizon is auxiliary. In the stationary coordinates
of (b), future causal curves may enter the added region $UV<0$
but cannot return to the exterior $UV>0$, by \eqref{app:causal-barrier}.}
\label{fig:future-horizon-section}
\end{figure}

\Needspace{4\baselineskip}
It remains to verify the horizon conditions in $I^+$-regularity.
The stationary field is complete by construction, and the
deformed section consists of a compact part and its original end,
with boundary $C$. On $U=0$, the original normal exponential
coordinates make $V$ affine. Hence
$\mathbf K=\kappa V\partial_V$ and $\D_{\partial_V}\partial_V=0$,
which imply
\[
 \D_{\mathbf K}\mathbf K=\kappa^2V\partial_V=\kappa\mathbf K.
\]
Stationary translation preserves this identity on the added
horizon. Together with $\mathbf Kt=1$, it gives
\[
 \D_{e^{-\kappa t}\mathbf K}(e^{-\kappa t}\mathbf K)
 =e^{-2\kappa t}
   (\D_{\mathbf K}\mathbf K-\kappa\mathbf K)=0.
\]
In the product coordinates, we identify $\{0\}\times\{0\}\times S_0$
with $C$ by $\iota$. The generator through $(0,0,p)$ is
\[
 \Phi_t^{\mathbf K}(0,0,p)
   =\begin{cases}
     (t,0,\Phi_{\Omega_Ht}^{\Z}p),&\T|_{S_0}\not\equiv0,\\
     (t,0,p),&\T|_{S_0}=0,
    \end{cases}
    \qquad p\in S_0.
\]
The rotations in the first case are complete, so the
generators are defined for all $t\in\RR$ and each meets $C$
exactly once.
Their future affine parameter is unbounded since
\[
 \int_0^t e^{\kappa s}\,ds=\frac{e^{\kappa t}-1}{\kappa}
       \longrightarrow\infty\qquad(t\to\infty).
\]
In the original normal chart, the short curves with tangent
$\partial_V-\partial_U$ are future timelike. For $p\in S_0$
and sufficiently small $s>0$, we therefore have
\[
 (s,2\varepsilon-s,p)\ll(0,2\varepsilon,p).
\]
The first point lies in $\E$, so it is chronologically preceded
by a point of the asymptotic end. Hence
\[
 C\subset\partial\E\cap I^+(\M^{(\mathrm{end})}).
\]
The asymptotic data are unchanged and satisfy
\eqref{exit:asymptotic-data}.
Together with global hyperbolicity of $\E$ from GR, we have
verified $I^+$-regularity as defined by Chru\'sciel and
Costa~\cite[Definition~1.1]{CC08}. The horizon is connected with
surface gravity $\kappa>0$. In the rotating case, stationary
translation extends the local rotational action to its added neighborhood.
\end{proof}
\section*{Acknowledgements}

The author was partially supported by a grant of the Ministry of Research,
Innovation and Digitization, CCCDI - UEFISCDI, project number
ROSUA-2024-0001, within PNCDI IV.
The author is grateful to Professors Alexandru D.~Ionescu and
Sergiu Klainerman for their encouragement and for suggesting this problem.

\subsection*{Competing interests}
The author has no conflicts of interest to declare.
\subsection*{Generative AI}

ChatGPT 5.6 Sol Ultra was used for literature checks and editorial work including transcribing handwritten notes into LaTeX. The author wrote and verified all mathematical content and takes full responsibility for the manuscript.

\end{document}